\documentclass[11pt]{amsart}
\usepackage{amsthm,amssymb,mathrsfs,mathtools,empheq}
\usepackage[shortlabels]{enumitem}
\usepackage[T1]{fontenc}
\usepackage{setspace}

\usepackage[notref,notcite,final]{showkeys}
\mathtoolsset{showonlyrefs}
\usepackage{fullpage,color}

\newtheorem{mainthm}{Theorem}
\newtheorem{theorem}{Theorem}[section]
\newtheorem*{theorem*}{Theorem}

\newtheorem{lemma}[theorem]{Lemma}
\newtheorem{proposition}[theorem]{Proposition}
\newtheorem*{proposition*}{Proposition}

\newtheorem*{conjecture*}{Conjecture}

\theoremstyle{definition}

\newtheorem{remark}[theorem]{Remark}

\numberwithin{equation}{section}

\def\bC {\mathbb{C}}
\def\bN {\mathbb{N}}

\def\bR {\mathbb{R}}

\def\cE {\mathcal{E}}

\def\cS {\mathcal{S}}

\def\cY {\mathcal{Y}}

\def\scrL{\mathscr{L}}

\def\grad {{\nabla}}

\def\la {\langle}
\def\ra {\rangle}

\newcommand{\tx}[1]{\mathrm{#1}}
\newcommand{\wto}{\rightharpoonup}

\newcommand{\wt}[1]{\widetilde{#1}}

\newcommand{\conj}[1]{\overline{#1}}

\newcommand{\sign}{\operatorname{sign}}
\newcommand{\spn}{\operatorname{span}}

\newcommand{\dist}{\operatorname{dist}}
\renewcommand{\ker}{\operatorname{ker}}

\newcommand{\Id}{\operatorname{Id}}

\newcommand{\eee}{\mathrm e}

\newcommand{\ud}{\mathrm{\,d}}
\newcommand{\vd}{\mathrm{d}}

\newcommand{\vD}{\mathrm{D}}
\newcommand{\dd}[1]{{\frac{\vd}{\vd{#1}}}}

\title{Construction of two-bubble solutions \\ for the energy-critical NLS in dimension 6}
\author{Jacek Jendrej}
\address{\hskip-1.15em Jacek Jendrej
	\hfill\newline Institut de Math\'ematiques de Jussieu,
	\hfill\newline Sorbonne Universit\'e, Universit\'e Paris Cit\'e,
	\hfill\newline 4 place Jussieu, 75005 Paris, France.
}
\email{jendrej@imj-prg.fr}

\author[]{Xuemei Li}
\address{\hskip-1.15em Xuemei Li
	\hfill\newline Laboratory of Mathematics and Complex Systems,
	\hfill\newline Ministry of Education,
	\hfill\newline School of Mathematical Sciences,
	\hfill\newline Beijing Normal University,
	\hfill\newline Beijing, 100875, People's Republic of China.}
\email{xuemei\_li@mail.bnu.edu.cn}

\author[]{Guixiang Xu}
\address{\hskip-1.15em Guixiang Xu
	\hfill\newline Laboratory of Mathematics and Complex Systems,
	\hfill\newline Ministry of Education,
	\hfill\newline School of Mathematical Sciences,
	\hfill\newline Beijing Normal University,
	\hfill\newline Beijing, 100875, People's Republic of China.}
\email{guixiang@bnu.edu.cn}

\begin{document}
	\onehalfspacing
	
	\begin{abstract}

         We construct pure two-bubble solutions for the energy-critical focusing nonlinear Schr\"odinger equation in space dimension $N = 6$. They are global in (at least) one time direction and approach a superposition of two stationary states,
		 both centered at the origin. One of the bubbles develops at scale $1$, whereas the length scale of the other converges to $0$ at rate $e^{-|t|}$. The phases of the two bubbles form the right angle.

         Such solutions were previously constructed in dimension $N \geq 7$. The six-dimension case presents specific difficulties, as the ground state does not belong to $\dot H^{-1}$. This prevents the use of the standard method of removing linear terms in modulation equations via suitable orthogonality conditions,
         due to loss of coercivity of the energy functional. The main novelty of this work is the introduction of modified modulation parameters to overcome this issue; these can be viewed as an analog of a normal form transformation in the context of modulation analysis.
         We also establish new coercivity estimates for the linearized energy, whose positive constants depend explicitly on the choice of the orthogonality conditions.

	\end{abstract}
	
	\maketitle
	\section{Introduction}
	\label{sec:intro}
	\subsection{Setting of the problem}
	\label{ssec:setting}
	We consider the Schr\"odinger equation in space dimension $N =6$ with the focusing energy-critical power nonlinearity:
	\begin{equation}
		\label{eq:nls}
		i\partial_t u(t, x) + \Delta u(t, x) + f(u(t, x)) = 0, \qquad f(u) := |u|u, \qquad t \in \bR,\; x \in \bR^6.
	\end{equation}
	This equation can be studied in space dimension $N \geq 3$, but here we will restrict our attention to the case $N =6 $. In this paper, we always assume that the initial data are radially symmetric. This symmetry is preserved by the flow. We denote by $\cE$ the space of radially symmetric functions in $\dot H^1(\bR^6; \bC)$.
	
	The \emph{energy functional} associated with this equation is defined for $u_0 \in \dot H^1(\bR^6; \bC)$ by the formula
	\begin{equation}\label{energy}
		E(u_0) := \int_{\bR^N}\Big( \frac 12|\grad u_0(x)|^2 - F(u_0(x))\Big)\ud x,
	\end{equation}
	where $F(u_0) := \frac{1}{3} |u_0|^3$. Note that $E(u_0)$ is well-defined due to the Sobolev inequality.
	The differential of the functional $E(u)$ is $\vD E(u) = -\Delta u - f(u)$, hence we have the following Hamiltonian form of the equation \eqref{eq:nls}:
	\begin{equation}
		\label{eq:nlsH}
		\partial_t u(t) = -i \vD E(u(t)).
	\end{equation}
	A crucial property of the solutions of \eqref{eq:nls} is that the energy $E$ is a conservation quantity.
	If $u_0 \in L^2$, then the mass $\|u(t)\|_{L^2}^2$ is another conservation quantity, but we will not use this conservation law.

	For a function $v \in \cE$, we denote
	\begin{equation*}
		v_\lambda(x) := \frac{1}{\lambda^2} v\big(\frac{x}{\lambda}\big).
	\end{equation*}
	A change of variables shows that
	\begin{equation*}
		E\big((u_0)_\lambda\big) = E(u_0).
	\end{equation*}
	Equation~\eqref{eq:nls} is invariant under the same scaling: if $u(t)$ is a solution of \eqref{eq:nls} and $\lambda > 0$, then
	$
	t \mapsto u\big(t_0 + \lambda^{-2}t\big)_\lambda
	$ is also a solution
	with initial data $(u_0)_\lambda$ at time $t = 0$.
	This is the reason why equation~\eqref{eq:nls} is called \emph{energy-critical}.

    Equation \eqref{eq:nls} is locally well-posed in the space $\dot H^1(\bR^N)$, as was proved by Cazenave and Weissler \cite{CaWe90}, see also a complete review of the Cauchy theory in \cite{KeMe06} (for $N \in \{3, 4, 5\}$), \cite{KiVi10} (for $N \geq 6$), \cite{Cazenave03} and \cite{Tao06}.
    The solutions of the corresponding \emph{defocusing} energy-critical NLS equation exist globally and scatter. In fact, by use of the induction-on-energy strategy and the space-localized Morawetz estimate, the global well-posedness and scattering result were proved by Bourgain \cite{Bourgain99} and Tao \cite{Tao05} for the radial case, later Colliander, Keel, Staffilani, Takaoka, and Tao \cite{Iteam08}, Ryckman and Visan \cite{RyVi07}, and Visan \cite{Visan07} removed the radial assumption to obtain the global well-posedness and scattering result by the induction-on-energy strategy and the localized interaction Morawetz estimate in both frequency space and physical space simultaneously.
	
	The study of the dynamical behavior of solutions of the focusing equation \eqref{eq:nls} for large initial data was initiated by Kenig and Merle \cite{KeMe06}.
	In this case, an important role is played by the family of stationary solutions $u(t) \equiv \eee^{i\theta} W_\lambda$, where
	\begin{equation}
		W(x) = \Big(1 + \frac{|x|^2}{24}\Big)^{-2}.
	\end{equation}
	The functions $\eee^{i\theta}W_\lambda$ are called \emph{ground states} or \emph{bubbles} (of energy). Up to the scaling symmetry, $W$ is the only positive, radial solutions
	of the critical elliptic problem
	\begin{equation}
		\label{eq:elliptic}
		-\Delta u - f(u) = 0.
	\end{equation}
	The ground states achieve the optimal constant in the critical Sobolev inequality, which was proved by Aubin \cite{Aubin76} and Talenti \cite{Talenti76}.
	They are the ``mountain passes'' for the potential energy.
	
	Kenig and Merle \cite{KeMe06} developed the concentration-compactness-rigidity argument to exhibit the special role of the ground states $\eee^{i\theta}W_\lambda$ as the \emph{threshold elements} for nonlinear dynamics of the solutions of \eqref{eq:nls}
	in spatial dimensions $N = 3, 4, 5$ for radial data. They proved the so-called \emph{Threshold Conjecture} by completely classifying the long time dynamical behavior of solutions $u(t)$ of \eqref{eq:nls}
	with $E(u(t)) < E(W)$. An analogous result in higher dimensions ($N\geq 4$), for non-radial data, was shown by Killip and Visan \cite{KiVi10} and Dodson \cite{Ben}. Duyckaerts and Merle \cite{DM09} classified  the long time dynamics of the radial solutions of \eqref{eq:nls} at the critical level $E(u(t)) = E(W)$ in space dimensions $N = 3, 4, 5$ and Li and Zhang \cite{LiZhang09} extended the results to higher dimensions. When the energy is at most slightly larger than that of the ground states, Nakanishi and Schlag \cite{NaSc12}, see also Nakanishi and Roy \cite{NaRo15p}, made use of the concentration-compactness-rigidity argument and one-pass lemma to show the global dynamics of all solutions of \eqref{eq:nls}  in dimension 3.
	
	A much stronger statement about the dynamics of solutions is the \emph{Soliton Resolution Conjecture}, which predicts that a global, bounded (in an appropriate sense) solution
	decomposes asymptotically into a sum of energy bubbles with weak interactions (i.e., at different scales) and a radiation term (a solution of the linear Schr\"odinger equation).
	This was proved for radial solutions of the focusing, energy-critical NLW in dimension $N \geq 3$ by the concentration-compactness-rigidity argument and the channel of energy by Duyckaerts, Kenig and Merle \cite{DKM4} in three dimension, and \cite{DKM23} in all odd dimensions, Jendrej and Lawrie \cite{JaLa23} in all even dimensions, and Duyckaerts, Jia, Kenig and Merle \cite{DJKM} in three dimension for the non-radial case, see also Duyckaerts, Kenig, Martel, and Merle \cite{DKMM} in four dimension, and Collot, Duyckaerts, Kenig and Merle \cite{CDKM24} in six dimension. For nonlinear Schr\"odinger equation \eqref{eq:nls} this problem is completely open. Related works can be found in \cite{Chenjie16p, Merle90, MeRa05-2, MeRa05, OrPe13, Tao07} for NLS, \cite{MeRaRo13, Perelman14} for Schr\"odinger maps, \cite{CDKM24, DKM23, JaMa20, Kenig15, KrScTa08, KrScTa09, MaMe25, Shen26} for wave equations, \cite{Cote, JaKr25, JaLa18, JaLa23map, JaLa25, HwKi26} for wave maps, \cite{JLX26, LLX26} for Hartree equation and \cite{Ma05} for gKdV equation, and references therein.
	
	
	
	The first author \cite{Jacek:nls} has constructed the pure two-bubbles for the energy-critical focusing nonlinear Schr\"odinger equation in dimension $N \geq 7$.  Here we restrict to the case $N=6$. In this case, the ground state does not belong to $\dot{H}^{-1}(\bR^6)$  and the modulation parameters decay exponentially rather than polynomially. Moreover, in order to close the bootstrap argument, we also need to estimate the spectrum of the linearized operator. Finally, the orthogonality conditions involve localized functions, which requires a new coercivity estimate for the linearized energy.

	\subsection{Main results}
	In view of the Soliton Resolution Conjecture, solutions which exhibit no dispersion in one or both time directions play a distinguished role.
	One obvious example of such solutions is the static solutions $\eee^{i\theta}W_\lambda$.
	As in \cite{Jacek:nls}, we construct global radial solutions in dimension $N =6 $ which approach, in the energy space, a sum of two bubbles. The ratio of the scales at which these bubbles develop tends to $0$.
	
	\begin{mainthm}
		\label{thm:deux-bulles}
		There exists a solution $u: (-\infty, T_0] \to \cE$ of \eqref{eq:nls} such that
		\begin{equation}
			\label{eq:mainthm}
			\lim_{t\to -\infty}\big\|u(t) - \big({-}iW + W_{e^{-|t|}}\big)\big\|_\cE = 0.
		\end{equation}
	\end{mainthm}
	\begin{remark}
		More precisely, we will prove that
		\begin{equation*}
			\big\|u(t) - \big({-}iW + W_{e^{-|t|}}\big)\big\|_\cE \leq Ce^{-\frac{1}{2}|t|},
		\end{equation*}
		for some constant $C > 0$; see \eqref{eq:uniform}.
	\end{remark}
	\begin{remark}
		We construct here \emph{pure} two-bubble solution, that is the solution approaches a superposition of two stationary states, with no energy transformed into radiation.
		By the conservation of energy and the decoupling of the two bubbles, we necessarily have $E(u(t)) = 2E(W)$.
		The radial threshold solutions (\emph{pure} one-bubble solution) cannot concentrate and are completely classified, see \cite{DM09}.
	\end{remark}
	\begin{remark}
		In \cite{moi16p}, the first author established an analogous result for the energy-critical wave equation in space dimension $6$ by use of a different proof scheme.
	\end{remark}
	\begin{remark}
		We expect that the phases of the two bubbles forming the right angle is the only configuration
		in which a \emph{pure} two-bubble solution can form. 
	\end{remark}
\begin{remark}
    The heuristic argument presented in Appendix~\ref{decay estimate} indicates that pure two-bubble solutions should not exist in dimensions $N \leq 5$. Thus, the present work settles the only remaining case left open by the first author in \cite{Jacek:nls}.
\end{remark}

	\subsection{Outline of the proof}

    We study solutions of \eqref{eq:nls} close to a superposition of two bubbles:
   \begin{equation}
      u(t)=\eee^{i\zeta(t)}W_{\mu(t)}+\eee^{i\theta(t)}W_{\lambda(t)}+g(t),
   \end{equation}
   where $\zeta(t)\simeq-\frac{\pi}{2}$, $\mu(t)\simeq 1$, $\theta(t)\simeq0$, $\lambda(t)\ll1$, and $\|g(t)\|_{\cE}\ll1$. The parameters describe the geometry of the two bubbles, while the remainder term $g$ represents the infinite-dimensional radiation part of the dynamics and satisfies
	\begin{align}
		\label{eq:orth}
		\la i\phi_{R\mu(t)}\eee^{i\zeta(t)}\Lambda W_{\mu(t)}, g(t)\ra &= \la -\phi_{R\mu(t)}\eee^{i\zeta(t)}W_{\mu(t)}, g(t)\ra \\
        &=\la i\phi_{R\lambda(t)}\eee^{i\theta(t)}\Lambda W_{\lambda(t)}, g(t)\ra
        = \la -\phi_{R\lambda(t)}\eee^{i\theta(t)}W_{\lambda(t)}, g(t)\ra =  0.
	\end{align}
	Here, $R$ is a sufficiently large constant independent of $t$. The orthogonality conditions are imposed against localized versions of the kernel elements of $Z_{\theta, \lambda}^*$, where $Z_{\theta,\lambda}$ is the linearized operator around $\eee^{i\theta}W_\lambda$. The localization is necessary since $W\notin \dot{H}^{-1}(\bR^6)$. We will often omit the time variable and write $\zeta$ for $\zeta(t)$ etc. We also introduce the unstable/stable coordinates
   \begin{equation}\label{a}
    a_1^\pm:=\la\alpha_{\zeta,\mu}^\pm,g\ra,
    \qquad
    a_2^\pm:=\la\alpha_{\theta,\lambda}^\pm,g\ra,
   \end{equation}
   where $\alpha_{\zeta,\mu}^{\pm}$ and $\alpha_{\theta,\lambda}^{\pm}$ are defined in \eqref{eq:alpha}.

   The overall structure of the proof follows the strategy by the first author in \cite{Jacek:nls}. The main ingredients are the modulation analysis, the spectral theory of the linearized operator, the topological argument, and the concentration compactness rigidity argument. However, the six-dimension setting presents several additional difficulties that require new ideas.

   The first major difficulty is that the ground state $W$ does not belong to $\dot H^{-1}(\bR^6)$, leading to the loss of coercivity of the energy functional of the linearized operator. Because of this, the standard global orthogonality conditions used in higher dimensions are no longer available. To overcome this issue, we introduce localized orthogonality conditions using a cut-off function, see \eqref{eq:orth}. We establish new localized coercivity estimates for the linearized energy, with the constant depending explicitly on the localization parameter $R$. Next, we first derive rough modulation estimates by differentiating the orthogonality conditions \eqref{eq:orth} and analyzing the resulting modulation system. Since these estimates are not enough to close the bootstrap argument, we introduce modified modulation parameters through suitable integral correction terms and derive refined modulation equations. It is shown that the differences between the modified and original modulation parameters remain sufficiently small, while the refined modulation equations provide improved evolution estimates. These estimates are crucial for closing the bootstrap argument and constructing the \emph{pure} two-bubble solution.

   Another important difference from \cite{Jacek:nls} is that the modulation parameters now decay exponentially rather than polynomially. To capture this new dynamical behavior, we employ a Lagrangian formulation in Appendix~\ref{decay estimate}. Finally, in order to control the unstable directions and close the bootstrap estimates, we perform a detailed spectral analysis of the linearized operator and establish the bound $\nu\le \frac12$ in \eqref{nu}.

   In Section~\ref{sec:variational}, we establish the spectral properties of the linearized operator $Z_{\theta,\lambda}$ around $\eee^{i\theta}W_\lambda$. We also prove coercivity estimates for the linearized energy near a two-bubble. The corresponding coercivity constants depend explicitly on the localization parameter $R$ and these estimates are crucial to estimate the infinite-dimensional part $g$. We will use the energy conservation to deal with this. Moreover, in order to close the bootstrap argument for the unstable directions, we derive a refined estimate on the spectrum of the linearized operator and obtain the key eigenvalue bound $\nu\leq\frac12$ in \eqref{nu}.  Finally, in order to finish the proof of Theorem~\ref{thm:deux-bulles}, we recall some useful lemmas, which are the continuation criterion and the weak stability lemma. 

   In Section~\ref{sec:mod}, we derive the modulation equations governing the evolution of the parameters. We first establish rough modulation estimates under suitable bootstrap assumptions in Lemma~\ref{lem:basicmod}. We then introduce modified modulation parameters through introducing suitable integral correction terms and derive refined modulation equations in Lemma~\ref{lem:mod}. In particular, we prove that the modified parameters differ from the original ones only by sufficiently small errors and satisfy improved evolution estimates. Finally, in Lemma~\ref{lem:proper}, we obtain quantitative control of the unstable modes, yielding an improved estimate for $\|g\|_{\mathcal E}$ and thereby closing the bootstrap argument.

   Finally, in Section~\ref{sec:boot}, we prove Theorem~\ref{thm:deux-bulles}. We first construct the modulation decomposition via the implicit function theorem and establish the bootstrap estimates. The parameters $\zeta$ and $\mu$ are estimated directly by integrating their modulation equations. The modified scaling parameter $\tilde{\lambda}$ and the unstable mode $\tilde{a}_1^+$ are then estimated simultaneously by analyzing the associated linear ODE system through its fundamental solution, and the corresponding estimates for $\lambda$ and $a_1^+$ follow from the smallness of the differences between the modified and original parameters. The estimate for $\theta$ is recovered from the refined modulation equation for $\tilde{\theta}$ together with a virial correction, which yields the additional smallness gain needed to close the bootstrap argument. Finally, the remaining unstable mode $a_2^+$ is treated by a topological argument based on the Wa\.zewski principle~\cite{wazewski}, thereby completing the proof of the theorem.


	\subsection{Acknowledgments}
	J. Jendrej was supported by  ERC project INSOLIT (No. 101117126).
	G. Xu was supported  by  NSFC (No. 12371240, No. 12431008).
	
	\subsection{Notation}
	For $z = x + iy \in \bC$ we denote $\Re(z) = x$ and $\Im(z) = y$. For two functions $v, w \in L^2(\bR^6, \bC)$ we denote
	\begin{equation}
		\la v, w\ra := \Re\int_{\bR^N} \conj{v(x)}\cdot w(x)\ud x.
	\end{equation}
	In this paper all the functions are radially symmetric. We write $L^2 := L^2_{\tx{rad}}(\bR^6; \bC)$ and $\cE := \dot H^1_{\tx{rad}}(\bR^6; \bC)$.
	We will think of them as of \emph{real} vector spaces.
	We denote $X^1 := \cE \cap \dot H^2(\bR^6)$. We also introduce the generators of the $\dot{H}^1$-critical and $L^2$-critical scaling. For a function $v:\bR^6\to \bC$, we define
    $$\Lambda v:=-\frac{\partial}{\partial\lambda}\Big|_{\lambda=1}(v_\lambda)=(2+x\cdot\nabla)v,$$
    $$\Lambda_0 v:=-\frac{\partial}{\partial\lambda}\Big|_{\lambda=1}(\frac{1}{\lambda}v_\lambda)=(3+x\cdot\nabla)v.$$
	
	\section{Variational estimates}
	\label{sec:variational}
	\subsection{Linearization near a ground state}
	Recall that for $u \in \bC$ we denote $f(u) := |u|u$ and $F(u) := \frac{1}{3}|u|^3$.
	For $u \in \bC$ we define the $\bR$-linear function $f'(u): \bC \to \bC$ by the following formula:
	\begin{equation}
		f'(u)g := |u|\Big(g +u\Re(u^{-1}g)\Big)
	\end{equation}
	(with the convention $f'(0)g = 0$). It is easy to check that for any $g, h, u \in \bC$ there holds
	\begin{equation}
		\label{eq:auto-scalar}
		\Re\big(\conj h(f'(u)g)\big) = \Re \big(\conj g(f'(u)h)\big) = \Re\big((\conj{f'(u)h})g\big).
	\end{equation}
	Integrating this identity on $\bR^6$ we see that for a complex function $u(x)$ the operator $g \mapsto f'(u)g$ is symmetric with respect to the real $L^2$ scalar product.
	We denote $|f'(u)| := 2|u|$, which is the norm of $f'(u)$ as a linear map up to a constant.
	For $u: \bR^6 \to \bC$ we define $\|f'(u)\|_{L^p} := \big(\int_{\bR^6}|f'(u(x))|^p\ud x\big)^\frac 1p$ for $1 \leq p < +\infty$ and $\|f'(u)\|_{L^\infty} := \sup_{x \in \bR^6}|f'(u(x))|$.
	\begin{lemma}
		\label{lem:pointwise}
		For $z_1, z_2, z_3 \in \bC$ there holds
		\begin{gather}
			|f'(z_1 + z_2) - f'(z_1)| \lesssim |f'(z_2)|\lesssim|z_2|, \label{eq:pointwise-5} \\
			|f(z_1 + z_2) - f(z_1)| \lesssim |f'(z_1)|\cdot|z_2| + |f(z_2)|\lesssim|z_1||z_2|+|z_2|^2, \label{eq:pointwise-2} \\
				|f(z_1 + z_2) - f(z_1) - f'(z_1)z_2| \lesssim f(|z_2|) \lesssim|z_2|^2,\label{eq:pointwise-1} \\
			\big|F(z_1 + z_2) - F(z_1) - \Re\big(\conj{f(z_1)}\cdot z_2\big)\big| \lesssim |f'(z_1)|\cdot|z_2|^2 + F(z_2)\lesssim|z_1||z_2|^2+|z_2|^3, \label{eq:pointwise-6} \\
			\big|F(z_1 + z_2) - F(z_1) - \Re\big(\conj{f(z_1)}\cdot z_2\big) - \Re\big(\conj{f'(z_1)z_2}\cdot z_2\big)\big| \lesssim F(z_2)\lesssim|z_2|^3. \label{eq:pointwise-3}
		\end{gather}
	\end{lemma}
	\begin{proof}
		The proof is similar to that in \cite[Lemma 2.1]{Jacek:nls} and is therefore omitted.
		
	\end{proof}
	We denote $Z_{\theta, \lambda} := i\Delta + if'(\eee^{i\theta}W_\lambda)$ the linearization of $i\Delta u + if(u)$ near $u = \eee^{i\theta}W_\lambda$.
	In order to express $Z_{\theta, \lambda}$ in a more explicit way, we introduce the following notation:
	\begin{equation}
		V^+ := -2W, \qquad V^- := -W, \qquad L^+ := -\Delta +V^+, \qquad L^- := -\Delta + V^-.
	\end{equation}
	It is known that for all $g \in \cE$ there hold $\la g, L^- g\ra \geq 0$ and $\ker L^- = \spn(W)$.
	The operator $L^+$ has one simple strictly negative eigenvalue and,
	restricting to radially symmetric functions, $\ker L^+ = \spn(\Lambda W)$.
	
	For future reference, we provide here the values of some integrals involving $W$ and $\Lambda W$:
	\begin{align}
		\int_{\bR^6} W^2 \ud x &=2306\pi^3=:C_1, \label{eq:explicit-1} \\
		-2\int_{\bR^6} W\Lambda W \ud x &=4608\pi^3=:C_2=2C_1 \label{eq:explicit-3}.
	\end{align}
	For the first integral,  we write $W^2 = -\Delta W$ and we integrate by parts.
	For the second integral, we write $-2W\Lambda W = V^+\Lambda W = \Delta \Lambda W$ and we integrate by parts.
	
	Using the definition of $f'$, one can check that if $g_1 = \Re g$ and $g_2 = \Im g$, then
	\begin{equation}
		Z_{\theta, \lambda}(\eee^{i\theta}g_\lambda) = \frac{\eee^{i\theta}}{\lambda^2}(L^-g_2 - iL^+ g_1)_\lambda.
	\end{equation}
	In particular, we obtain
	\begin{gather}
		Z_{\theta, \lambda}(i\eee^{i\theta}W_\lambda) = \frac{\eee^{i\theta}}{\lambda^2}(L^- W)_\lambda = 0, \\
		Z_{\theta, \lambda}(\eee^{i\theta}\Lambda W_\lambda) = \frac{\eee^{i\theta}}{\lambda^2}(-iL^+ \Lambda W)_\lambda = 0.
	\end{gather}
	This can also be seen by differentiating $i\Delta (\eee^{i\theta}W_\lambda) + if(\eee^{i\theta}W_\lambda)=0$
	with respect to $\theta$ and $\lambda$.

	%
	One can show that there exist real functions $\cY^{(1)}, \cY^{(2)} \in \cS$ and a real number $\nu > 0$ such that
	\begin{equation}
		\label{eq:Y1Y2}
		L^+ \cY^{(1)} = -\nu \cY^{(2)}, \qquad L^- \cY^{(2)} = \nu \cY^{(1)}
	\end{equation}
	(the proof given in \cite[Section 7]{DM09} for $N = 3, 4, 5$ and \cite{LiZhang09} for $N \geq 6$). We can assume that $\|\cY^{(1)}\|_{L^2} = \|\cY^{(2)}\|_{L^2} = 1$. 
	 We denote
	\begin{equation}
		\label{eq:alpha}
		\alpha_{\theta, \lambda}^+ := \frac{\eee^{i\theta}}{\lambda^2}\big(\cY_\lambda^{(2)} + i\cY_\lambda^{(1)}\big), \qquad \alpha_{\theta, \lambda}^- := \frac{\eee^{i\theta}}{\lambda^2}\big(\cY_\lambda^{(2)} - i\cY_\lambda^{(1)}\big).
	\end{equation}
	For $g = g_1 + ig_2$ we have $\la \alpha_{\theta, \lambda}^+, \eee^{i\theta}g_\lambda\ra = \la \cY^{(2)}, g_1\ra + \la \cY^{(1)}, g_2\ra$
	and $\la \alpha_{\theta, \lambda}^-, \eee^{i\theta}g_\lambda\ra = \la \cY^{(2)}, g_1\ra - \la \cY^{(1)}, g_2\ra$.
	Note that
	\begin{gather}
		\la W, \cY^{(1)}\ra = \frac{1}{\nu}\la W, L^- \cY^{(2)}\ra = \frac{1}{\nu}\la L^- W, \cY^{(2)}\ra = 0, \label{eq:WY}\\
		\la \Lambda W, \cY^{(2)}\ra = -\frac{1}{\nu}\la \Lambda W, L^+ \cY^{(1)}\ra = -\frac{1}{\nu}\la L^+(\Lambda W), \cY^{(1)}\ra = 0.\label{eq:LWY}
	\end{gather}
	It follows that
	\begin{gather}
		\la \alpha_{\theta, \lambda}^+, i\eee^{i\theta}W_\lambda \ra = \la \alpha_{\theta, \lambda}^-, i\eee^{i\theta}W_\lambda\ra = 0, \label{eq:proper-iW} \\
		\la \alpha_{\theta, \lambda}^+, \eee^{i\theta}\Lambda W_\lambda \ra = \la \alpha_{\theta, \lambda}^-, \eee^{i\theta}\Lambda W_\lambda\ra = 0. \label{eq:proper-LW}
	\end{gather}
	Since $\cY^{(2)} \neq W$, we also have
	\begin{equation}
		\label{eq:Y1Y2-prod}
		\la \cY^{(1)}, \cY^{(2)}\ra = \frac{1}{\nu}\la \cY^{(2)}, L^-\cY^{(2)}\ra=: M > 0.
	\end{equation}

    In order to close the bootstrap estimate of unstable components $a_1^+(t),\;a_2^+(t)$, which are defined by \eqref{a}, we claim: 
	\begin{equation}\label{nu}
		\nu\leq\frac12.
	\end{equation}
	From the definition of $L^+ $ and  $L^- $, we have $L^+=L^--W$. So we get
	\begin{equation}
		L^- \cY^{(1)} = W\cY^{(1)}-\nu \cY^{(2)},
	\end{equation}
	hence, the self-adjointness of $L^- $ yield
	\begin{align}
		\|L^- \cY^{(1)} \|_{L^2}^2&=\la L^- \cY^{(1)},L^- \cY^{(1)} \ra
		=\la W\cY^{(1)}-\nu \cY^{(2)},L^- \cY^{(1)} \ra\\
		&=\la W\cY^{(1)},L^- \cY^{(1)} \ra-\nu\la L^- \cY^{(2)},\cY^{(1)} \ra\\
		&\leq\|W\cY^{(1)} \|_{L^2}\|L^-\cY^{(1)} \|_{L^2}-\nu\la \nu \cY^{(1)},\cY^{(1)} \ra\\
		&\leq\|L^-\cY^{(1)} \|_{L^2}-\nu^2.
	\end{align}
	Thus, we have
	\begin{equation}
		\nu^2\leq \|L^-\cY^{(1)} \|_{L^2}-\|L^-\cY^{(1)} \|_{L^2}^2\leq\frac{1}{4}.
	\end{equation}
	Therefore, the claim holds.

		In order to prove the coercivity, we define
	\begin{equation}
		\label{eq:cY}
		\cY_{\theta, \lambda}^+ := \frac{1}{2M}\eee^{i\theta}\big(\cY_\lambda^{(1)} + i\cY_\lambda^{(2)}\big), \qquad \cY_{\theta, \lambda}^- := \frac{1}{2M}\eee^{i\theta}\big(\cY_\lambda^{(1)} - i\cY_\lambda^{(2)}\big).
	\end{equation}
	
	As in (2-10), (2-11) of \cite{Jacek:nls}, we can get $\{\eee^{i\theta}W_\lambda, i\eee^{i\theta}\Lambda W_\lambda\} \subset \ker Z_{\theta, \lambda}^*$ (In fact, from \cite{Rey1990jfa}, we can get $\{\eee^{i\theta}W_\lambda, i\eee^{i\theta}\Lambda W_\lambda\} = \ker Z_{\theta, \lambda}^*$ in  the radial case). 
    Moreover, $\alpha_{\theta, \lambda}^+$ and $\alpha_{\theta, \lambda}^-$ are eigenfunctions of $Z_{\theta, \lambda}^*$,
	with eigenvalues $\frac{\nu}{\lambda^2}$ and $-\frac{\nu}{\lambda^2}$ respectively; see (2-17) in \cite{Jacek:nls}.

	In order to define the  localized orthogonality conditions, we first define the cut-off function $\phi(x)$. Let $\phi(x)$ be a smooth radial function such that 
	$\phi(x)=1$ for $|x|\leq 1$, $\phi(x)=0$ for $|x|\geq 2$,  $\phi_R(x)=\phi\left(\frac{x}{R}\right),R>0$ and $\phi_{R\lambda}(x)=\phi\left(\frac{x}{R\lambda}\right).$
	\subsection{Coercivity of the energy near a two-bubble}
	\label{ssec:coer-en}
	We consider $u \in \cE$ of the form $u = \eee^{i\zeta(t)}W_{\mu(t)} + \eee^{i\theta(t)}W_{\lambda(t)} + g(t)$
	with
	\begin{equation}
		\big|\zeta(t) + \frac{\pi}{2}\big| + |\mu(t) - 1| + |\theta(t)| + \lambda(t) + \|g(t)\|_\cE \ll 1.
	\end{equation}
	Moreover, we will assume that $g$ satisfies \eqref{eq:orth}.

	Our objective is to prove the following result.
	\begin{proposition}
		\label{prop:coercivity}
		There exist constants $\eta, C_0, C > 0$ independent of $R$ such that for all $u \in \cE$ of the form $u = \eee^{i\zeta}W_\mu + \eee^{i\theta}W_\lambda + g$,
		with $\big|\zeta + \frac{\pi}{2}\big| + |\mu - 1| + |\theta| + \lambda + \|g\|_\cE \leq \eta$ and $g$ verifying \eqref{eq:orth}, there holds
		\begin{gather}
			\label{eq:coer-bound}
			|E(u) - 2E(W)| \leq C\Big(\big(\big|\zeta + \frac{\pi}{2}\big| + |\mu - 1| + |\theta| + \lambda\big)\lambda^2 + \|g\|_\cE^2 \Big), \\
			\label{eq:coer-conclusion}
			\begin{aligned}
				\|g\|_\cE^2 + C_0(\ln R)^{\frac{4}{3}}\theta\lambda^2 &\leq C(\ln R)^{\frac{4}{3}}\Big(\lambda^2\big(\big|\zeta+\frac{\pi}{2}\big| + |\mu - 1|
				+ |\theta|^3 + \lambda\big) \\
				&+ E(u) - 2E(W) + \sum_{j = 1, 2}\big((a_j^+)^2 + (a_j^-)^2\big)\Big).
			\end{aligned}
		\end{gather}
	\end{proposition}
	The scheme of the proof is as follows. The inequality \eqref{eq:pointwise-3} yields the Taylor expansion of the energy:
	\begin{equation}
		\label{eq:energy-taylor}
		\Big|E(u) - E(\eee^{i\zeta}W_\mu + \eee^{i\theta}W_\lambda) - \la \vD E(\eee^{i\zeta}W_\mu + \eee^{i\theta}W_\lambda), g\ra -
		\frac 12 \la \vD^2 E(\eee^{i\zeta}W_\mu + \eee^{i\theta}W_\lambda)g, g\ra\Big| \lesssim \|g\|_\cE^3.
	\end{equation}
	We just have to compute all the terms with a sufficiently high precision.
	We split this computation into a few lemmas.
	\begin{lemma}
		\label{lem:coer-sans-g}
		Let $\zeta, \mu, \theta, \lambda$ be as in Proposition~\ref{prop:coercivity}. Then
		\begin{equation}
			\label{eq:coer-sans-g}
				\Big|E(\eee^{i\zeta}W_\mu + \eee^{i\theta}W_\lambda) - 2E(W) -C_1 \theta\lambda^2\Big|
				\leq C\lambda^2\big(\big|\zeta + \frac{\pi}{2}\big| + |\mu - 1| + |\theta|^3 + \lambda\big),
		\end{equation}
		where $C_1$ is given by \eqref{eq:explicit-1} and $C>0$ is a constant. 
	\end{lemma}
	\begin{proof}The proof of this lemma is the same as \cite[Lemma 2.5]{Jacek:nls}, so we omit it.
	\end{proof}

	\begin{lemma}
		\label{lem:energy-linear}
		Under the assumptions of Proposition~\ref{prop:coercivity}, there holds
		\begin{equation}
			\label{eq:energy-linear}
			\big|\la \vD E(\eee^{i\zeta}W_\mu + \eee^{i\theta}W_\lambda), g\ra\big| \lesssim \|g\|_\cE\cdot \lambda^\frac32.
		\end{equation}
	\end{lemma}
	\begin{proof}
		Using the fact that $\vD E(\eee^{i\zeta}W_\mu) = \vD E(\eee^{i\theta}W_\lambda) = 0$, \eqref{eq:energy-linear} is seen
		to be equivalent to
		\begin{equation}
			\label{eq:energy-linear-1}
			\big|\la f(\eee^{i\zeta} W_\mu + \eee^{i\theta}W_\lambda) - f(\eee^{i\zeta}W_\mu) - f(\eee^{i\theta}W_\lambda), g\ra\big| \lesssim \|g\|_\cE\cdot \lambda^\frac32.
		\end{equation}
		By the Sobolev inequality, it suffices to check that
		\begin{equation}
			\label{eq:energy-linear-2}
			\|f(\eee^{i\zeta} W_\mu + \eee^{i\theta}W_\lambda) - f(\eee^{i\zeta} W_\mu) - f(\eee^{i\theta}W_\lambda)\|_{L^\frac32} \lesssim \lambda^\frac32.
		\end{equation}
		By \eqref{eq:pointwise-2} we have
		\begin{equation}
			|f(\eee^{i\zeta} W_\mu + \eee^{i\theta}W_\lambda) - f(\eee^{i\zeta} W_\mu) - f(\eee^{i\theta}W_\lambda)| \lesssim W_\lambda W_\mu.
		\end{equation}
		As usual, we consider separately the regions $|x| \leq \sqrt\lambda$ and $|x| \geq \sqrt\lambda$. In the first region we have $W_\mu \lesssim 1$. Hence, by a change of variable, we obtain
		\begin{equation}
			\begin{aligned}
				\|W_\lambda\|_{L^\frac32(|x| \leq \sqrt\lambda)} &= \lambda^2\|W\|_{L^\frac32(|x| \leq 1/\sqrt\lambda)} \\
				&\lesssim \lambda^2\Big(\int_0^1 \ud x+\int_1^{1/\sqrt\lambda} r^{-4\times\frac32} r^5\ud r\Big)^\frac23 \lesssim \lambda^2 \big(\ln(\lambda^{-\frac12})\big)^\frac23\lesssim \lambda^\frac32.
			\end{aligned}
		\end{equation}
		
		In the region $|x| \geq \sqrt\lambda$ we have $W_\mu \lesssim W$. Hence, we have
		\begin{equation}
			\begin{aligned}
				\|W_\lambda W\|_{L^\frac32 (|x|\geq \sqrt\lambda)} &= \lambda^{-2}(\int_{|x|\geq \sqrt\lambda}\la \frac{x}{\lambda}\ra^{-6}\la x \ra^{-6}\ud x)^{\frac23} \\
				&\lesssim \lambda^{-2}\Big(\int_{\sqrt\lambda\leq |x|\leq 1}\la \frac{x}{\lambda}\ra^{-6}\ud x+\int_{|x|\geq1} \la \frac{x}{\lambda}\ra^{-6}\la x \ra^{-6}\ud x\Big)^\frac23\\
				&\lesssim\lambda^{-2}\Big(\int_{\sqrt\lambda\leq |x|\leq 1}| \frac{x}{\lambda}|^{-6}\ud x+\int_{|x|\geq1} | \frac{x}{\lambda}|^{-6}| x| ^{-6}\ud x\Big)^\frac23\\
			    &\lesssim \lambda^2 \big(\ln(\lambda^{-\frac12})\big)^\frac23\lesssim \lambda^\frac32.
			\end{aligned}
		\end{equation}
	\end{proof}
	
	We now examine the coercivity of the quadratic part in \eqref{eq:energy-taylor}.
	
	\begin{lemma}
		\label{lem:coer-Lp-Lm}
		There exist constants $c, C > 0$, which are independent of $R$, such that
		\begin{itemize}
			\item for any real-valued radial function $h \in \cE$ there holds
			\begin{gather}
				\label{eq:coer-Lp-1}
				\la h, L^+h\ra \geq c(\ln R)^{-\frac{4}{3}}\int_{\bR^6}|\grad h|^2 \ud x -C\big(\la \phi_{R}W, h\ra^2 + \la \cY^{(2)}, h\ra^2\big), \\
				\label{eq:coer-Lm-1}
				\la h, L^-h\ra \geq c(\ln R)^{-\frac{4}{3}}\int_{\bR^6}|\grad h|^2 \ud x -C\la \phi_{R} \Lambda W, h\ra^2,
			\end{gather}
			\item if $r_1 \geq c^{-1}R(\ln R)^{\frac{4}{3}}$, then for any real-valued radial function $h \in \cE$ there holds
			\begin{gather}
					\label{eq:coer-Lp-2}
					(1-2c(\ln R)^{-\frac{4}{3}})\int_{|x|\leq r_1}|\grad h|^2 \ud x + c(\ln R)^{-\frac{4}{3}}\int_{|x|\geq r_1}|\grad h|^2\ud x + \int_{\bR^6}V^+|h|^2\ud x \\
					\geq -C\big(\la \phi_{R}W, h\ra^2 + \la \cY^{(2)}, h\ra^2\big),\\
					\label{eq:coer-Lm-2}
					(1-2c(\ln R)^{-\frac{4}{3}})\int_{|x|\leq r_1}|\grad h|^2 \ud x + c(\ln R)^{-\frac{4}{3}}\int_{|x|\geq r_1}|\grad h|^2\ud x + \int_{\bR^6}V^-|h|^2\ud x \\
                    \geq -C\la \phi_{R}\Lambda W, h\ra^2,
			\end{gather}
			\item if $0 < r_2 \leq c(\ln R)^{-\frac{4}{3}}$, then for any real-valued radial function $h \in \cE$ there holds
			\begin{gather}
				\label{eq:coer-Lp-3}
				(1-2c(\ln R)^{-\frac{4}{3}})\int_{|x|\geq r_2}|\grad h|^2 \ud x + c(\ln R)^{-\frac{4}{3}}\int_{|x|\leq r_2}|\grad h|^2\ud x + \int_{\bR^6}V^+|h|^2\ud x \\
				\geq -C\big(\la \phi_{R}W, h\ra^2 + \la \cY^{(2)}, h\ra^2\big), \\
				\label{eq:coer-Lm-3}
				(1-2c(\ln R)^{-\frac{4}{3}})\int_{|x|\geq r_2}|\grad h|^2 \ud x + c(\ln R)^{-\frac{4}{3}}\int_{|x|\leq r_2}|\grad h|^2\ud x + \int_{\bR^6}V^-|h|^2\ud x \\
                \geq -C\la \phi_{R}\Lambda W, h\ra^2.
			\end{gather}
		\end{itemize}
	\end{lemma}
	\begin{proof}
		First, we prove \eqref{eq:coer-Lp-1} and \eqref{eq:coer-Lm-1}. Compared with (2-28) and (2-29) in \cite{Jacek:nls}, the presence of the localization function makes the analysis more delicate. The main new difficulty is to determine how the localization parameter $R$ affects the coercivity constant.  
        
        \textbf{Claim:}\;For any real-valued radial function $h \in \cE$ there holds
			\begin{gather}
				\label{eq:coer-Lp-0}
				\la h, L^+h\ra + \frac{\nu}{M}\la \cY^{(2)}, h\ra^2 \gtrsim \inf_{a_0,b_0}\|h-a_0\cY^{(1)}-b_0\Lambda W\|_{\cE}^2, \\
				\label{eq:coer-Lm-0}
				\la h, L^-h\ra \gtrsim \inf_{c_0}\|h-c_0 W\|_{\cE}^2.
			\end{gather}
        Let $\chi\in \mathcal{S}(\mathbb{R}^N)$ be a real-valued radial function such that
        \begin{equation}\label{chi}
            \la W, \chi\ra\ne 0,\quad \la \Lambda W, \chi\ra\ne 0.
        \end{equation}
        We can decompose $h$ as $h=a_0\cY^{(1)}+b_0\Lambda W+z$, where $z$ satisfies 
        \begin{equation}\label{z}
            \la \cY^{(2)}, z\ra = 0,\quad \la \chi, z\ra = 0.
        \end{equation}
        By \eqref{eq:LWY} and \eqref{eq:Y1Y2-prod}, we know $a_0=\frac{\la\cY^{(2)}, h\ra}{M}.$
        Hence, using \eqref{eq:Y1Y2}, we get
        \begin{align*}
            \la h, L^+h\ra&=a_0^2\la \cY^{(1)}, L^+\cY^{(1)}\ra+2a_0\la L^+\cY^{(1)}, z\ra+\la z, L^+z\ra\\
            &=a_0^2\la \cY^{(1)}, -\nu\cY^{(2)}\ra+2a_0\la -\nu\cY^{(2)}, z\ra+\la z, L^+z\ra\\
            &=-\frac{\nu}{M}\la\cY^{(2)}, h\ra^2+\la z, L^+z\ra.
        \end{align*}
        Same as (2-28) in \cite{Jacek:nls}, we can  get $\la z, L^+z\ra\gtrsim \|z\|_{\cE}^2$. Thus, \eqref{eq:coer-Lp-0} holds. Similarly, we can get \eqref{eq:coer-Lm-0}.

        Now, we prove \eqref{eq:coer-Lp-1}, which will follow from
        \begin{equation}\label{eeq:coer-Lp-1}
            \la h, L^+h\ra + \frac{\nu}{M}\la \cY^{(2)}, h\ra^2\geq c(\ln R)^{-\frac{4}{3}}\int_{\bR^6}|\grad h|^2 \ud x -C\la \phi_{R}W, h\ra^2.
        \end{equation}
         Let $\|h\|_{\cE}=1.$ If $\la h, L^+h\ra + \frac{\nu}{M}\la \cY^{(2)}, h\ra^2 \gtrsim(\ln R)^{-\frac{4}{3}}$, then \eqref{eeq:coer-Lp-1} holds. Thus, we can assume  $ \la h, L^+h\ra + \frac{\nu}{M}\la \cY^{(2)}, h\ra^2 \ll\eta\ll(\ln R)^{-\frac{4}{3}}.$ Let $a_0, b_0$ be such that $h=a_0\cY^{(1)}+b_0\Lambda W+z$ and $z$ satisfies \eqref{z}. From \eqref{eq:coer-Lp-0}, we know that $\|z\|_{\cE}^2\lesssim \eta$. Then using the H\"older inequality we obtain 
         \begin{equation*}
             |\la \phi_{R}W, z\ra|\lesssim \|\phi_{R}W\|_{L^{\frac32}}\|z\|_{L^3} \lesssim (\ln R)^{\frac{2}{3}}\|z\|_{\cE}\lesssim (\ln R)^{\frac{2}{3}}\sqrt{\eta}.
         \end{equation*}
         Hence,
        \begin{equation}
            |\la \phi_{R}W, h\ra|\gtrsim |b_0|-|a_0|-(\ln R)^{\frac{2}{3}}\sqrt{\eta}\gtrsim |b_0|-|a_0|.
        \end{equation}
        Thus, if $|a_0|\gtrsim1$, i.e. $|\la\cY^{(2)}, h\ra|\gtrsim1$, we can get \eqref{eeq:coer-Lp-1} holds. If $|a_0|\ll 1$, then $|b_0|\simeq1$, thus \eqref{eeq:coer-Lp-1} holds. So we prove \eqref{eq:coer-Lp-1}. Similarly, we can get \eqref{eq:coer-Lm-1}.

        Next, we prove the bounds \eqref{eq:coer-Lp-2}, \eqref{eq:coer-Lm-2}, \eqref{eq:coer-Lp-3}, and \eqref{eq:coer-Lm-3}, which follow by repeating the proof of Lemma 2.1 in \cite{moi15p-3}. The only additional point is to keep track of the dependence on the localization parameter $R$.

    We define the projections $\Pi_r, \Psi_r: \dot H^1 \to \dot H^1$:
    \begin{equation*}
      (\Pi_rh)(x) := \Big\{
        \begin{aligned}
          h(r) \qquad &\text{if }|x| \leq r, \\
          h(x) \qquad &\text{if }|x| \geq r,
        \end{aligned} \qquad
      (\Psi_r h)(x) := \Big\{
        \begin{aligned}
          h(x) - h(r) \qquad &\text{if }|x| \leq r, \\
          0 \qquad &\text{if }|x| \geq r,
        \end{aligned}
    \end{equation*}
    (thus $\Pi_r + \Psi_r = \Id$).

    Applying \eqref{eq:coer-Lp-1} to $\Psi_{r_1}g$ with $c(\ln R)^{-\frac{4}{3}}$ replaced by $3c(\ln R)^{-\frac{4}{3}}$ and $C$ replaced by $\frac C2$ we get
    \begin{equation}
      \label{eq:lin-coer-dem-1}
      \begin{aligned}
        (1-2c(\ln R)^{-\frac{4}{3}})\int_{|x| \leq r_1}|\grad h|^2 \ud x &= (1-2c(\ln R)^{-\frac{4}{3}})\int_{\bR^6}|\grad(\Psi_{r_1}h)|^2 \ud x \\
        &\geq (1+c(\ln R)^{-\frac{4}{3}})\int_{\bR^6}(-V^+)|\Psi_{r_1}h|^2\ud x - \frac C2\big(\la \phi_{R}W, h\ra^2 + \la \cY^{(2)}, h\ra^2\big),
    \end{aligned}
    \end{equation}
    where $V^+ = -2W$.
    By Sobolev, H\"older inequalities and the fact $r_1 \geq c^{-1}R(\ln R)^{\frac{4}{3}}$ we have
    \begin{equation}
      \label{eq:lin-coer-dem-2}
      \begin{aligned}
      \int_{|x| \geq r_1}(-V^+)|h|^2 \ud x &= \int_{|x| \geq r_1}(-V^+)|\Pi_{r_1}h|^2 \ud x \\
      &\lesssim \|W\|_{L^3(|x| \geq r_1)}\cdot \|\Pi_{r_1}h\|_{\dot H^1}^2 \\
      &\leq r_1^{-2} \int_{|x| \geq r_1}|\grad h|^2 \ud x\\
      &\leq c^2R^{-2}(\ln R)^{-\frac{8}{3}}\int_{|x| \geq r_1}|\grad h|^2 \ud x\\
      &\leq\frac c4(\ln R)^{-\frac{4}{3}}\int_{|x| \geq r_1}|\grad h|^2 \ud x,
    \end{aligned}
    \end{equation}
    if $r_1$ is large enough.

    In the region $|x| \leq r_1$ we apply the pointwise inequality
    \begin{equation}
      \label{eq:lin-coer-dem-3}
      |h(x)|^2 \leq (1+c(\ln R)^{-\frac{4}{3}})|(\Psi_{r_1}h)(x)|^2 + (1 + c^{-1}(\ln R)^{\frac{4}{3}})|h(r_1)|^2,\qquad |x| \leq r_1.
    \end{equation}
    Recall that by the Strauss Lemma \cite{Strauss77}, for a radial function $h$ there holds
    \begin{equation*}
      |h(r_1)| \lesssim \|\Pi_{r_1}h\|_{\dot H^1}\cdot r_1^{-\frac{N-2}{2}} = \|\Pi_{r_1}h\|_{\dot H^1}\cdot r_1^{-2}.
    \end{equation*}
    Since $-V^+(r) \sim r^{-4}$ as $r \to +\infty$, we have
    \begin{equation*}
      \int_{|x|\leq r_1}-V^+\ud x \lesssim r_1^2 ,\qquad \text{as }r_1 \to +\infty,
    \end{equation*}
    hence
    \begin{align}
      \label{eq:lin-coer-dem-4}
      \int_{|x|\leq r_1}(-V^+)\cdot(1 + c^{-1}(\ln R)^{\frac{4}{3}})|h(r_1)|^2\ud x 
      &\lesssim (1 + c^{-1}(\ln R)^{\frac{4}{3}})\|\Pi_{r_1}h\|_{\dot H^1}\cdot r_1^{-4}\cdot r_1^2\\
      &\leq r_1^{-2}(1 + c^{-1}(\ln R)^{\frac{4}{3}})\|\Pi_{r_1}h\|_{\dot H^1}\\
      &\leq c^2R^{-2}(\ln R)^{-\frac{8}{3}}(1 + c^{-1}(\ln R)^{\frac{4}{3}})\|\Pi_{r_1}h\|_{\dot H^1}\\
      &\leq \frac c4 (\ln R)^{-\frac{4}{3}}\int_{|x| \geq r_1}|\grad h|^2 \ud x,
    \end{align}
    if $r_1$ is large enough.

    Estimates \eqref{eq:lin-coer-dem-2}, \eqref{eq:lin-coer-dem-3} and \eqref{eq:lin-coer-dem-4} yield
    \begin{equation}
      \label{eq:lin-coer-dem-5}
      \int_{\bR^6} (-V^+)|h|^2 \ud x \leq (1+c(\ln R)^{-\frac{4}{3}})\int_{\bR^6} (-V^+)|(\Psi_{r_1}h)(x)|^2\ud x + \frac{c}{2}(\ln R)^{-\frac{4}{3}} \int_{|x| \geq r_1}|\grad h|^2 \ud x.
    \end{equation}
    Using $r_1 \geq c^{-1}R(\ln R)^{\frac{4}{3}}>2R$ we get
    \begin{align*}
        |\la \phi_R W, \Pi_{r_1}h\ra| &\lesssim \int_{|x|\leq r_1} \phi_R W\cdot|\Pi_{r_1}h|\ud x + \int_{|x|\geq r_1} \phi_R W\cdot|\Pi_{r_1}h|\ud x\\
        &\lesssim |h(r_1|\int_{|x|\leq r_1} \phi_R W\ud x\\
        &\lesssim \|\Pi_{r_1}g\|_{\dot H^1}\cdot r_1^{-2}\int_{|x|\leq 2R} W\ud x\\
        &\lesssim \big(\frac{R}{r_1}\big)^2\|\Pi_{r_1}g\|_{\dot H^1}\\
        &\lesssim c^{{2}}(\ln R)^{-\frac{8}{3}}\|\Pi_{r_1}g\|_{\dot H^1}\\
        &\lesssim \frac c4(\ln R)^{-\frac{4}{3}}\|\Pi_{r_1}g\|_{\dot H^1},
    \end{align*}
    hence
    \begin{equation}
      \label{eq:lin-coer-dem-6}
      \frac C2\la \phi_R W, \Psi_{r_1}h\ra^2 \leq C\la \phi_R W, h\ra^2 + C\la \phi_R W, \Pi_{r_1} h\ra^2 \leq C\la \phi_R W, h\ra^2 + \frac c4(\ln R)^{-\frac{4}{3}}\int_{|x| \geq r_1}|\grad h|^2 \ud x,
    \end{equation}
    provided that $r_1$ is chosen large enough. Similarly,
    \begin{equation}
      \label{eq:lin-coer-dem-7}
      \frac C2\la \cY^{(2)}, \Psi_{r_1}h\ra^2 \leq C\la \cY^{(2)}, h\ra^2 + C\la \cY^{(2)}, \Pi_{r_1} h\ra^2 \leq C\la \cY^{(2)}, h\ra^2 + \frac c4(\ln R)^{-\frac{4}{3}}\int_{|x| \geq r_1}|\grad h|^2 \ud x.
    \end{equation}
    Estimate \eqref{eq:coer-Lp-2} follows from \eqref{eq:lin-coer-dem-1}, \eqref{eq:lin-coer-dem-5}, \eqref{eq:lin-coer-dem-6} and \eqref{eq:lin-coer-dem-7}. Similarly, we can get \eqref{eq:coer-Lm-2}.

    We turn to the proof of \eqref{eq:coer-Lp-3}.
    Applying \eqref{eq:coer-Lp-1} to $\Pi_{r_2}g$ with $c(\ln R)^{-\frac{4}{3}}$ replaced by $3c(\ln R)^{-\frac{4}{3}}$ and $C$ replaced by $\frac C2$ we get
    \begin{equation}
      \label{eq:lin-coer-dem-11}
      \begin{aligned}
        (1-3c(\ln R)^{-\frac{4}{3}})\int_{|x| \geq r_2}|\grad h|^2 \ud x &= (1-3c(\ln R)^{-\frac{4}{3}})\int_{\bR^6}|\grad(\Pi_{r_2}h)|^2 \ud x \\
        &\geq \int_{\bR^6}(-V^+)|\Pi_{r_2}h|^2\ud x - \frac C2\big(\la \phi_{R}W, \Pi_{r_2}h\ra^2 + \la \cY^{(2)}, \Pi_{r_2}h\ra^2\big).
    \end{aligned}
    \end{equation}
    By Sobolev and H\"older inequalities we have for $r_2$ small enough
    \begin{equation}
      \label{eq:lin-coer-dem-12}
      \int_{|x| \leq r_2}(-V^+)|h|^2 \ud x \leq r_2^2 \int_{\bR^6}|\grad h|^2 \ud x\leq c^2(\ln R)^{-\frac{8}{3}}\leq\frac c2 (\ln R)^{-\frac{4}{3}}\int_{\bR^6}|\grad h|^2 \ud x.
    \end{equation}
    By definition of $\Pi_r$ there holds
    \begin{equation*}
      \int_{|x| \geq r_2}(-V^+)|h|^2 \ud x \leq \int_{\bR^6}(-V^+)|\Pi_{r_2}h|^2\ud x,
    \end{equation*}
    hence \eqref{eq:lin-coer-dem-11} and \eqref{eq:lin-coer-dem-12} imply
    \begin{align}
      \label{eq:lin-coer-dem-13}
      &(1-2c(\ln R)^{-\frac{4}{3}})\int_{|x| \geq r_2}|\grad h|^2 \ud x + \frac c2(\ln R)^{-\frac{4}{3}}\int_{|x| \leq r_2}|\grad h|^2 \ud x \\
      \geq
      &\int_{\bR^6}(-V^+)|h|^2\ud x - \frac C2\big(\la \phi_{R}W, \Pi_{r_2}h\ra^2 + \la \cY^{(2)}, \Pi_{r_2}h\ra^2\big).
    \end{align}
    Using the fact that $0 < r_2 \leq c(\ln R)^{-\frac{4}{3}}$ we obtain
    \begin{equation*}
      |\la \phi_{R}W, \Psi_{r_2}h\ra| \lesssim \int_{|x| \leq r_2}\phi_{R}W|h|\ud x \lesssim \|\phi_{R}W\|_{L^\frac{3}{2}(|x| \leq r_2)}\|\Psi_{r_2} h\|_{\dot H^1}\lesssim r_2^4\|\Psi_{r_2} h\|_{\dot H^1}\leq \frac c4(\ln R)^{-\frac{4}{3}}\int_{|x| \leq r_2}|\grad h|^2 \ud x,
    \end{equation*}
    hence
    \begin{equation}
      \label{eq:lin-coer-dem-14}
      \frac C2\la \phi_{R}W, \Pi_{r_2}h\ra^2 \leq C\la \phi_{R}W, h\ra^2 + C\la \phi_{R}W, \Psi_{r_2} h\ra^2 \leq C\la \phi_{R}W, h\ra^2 + \frac c4(\ln R)^{-\frac{4}{3}}\int_{|x| \leq r_2}|\grad h|^2 \ud x,
    \end{equation}
    provided that $r_2$ is chosen small enough. Similarly,
    \begin{equation}
      \label{eq:lin-coer-dem-15}
      \frac C2\la \cY^{(2)}, \Pi_{r_2}h\ra^2 \leq C\la \cY^{(2)}, h\ra^2 + C\la \cY^{(2)}, \Psi_{r_2} h\ra^2 \leq C\la \cY^{(2)}, h\ra^2 + \frac c4(\ln R)^{-\frac{4}{3}}\int_{|x| \leq r_2}|\grad h|^2 \ud x.
    \end{equation}
    Estimate \eqref{eq:coer-Lp-3} follows from \eqref{eq:lin-coer-dem-13}, \eqref{eq:lin-coer-dem-14} and \eqref{eq:lin-coer-dem-15}. Similarly, we can get \eqref{eq:coer-Lm-3}.
	\end{proof}
	We now use this lemma to study the linearization around $\eee^{i\theta}W_\lambda$ for a complex-valued perturbation~$h$.
	\begin{proposition}
		\label{prop:coer-L}
		There exist constants $c, C > 0$, which are independent of $R$, such that for any $\theta \in \bR$ and $\lambda > 0$
		\begin{itemize}
			\item for any complex-valued radial function $h \in \cE$ there holds
			\begin{equation}
				\label{eq:coer-L-1}
				\begin{gathered}
					\int_{\bR^N}|\grad h|^2 \ud x - \Re \int_{\bR^N}\conj h\cdot f'(\eee^{i\theta}W_\lambda)h\ud x \geq  
					 c(\ln R)^{-\frac{4}{3}}\int_{\bR^N}|\grad h|^2 \ud x \\
					 -C\big(\la \lambda^{-2}\phi_{R\lambda}\eee^{i\theta}W_\lambda, h\ra^2 + \la \lambda^{-2}i\phi_{R\lambda}\eee^{i\theta}\Lambda W_\lambda, h\ra^2 + \la \alpha_{\theta, \lambda}^+, h\ra^2 + \la \alpha_{\theta, \lambda}^-, h\ra^2\big),
				\end{gathered}
			\end{equation}
			\item if $r_1 \geq c^{-1}R(\ln R)^{\frac{4}{3}}$, then for any complex-valued radial function $h \in \cE$ there holds
			\begin{equation}
				\label{eq:coer-L-2}
				\begin{gathered}
					(1-2c(\ln R)^{-\frac{4}{3}})\int_{|x|\leq r_1}|\grad h|^2 \ud x + c(\ln R)^{-\frac{4}{3}}\int_{|x|\geq r_1}|\grad h|^2\ud x - \Re \int_{\bR^N}\conj h\cdot f'(\eee^{i\theta}W_\lambda)h\ud x \geq \\
				 {-}C\big(\la \lambda^{-2}\phi_{R\lambda}\eee^{i\theta}W_\lambda, h\ra^2 + \la \lambda^{-2}i\phi_{R\lambda}\eee^{i\theta}\Lambda W_\lambda, h\ra^2 + \la \alpha_{\theta, \lambda}^+, h\ra^2 + \la \alpha_{\theta, \lambda}^-, h\ra^2\big),
				\end{gathered}
			\end{equation}
			\item if $0 < r_2 \leq c(\ln R)^{-\frac{4}{3}}$, then for any complex-valued radial function $h \in \cE$ there holds
			\begin{equation}
				\label{eq:coer-L-3}
				\begin{gathered}
					(1-2c(\ln R)^{-\frac{4}{3}})\int_{|x|\geq r_2}|\grad h|^2 \ud x + c(\ln R)^{-\frac{4}{3}}\int_{|x|\leq r_2}|\grad h|^2\ud x - \Re \int_{\bR^N}\conj h\cdot f'(\eee^{i\theta}W_\lambda)h\ud x \geq \\
					{-}C\big(\la \lambda^{-2}\phi_{R\lambda}\eee^{i\theta}W_\lambda, h\ra^2 + \la \lambda^{-2}i\phi_{R\lambda}\eee^{i\theta}\Lambda W_\lambda, h\ra^2 + \la \alpha_{\theta, \lambda}^+, h\ra^2 + \la \alpha_{\theta, \lambda}^-, h\ra^2\big).
				\end{gathered}
			\end{equation}
		\end{itemize}
	\end{proposition}
	\begin{remark}
		Note that the scalar products on the right hand side of these estimates are the ones which appear in the orthogonality conditions. For the definition of $\alpha_{\theta, \lambda}^\pm$, see \eqref{eq:alpha}.
	\end{remark}
		\begin{proof}The proof can be seen in \cite[Proposition 2.8]{Jacek:nls}.
	\end{proof}
	
	One consequence of the last proposition is the coercivity near a sum of two bubbles at different scales:
	\begin{lemma}
		\label{lem:coer-L-two}
		There exist $\eta, C > 0$ such that if $\lambda \leq \eta\mu$, then for all $g \in \cE$ satisfying \eqref{eq:orth}
		there holds
		\begin{equation}
			\label{eq:coer-L-two}
			\|g\|_\cE^2 \leq C(\ln R)^{\frac{4}{3}}\Big(\frac 12 \la \vD^2 E(\eee^{i\zeta}W_\mu + \eee^{i\theta}W_\lambda)g, g\ra + \frac{\nu}{2M}\big((a_1^+)^2 + (a_1^-)^2 + (a_2^+)^2 + (a_2^-)^2\big)\Big) .
		\end{equation}
	\end{lemma}
	\begin{proof}
		It is essentially the same as the proof of \cite[Lemma 3.5]{moi15p-3}.
	\end{proof}
	\begin{proof}[Proof of Proposition~\ref{prop:coercivity}]
		Bound \eqref{eq:coer-bound} follows immediately from \eqref{eq:energy-taylor}, Lemmas~\ref{lem:coer-sans-g}, \ref{lem:energy-linear},
		\ref{lem:coer-L-two} and the triangle inequality.
		
		For any $c > 0$ we have $\|g\|_\cE^3 \leq c\|g\|_\cE^2$ if $\eta$ is chosen small enough,
		hence \eqref{eq:energy-taylor} and Lemmas~\ref{lem:coer-sans-g}, \ref{lem:energy-linear} yield
		\begin{equation}
			\begin{aligned}
				&\Big|E(u) - 2E(W) -C_1\theta\lambda^2 - \frac 12\la \vD^2 E(\eee^{i\zeta}W_\mu + \eee^{i\theta}W_\lambda)g, g\ra\Big| \\
				&\leq C\big(\big|\zeta + \frac{\pi}{2}\big| + |\mu - 1| + |\theta|^3 + \lambda\big)\lambda^2 + c\|g\|_\cE^2,
			\end{aligned}
		\end{equation}
		hence
		\begin{equation}
			\begin{aligned}
				&C_1\theta\lambda^2 + \frac 12\la \vD^2 E(\eee^{i\zeta}W_\mu + \eee^{i\theta}W_\lambda)g, g\ra \\
				&\leq E(u) - 2E(W) + C\big(\big|\zeta + \frac{\pi}{2}\big| + |\mu - 1| + |\theta|^3 + \lambda\big)\lambda^2 + c\|g\|_\cE^2.
			\end{aligned}
		\end{equation}
		Choosing $c$ small enough and invoking Lemma~\ref{lem:coer-L-two} finishes the proof of \eqref{eq:coer-conclusion}.
	\end{proof}
	
	\subsection{Useful lemma}
	The following continuation criterion is used in the topological argument.
	\begin{lemma}
		\label{cor:leaves-compact}
		There exists a constant $\eta > 0$ such that the following holds. Let $u: [t_0, T_+) \to \cE$
		be a maximal solution of \eqref{eq:nls} with $T_+ < +\infty$. Then for any compact set $K \subset \cE$
		there exists $\tau < T_+$ such that $\dist(u(t), K) > \eta$ for $t \in [\tau, T_+)$.
	\end{lemma}
	\begin{proof}
		See \cite{Jacek:nls}, Corollary A.3.
	\end{proof}
	The following weak stability lemma is used to prove Theorem \ref{thm:deux-bulles}.
	\begin{lemma}
		\label{cor:weak-cont}
		There exists a constant $\eta > 0$ such that the following holds.
		Let $K \subset \cE$ be a compact set and let $u_n: [T_1, T_2] \to \cE$ be a sequence of solutions of \eqref{eq:nls} such that
		\begin{equation}
			\dist(u_n(t), K) \leq \eta,\qquad \text{for all }n \in \bN\text{ and }t \in [T_1, T_2].
		\end{equation}
		Suppose that $u_n(T_1) \wto u_0 \in \cE$. Then the solution $u(t)$ of \eqref{eq:nls} with the initial condition $u(T_1) = u_0$
		is defined for $t \in [T_1, T_2]$ and
		\begin{equation}
			u_n(t) \wto u(t),\qquad \text{for all }t \in [T_1, T_2].
		\end{equation}
	\end{lemma}
	\begin{proof}
		See \cite{Jacek:nls}, Corollary A.4.
	\end{proof}

	\section{Modulation Analysis}
	\label{sec:mod}
	\subsection{Bounds on the modulation parameters}
    \label{subsec:bounds}
	We study solutions of the following form:
	\begin{equation}
		\label{eq:decompose}
		u(t) = \eee^{i\zeta(t)}W_{\mu(t)} + \eee^{i\theta(t)}W_{\lambda(t)} + g(t),
	\end{equation}
	with
	\begin{equation}
		\label{eq:param-rough}
		|\mu(t) - 1| \ll 1,\quad \big|\zeta(t)+\frac{\pi}{2}\big| \ll 1,\quad \lambda(t) \ll 1,\quad |\theta(t)| \ll 1\quad\text{and}\quad \|g\|_\cE \ll 1,
	\end{equation}
    where $g$ satisfies the orthogonality conditions corresponding to \eqref{eq:orth}.

	Differentiating \eqref{eq:decompose} in time we obtain
	\begin{equation}
		\label{eq:dtu}
		\partial_t u = \zeta'i\eee^{i\zeta}W_\mu - \frac{\mu'}{\mu}\eee^{i\zeta}\Lambda W_\mu + \theta'i\eee^{i\theta}W_\lambda - \frac{\lambda'}{\lambda}\Lambda W_\lambda + \partial_t g.
	\end{equation}
	On the other hand, using $\Delta(W_\mu) + f(W_\mu) = \Delta(W_\lambda) + f(W_\lambda) = 0$ we get
	\begin{equation}
		\label{eq:rhsu}
		i\Delta u + if(u) = i\Delta g + i\big(f(\eee^{i\zeta}W_\mu + \eee^{i\theta}W_\lambda + g) - f(\eee^{i\zeta}W_\mu) - f(\eee^{i\theta}W_\lambda)\big),
	\end{equation}
	hence \eqref{eq:nls} yields
	\begin{equation}
		\label{eq:dtg}
		\begin{aligned}
			\partial_t g &= i\Delta g + i\big(f(\eee^{i\zeta}W_{\mu} + \eee^{i\theta}W_{\lambda} + g) - f(\eee^{i\zeta}W_\mu) - f(\eee^{i\theta}W_\lambda)\big)  \\
			&\quad -\zeta' i\eee^{i\zeta}W_\mu + \frac{\mu'}{\mu}\eee^{i\zeta}\Lambda W_\mu - \theta' i\eee^{i\theta}W_\lambda + \frac{\lambda'}{\lambda}\eee^{i\theta}\Lambda W_\lambda.
		\end{aligned}
	\end{equation}
	Since we work with non-classical solutions, it is worth pointing out
	that the equation above should be understood as a notational simplification.
	Any computation involving $g(t)$ could be rewritten in terms of $u(t)$
	and the modulation parameters $\zeta$, $\mu$, $\theta$, $\lambda$.
	Most of the time, we only use the fact that \eqref{eq:dtg} holds in the weak sense,
	but later we will also need to compute the time derivative of a quadratic form in $g(t)$,
	and interpreting \eqref{eq:dtg} in the weak sense would be insufficient.

    We impose the orthogonality conditions \eqref{eq:orth}. By standard arguments based on the implicit function theorem, these conditions uniquely determine the modulation parameters; see, for example, \cite{DM09,DM10,LLTX25}.

    \subsubsection{Basic modulation}
    \label{Basic modulation}
    First, we give rough estimates of the modulation parameters.
	\begin{lemma}
		\label{lem:basicmod}
		Let $c > 0$ be an arbitrarily small constant. Let $T_0 < 0$ with $|T_0|$ large enough (depending on $c$)
		and $T < T_1 \leq T_0$. Suppose that for $T \leq t \leq T_1$ there holds
		\begin{align}
			\big|\zeta(t) + \frac{\pi}{2}\big| &\leq \eee^{-\frac98|t|}, \label{eq:bootstrap-zeta} \\
			|\mu(t) - 1| &\leq  \eee^{-\frac98|t|}, \label{eq:bootstrap-mu} \\
			|\theta(t)| &\leq \eee^{-\frac12|t|}, \label{eq:bootstrap-theta} \\
			\big|\lambda(t) -\eee^{-|t|}\big| &\leq  \eee^{-\frac54|t|}, \label{eq:bootstrap-lambda} \\
			\|g\|_\cE &\leq  \eee^{-\frac54|t|}. \label{eq:bootstrap-g}
		\end{align}
        Then, for $T \leq t \leq T_1$,
	\begin{align}
			\label{eq:mod-zeta}
			|\zeta'(t)| &\leq c\,\eee^{-\frac98|t|}, \\
			\label{eq:mod-mu}
			|\mu'(t)| &\leq c\,\eee^{-\frac98|t|}, \\
			\label{eq:mod-l}
			|\lambda'(t) | &\lesssim R^{-2}\,\eee^{-\frac14|t|}, \\
			\label{eq:mod-th}
			|\theta'(t)| & \lesssim R^{-2}\,\eee^{\frac34|t|},
		\end{align}
		where $R$ is defined in \eqref{eq:orth} and sufficiently large.
	\end{lemma}
    \begin{proof}
  We use the usual method of differentiating the orthogonality conditions in time, which will yield a linear system of the form:
  \begin{equation}
    \label{eq:mod-system}
    \begin{pmatrix}
      M_{11} & M_{12} & M_{13} & M_{14} \\ M_{21} & M_{22} & M_{23} & M_{24} \\ M_{31} & M_{32} & M_{33} & M_{34} \\ M_{41} & M_{42} & M_{43} & M_{44}
    \end{pmatrix} \begin{pmatrix}\mu^2 \zeta' \\ \mu \mu' \\ \lambda^2 \theta' \\ \lambda\lambda'\end{pmatrix} = \begin{pmatrix}B_1 \\ B_2 \\ B_3 \\ B_4 \end{pmatrix}.
  \end{equation}
  Here, the coefficients $M_{ij}$ and $B_i$ depend on $g$, $\zeta$, $\mu$, $\theta$ and $\lambda$. We will now compute all these coefficients and prove appropriate bounds.

  \textbf{First row.}
  Differentiating $\la i\phi_{R\mu(t)}\eee^{i\zeta(t)}\Lambda W_{\mu(t)}, g(t)\ra = 0$ and using \eqref{eq:dtg},  we obtain
    \begin{align}
      0 &= \dd t \la i\phi_{R\mu(t)}\eee^{i\zeta(t)}\Lambda W_{\mu(t)}, g(t)\ra \\
      &=-\frac{\mu'}{R\mu^2}\la ix\cdot \nabla\phi_{R\mu}\eee^{i\zeta}\Lambda W_\mu, g\ra-\zeta'\la \phi_{R\mu}\eee^{i\zeta}\Lambda W_\mu, g\ra - \frac{\mu'}{\mu}\la i\phi_{R\mu}\eee^{i\zeta}\Lambda\Lambda W_\mu, g\ra + \la i\phi_{R\mu}\eee^{i\zeta}\Lambda W_\mu, \partial_t g\ra \\
      &= \zeta'\big({-}\la i\phi_{R\mu}\eee^{i\zeta}\Lambda W_\mu, i\eee^{i\zeta} W_\mu\ra - \la \phi_{R\mu}\eee^{i\zeta}\Lambda W_\mu, g\ra\big) \\
      &+ \frac{\mu'}{\mu}\big(\la i\phi_{R\mu}\eee^{i\zeta}\Lambda W_\mu, \eee^{i\zeta}\Lambda W_\mu\ra - \la i\phi_{R\mu}\eee^{i\zeta}\Lambda\Lambda W_\mu, g\ra-\frac{1}{R\mu}\la ix\cdot \nabla\phi_{R\mu}i\eee^{i\zeta}\Lambda W_\mu, g\ra\big) \\
      &+ \theta'\la i\phi_{R\mu}\eee^{i\zeta}\Lambda W_\mu, -i\eee^{i\theta} W_\lambda\ra + \frac{\lambda'}{\lambda}\la i\phi_{R\mu}\eee^{i\zeta}\Lambda W_\mu, \eee^{i\theta}\Lambda W_\lambda\ra \\
      &+ \big\la i\phi_{R\mu}\eee^{i\zeta}\Lambda W_\mu, i\Delta g + i\big(f(\eee^{i\zeta}W_{\mu} + \eee^{i\theta}W_{\lambda} + g) - f(\eee^{i\zeta}W_\mu) - f(\eee^{i\theta}W_\lambda)\big) \big\ra.
  \end{align}
  Note that $\la -\Lambda W_\mu, W_\mu\ra = \|W_\mu\|_{L^2}^2 = \mu^2 \|W\|_{L^2}^2$, hence we get
  \begin{align}
    M_{11} &= \mu^{-2}\big({-}\la i\phi_{R\mu}\eee^{i\zeta}\Lambda W_\mu, i\eee^{i\zeta} W_\mu\ra - \la \phi_{R\mu}\eee^{i\zeta}\Lambda W_\mu, g\ra\big)\\
    &=\mu^{-2}\big({-}\la i\eee^{i\zeta}\Lambda W_\mu, i\eee^{i\zeta} W_\mu\ra + \la i(1-\phi_{R\mu})\eee^{i\zeta}\Lambda W_\mu, i\eee^{i\zeta} W_\mu\ra - \la \phi_{R\mu}\eee^{i\zeta}\Lambda W_\mu, g\ra\big)\\
    &=\|W\|_{L^2}^2 + O(R^{-2}+R^{\epsilon}\|g\|_\cE)=\|W\|_{L^2}^2 + O(R^{-2}), \\
    M_{12} &= \mu^{-2}\big(\la i\phi_{R\mu}\eee^{i\zeta}\Lambda W_\mu, \eee^{i\zeta}\Lambda W_\mu\ra - \la i\phi_{R\mu}\eee^{i\zeta}\Lambda\Lambda W_\mu, g\ra - \frac{1}{R\mu}\la ix\cdot \nabla\phi_{R\mu}i\eee^{i\zeta}\Lambda W_\mu, g\ra\big)\\
    &= O(R^{-2}+R^{\epsilon}\|g\|_\cE)=O(R^{-2}) , \\
    M_{13} &= \lambda^{-2}\la i\phi_{R\mu}\eee^{i\zeta}\Lambda W_\mu, -i\eee^{i\theta} W_\lambda\ra = O\Big( \big(\frac{R\mu}{\lambda}\big)^{\epsilon}\Big)\lesssim O(\eee^{\frac19|t|}), \\
    M_{14} &= \lambda^{-2}\la i\phi_{R\mu}\eee^{i\zeta}\Lambda W_\mu, \eee^{i\theta}\Lambda W_\lambda\ra = O\Big( \big(\frac{R\mu}{\lambda}\big)^{\epsilon}\Big)\lesssim O(\eee^{\frac19|t|}),
  \end{align}
 where $\epsilon\ll 1$.
 
  Let us consider the term
  \begin{equation}
    \label{eq:B1}
    B_1 = -\big\la i\phi_{R\mu}\eee^{i\zeta}\Lambda W_\mu, i\Delta g + i\big(f(\eee^{i\zeta}W_{\mu} + \eee^{i\theta}W_{\lambda} + g) - f(\eee^{i\zeta}W_\mu) - f(\eee^{i\theta}W_\lambda)\big) \big\ra.
  \end{equation}
  
  First, we consider $\la i\phi_{R\mu}\eee^{i\zeta}\Lambda W_\mu, i\Delta g\ra$,
  \begin{equation}
  \label{eq:B1-estim-1}
      \la i\phi_{R\mu}\eee^{i\zeta}\Lambda W_\mu, i\Delta g\ra\lesssim \|\phi_{R\mu}\eee^{i\zeta}\Lambda W_\mu\|_{\dot{H}^1}\|g\|_{\dot{H}^1}\lesssim\|g\|_{\dot{H}^1} \lesssim \eee^{-\frac54|t|}.
  \end{equation}

  Next, we show that
  \begin{equation}
    \label{eq:B1-estim-2}
    \big|\big\la \phi_{R\mu}\eee^{i\zeta}\Lambda W_\mu, f(\eee^{i\zeta}W_{\mu} + \eee^{i\theta}W_{\lambda} + g) - f(\eee^{i\zeta}W_\mu) - f(\eee^{i\theta}W_\lambda)\big\ra\big| \lesssim \|g\|_\cE\lesssim \eee^{-\frac54|t|}.
  \end{equation}
  Note that \eqref{eq:pointwise-2} with $z_1 = \eee^{i\zeta}W_\mu + \eee^{i\theta}W_\lambda$ and $z_2 = g$ yields
  \begin{equation}
    |f(\eee^{i\zeta}W_\mu + \eee^{i\theta}W_\lambda + g) - f(\eee^{i\zeta}W_\mu + \eee^{i\theta}W_\lambda) | \lesssim |\eee^{i\zeta}W_\mu + \eee^{i\theta}W_\lambda||g|+|g|^2.
  \end{equation}
  Using the fact that $|\Lambda W| \lesssim W$ and the H\"older inequality we obtain
  \begin{equation}
     \label{eq:B1-estim-21}
    \big|\big\la \phi_{R\mu}\eee^{i\zeta}\Lambda W_\mu, f(\eee^{i\zeta}W_{\mu} + \eee^{i\theta}W_{\lambda} + g) - f(\eee^{i\zeta}W_\mu+\eee^{i\theta}W_\lambda) \big\ra\big| \lesssim \|g\|_\cE.
  \end{equation}
  Next we prove that
  \begin{equation}
    \label{eq:B1-estim-22}
    \big|\big\la \phi_{R\mu}\eee^{i\zeta}\Lambda W_\mu, f(\eee^{i\zeta}W_\mu + \eee^{i\theta}W_\lambda) -
    f(\eee^{i\zeta}W_\mu) - f(\eee^{i\theta}W_\lambda)\big\ra\big| \lesssim \lambda^2\ll \|g\|_\cE.
  \end{equation}
  Using \eqref{eq:pointwise-2} we get
  \begin{equation}
    |f(\eee^{i\zeta}W_\mu + \eee^{i\theta}W_\lambda) -
    f(\eee^{i\zeta}W_\mu) - f(\eee^{i\theta}W_\lambda)| \lesssim W_\mu W_\lambda .
  \end{equation}
   In the region $|x| \leq 1$ we write
  \begin{equation}
    \|W_\lambda\|_{L^1(|x| \leq 1)} = \lambda^4\|W\|_{L^1(|x| \leq \lambda^{-1})}
    \lesssim \lambda^4\Big(\int_0^1 \ud x + \int_1^{\lambda^{-1}}r^{-4}r^{5}\ud r\Big) \lesssim \lambda^2.
  \end{equation}
  As for $|x| \geq 1$, we notice that $\|W_\lambda\|_{L^\infty(|x| \geq 1)} \lesssim \lambda^2$
  and $|\Lambda W_\mu|\cdot|W_\mu|$ is bounded in $L^1$.
  Combining with \eqref{eq:B1-estim-21} and \eqref{eq:B1-estim-22}, we can get \eqref{eq:B1-estim-2}.


Taking the sum of \eqref{eq:B1-estim-1} and \eqref{eq:B1-estim-2}  we obtain
  \begin{equation}
    \label{eq:B1-estim}
    |B_1| \lesssim \eee^{-\frac54|t|}.
  \end{equation}

  \textbf{Second row.}
  Differentiating $\la -\phi_{R\mu(t)}\eee^{i\zeta(t)}W_{\mu(t)}, g(t)\ra = 0$, we obtain
  \begin{equation}
    \begin{aligned}
      0 &= \dd t \la -\phi_{R\mu(t)}\eee^{i\zeta(t)}W_{\mu(t)}, g(t)\ra \\
      &=\frac{\mu'}{R\mu^2}\la x\cdot\nabla\phi_{R\mu}\eee^{i\zeta}W_\mu, g\ra-\zeta'\la i\phi_{R\mu}\eee^{i\zeta}W_\mu, g\ra + \frac{\mu'}{\mu}\la \phi_{R\mu}\eee^{i\zeta}\Lambda W_\mu, g\ra - \la \phi_{R\mu}\eee^{i\zeta}W_\mu, \partial_t g\ra \\
      &= \zeta'\big(\la \phi_{R\mu}\eee^{i\zeta}W_\mu, i\eee^{i\zeta} W_\mu\ra - \la i\phi_{R\mu}\eee^{i\zeta}\Lambda W_\mu, g\ra\big)\\
      &+ \frac{\mu'}{\mu}\big({-}\la \phi_{R\mu}\eee^{i\zeta}W_\mu, \eee^{i\zeta}\Lambda W_\mu\ra + \la \phi_{R\mu}\eee^{i\zeta}\Lambda W_\mu, g\ra + \frac{1}{R\mu}\la x\cdot\nabla\phi_{R\mu}\eee^{i\zeta}W_\mu, g\ra\big) \\
      &+ \theta'\la \phi_{R\mu}\eee^{i\zeta}W_\mu, i\eee^{i\theta} W_\lambda\ra + \frac{\lambda'}{\lambda}\la {-}\phi_{R\mu}\eee^{i\zeta}W_\mu, \eee^{i\theta}\Lambda W_\lambda\ra \\
      &- \big\la \phi_{R\mu}\eee^{i\zeta} W_\mu, i\Delta g + i\big(f(\eee^{i\zeta}W_{\mu} + \eee^{i\theta}W_{\lambda} + g) - f(\eee^{i\zeta}W_\mu) - f(\eee^{i\theta}W_\lambda)\big) \big\ra,
  \end{aligned}
  \end{equation}
  which yields
  \begin{align}
    M_{21} &= \mu^{-2}\big(\la \phi_{R\mu}\eee^{i\zeta} W_\mu, i\eee^{i\zeta} W_\mu\ra - \la i\phi_{R\mu}\eee^{i\zeta} W_\mu, g\ra\big) = O(R^{-2}+R^{\epsilon}\|g\|_\cE)=O(R^{-2}), \\
    M_{22} &= \mu^{-2}\big({-}\la \phi_{R\mu}\eee^{i\zeta}W_\mu, \eee^{i\zeta}\Lambda W_\mu\ra + \la \phi_{R\mu}\eee^{i\zeta}\Lambda W_\mu, g\ra+ \frac{1}{R\mu}\la x\cdot\nabla\phi_{R\mu}\eee^{i\zeta}W_\mu, g\ra\big) \\
    &= \|W\|_{L^2}^2 + O(R^{-2}+R^{\epsilon}\|g\|_\cE)=\|W\|_{L^2}^2 + O(R^{-2}), \\
    M_{23} &= \lambda^{-2}\la \phi_{R\mu}\eee^{i\zeta} W_\mu, i\eee^{i\theta} W_\lambda\ra = O\Big( \big(\frac{R\mu}{\lambda}\big)^{\epsilon}\Big)=O(\eee^{\frac19|t|}), \\
    M_{24} &= \lambda^{-2}\la -\phi_{R\mu}\eee^{i\zeta} W_\mu, \eee^{i\theta}\Lambda W_\lambda\ra = O\Big( \big(\frac{R\mu}{\lambda}\big)^{\epsilon}\Big)=O(\eee^{\frac19|t|}),
  \end{align}
  where $\epsilon\ll 1$.

  Consider now the term
  \begin{equation}
    \label{eq:B2}
    B_2 = \big\la \phi_{R\mu}\eee^{i\zeta} W_\mu, i\Delta g + i\big(f(\eee^{i\zeta}W_{\mu} + \eee^{i\theta}W_{\lambda} + g) - f(\eee^{i\zeta}W_\mu) - f(\eee^{i\theta}W_\lambda)\big) \big\ra .
  \end{equation}
  Similarly to \eqref{eq:B1-estim}, we obtain
\begin{equation}
  \label{eq:B2-estim}
  |B_2| \lesssim \eee^{-\frac54|t|}.
\end{equation}

  \textbf{Third row.}
  Differentiating $\la i\phi_{R\lambda(t)}\eee^{i\theta(t)}\Lambda W_{\lambda(t)}, g(t)\ra = 0$, we obtain
  \begin{equation}
    \begin{aligned}
      0 &= \dd t \la i\phi_{R\lambda(t)}\eee^{i\theta(t)}\Lambda W_{\lambda(t)}, g(t)\ra\\
      &=-\frac{\lambda'}{R\lambda^2}\la ix\cdot \nabla\phi_{R\lambda}i\eee^{i\theta}\Lambda W_\lambda, g\ra  -\theta'\la \phi_{R\lambda}\eee^{i\theta}\Lambda W_\lambda, g\ra - \frac{\lambda'}{\lambda}\la i\phi_{R\lambda}\eee^{i\theta}\Lambda\Lambda W_\lambda, g\ra + \la i\phi_{R\lambda}\eee^{i\theta}\Lambda W_\lambda, \partial_t g\ra \\
      &= \zeta'\la i\phi_{R\lambda}\eee^{i\theta}\Lambda W_\lambda, -i\eee^{i\zeta} W_\mu\ra + \frac{\mu'}{\mu}\la i\phi_{R\lambda}\eee^{i\theta}\Lambda W_\lambda, \eee^{i\zeta}\Lambda W_\mu\ra \\
      &+ \theta'\big(\la i\phi_{R\lambda}\eee^{i\theta}\Lambda W_\lambda, {-}i\eee^{i\theta} W_\lambda\ra -\la \phi_{R\lambda}\eee^{i\theta}\Lambda W_\lambda, g\ra\big)\\
      &+ \frac{\lambda'}{\lambda}\big(\la i\phi_{R\lambda}\eee^{i\theta}\Lambda W_\lambda, \eee^{i\theta}\Lambda W_\lambda\ra - \la i\phi_{R\lambda}\eee^{i\theta}\Lambda\Lambda W_\lambda, g\ra - \frac{1}{R\lambda}\la ix\cdot \phi_{R\lambda}i\eee^{i\theta}\Lambda W_\lambda, g\ra\big) \\
      &+ \big\la i\phi_{R\lambda}\eee^{i\theta}\Lambda W_\lambda, i\Delta g + i\big(f(\eee^{i\zeta}W_{\mu} + \eee^{i\theta}W_{\lambda} + g) - f(\eee^{i\zeta}W_\mu) - f(\eee^{i\theta}W_\lambda)\big) \big\ra,
  \end{aligned}
  \end{equation}
  which yields
  \begin{align}
    M_{31} =& \mu^{-2}\la i\phi_{R\lambda}\eee^{i\theta}\Lambda W_\lambda, -i\eee^{i\zeta} W_\mu\ra =O(R^\epsilon\lambda^2)=O(R^\epsilon\eee^{-2|t|}),  \\
    M_{32} =& \mu^{-2}\la i\phi_{R\lambda}\eee^{i\theta}\Lambda W_\lambda, \eee^{i\zeta}\Lambda W_\mu\ra = O(R^\epsilon\lambda^2) =O(R^\epsilon\eee^{-2|t|}), \\
    M_{33} =& \lambda^{-2}\big(\la i\phi_{R\lambda}\eee^{i\theta}\Lambda W_\lambda, {-}i\eee^{i\theta} W_\lambda\ra -\la \phi_{R\lambda}\eee^{i\theta}\Lambda W_\lambda, g\ra - \frac{1}{R\lambda}\la ix\cdot \phi_{R\lambda}i\eee^{i\theta}\Lambda W_\lambda, g\ra\big)\\
    =&\lambda^{-2}\big(\la i\eee^{i\theta}\Lambda W_\lambda, {-}i\eee^{i\theta} W_\lambda\ra + \la i(\phi_{R\lambda}-1)\eee^{i\theta}\Lambda W_\lambda, {-}i\eee^{i\theta} W_\lambda\ra \\
    &-\la \phi_{R\lambda}\eee^{i\theta}\Lambda W_\lambda, g\ra- \frac{1}{R\lambda}\la ix\cdot \phi_{R\lambda}i\eee^{i\theta}\Lambda W_\lambda, g\ra\big)\\
    =& \|W\|_{L^2}^2 + O(R^{-2}+R^{\epsilon}\|g\|_\cE)=\|W\|_{L^2}^2 + O(R^{-2}) , \\
    M_{34} =& \lambda^{-2}\big(\la i\phi_{R\lambda}\eee^{i\theta}\Lambda W_\lambda, \eee^{i\theta}\Lambda W_\lambda\ra - \la i\phi_{R\lambda}\eee^{i\theta}\Lambda\Lambda W_\lambda, g\ra\big) = O(R^{-2}+R^{\epsilon}\|g\|_\cE)=O(R^{-2}),
  \end{align}
  where $\epsilon\ll 1$.

  Let us consider the term
  \begin{equation}
			\begin{aligned}
				B_3 =&  -\big\la i\phi_{R\lambda}\eee^{i\theta}\Lambda W_\lambda, i\Delta g + i\big(f(\eee^{i\zeta}W_{\mu} + \eee^{i\theta}W_{\lambda} + g) - f(\eee^{i\zeta}W_\mu) - f(\eee^{i\theta}W_\lambda)\big) \big\ra \\
				=& -\big\la i\eee^{i\theta}\Lambda W_\lambda, i\Delta g + i\big(f(\eee^{i\zeta}W_{\mu} + \eee^{i\theta}W_{\lambda} + g) - f(\eee^{i\zeta}W_\mu) - f(\eee^{i\theta}W_\lambda)\big) \big\ra \\
				 &-\big\la (\phi_{R\lambda}-1)i\eee^{i\theta}\Lambda W_\lambda, i\Delta g \big\ra \\
				 &-\big\la (\phi_{R\lambda}-1)i\eee^{i\theta}\Lambda W_\lambda, i\big(f(\eee^{i\zeta}W_{\mu} + \eee^{i\theta}W_{\lambda} + g) - f(\eee^{i\zeta}W_\mu) - f(\eee^{i\theta}W_\lambda)\big) \big\ra \\
				=&:I+II+III.
			\end{aligned}
		\end{equation}
  Using $i\eee^{i\theta}\Lambda W_\lambda\in \ker Z_{\theta, \lambda}^*$, we can rewrite the first term as follows:
		\begin{equation}\label{I}
			I=-\big\la \eee^{i\theta}\Lambda W_\lambda, f(\eee^{i\zeta}W_{\mu} + \eee^{i\theta}W_{\lambda} + g) - f(\eee^{i\zeta}W_\mu) - f(\eee^{i\theta}W_\lambda) - f'(\eee^{i\theta}W_\lambda)g \big\ra.
		\end{equation}
Let us denote
\[
			K := -\big\la \eee^{i\theta}\Lambda W_\lambda, f(\eee^{i\zeta}W_{\mu} + \eee^{i\theta}W_{\lambda} + g) - f(\eee^{i\zeta}W_\mu + \eee^{i\theta}W_\lambda) - f'(\eee^{i\zeta}W_\mu + \eee^{i\theta}W_\lambda)g \big\ra. 
			\]
         Note that \eqref{eq:pointwise-1} with $z_1 = \eee^{i\zeta}W_\mu + \eee^{i\theta}W_\lambda$ and $z_2 = g$ yields
         $$|f(\eee^{i\zeta}W_{\mu} + \eee^{i\theta}W_{\lambda} + g) - f(\eee^{i\zeta}W_\mu + \eee^{i\theta}W_\lambda) - f'(\eee^{i\zeta}W_\mu + \eee^{i\theta}W_\lambda)g|\lesssim |g|^2.$$
         Using the fact that $|\Lambda W|\lesssim W$ and the H\"older inequality we obtain
         \begin{equation}\label{eq:est-K}
             |K|\lesssim \|g\|_{\cE}^2.
         \end{equation}
        Therefore, we have
		\begin{align*}
			I-K=&-\big\la \eee^{i\theta}\Lambda W_\lambda, f(\eee^{i\zeta}W_{\mu} + \eee^{i\theta}W_{\lambda} ) - f(\eee^{i\zeta}W_\mu) - f(\eee^{i\theta}W_\lambda)  \big\ra\\
			&-\big\la \eee^{i\theta}\Lambda W_\lambda,\big( f'(\eee^{i\zeta}W_{\mu} + \eee^{i\theta}W_{\lambda} ) - f'(\eee^{i\theta}W_\lambda)\big)g \big\ra\\
			=:&(i)+(ii).
		\end{align*}
		The term $(i)$ can be treated as (3-20) in \cite{Jacek:nls}, and we can obtain
		\begin{align}\label{I1}
			&|\big\la \eee^{i\theta}\Lambda W_\lambda, f(\eee^{i\zeta}W_{\mu} + \eee^{i\theta}W_{\lambda} ) - f(\eee^{i\zeta}W_\mu) - f(\eee^{i\theta}W_\lambda)  \big\ra-C_2\theta\lambda^2|\\
			&\lesssim\lambda^2(|\theta|^3+|\zeta+\frac{\pi}{2}|+|\mu-1|)+\lambda^3\lesssim \lambda^3.
		\end{align}
		From \eqref{eq:pointwise-5}, we have
		$$|f'(\eee^{i\zeta}W_{\mu} + \eee^{i\theta}W_{\lambda} ) - f'(\eee^{i\theta}W_\lambda)|\lesssim W_\mu.$$
		Hence,
		\begin{equation}\label{I2}
			|(ii)|\lesssim \|W_\mu W_\lambda\|_{L^{\frac32}}\|g\|_{L^3}
			\lesssim\lambda^{2-\epsilon}\|g\|_\cE,
		\end{equation}
		where $\epsilon\ll1$.
		Thus, we get 
		\begin{equation}\label{B31}
			|I-K+C_2\theta\lambda^2|\lesssim\lambda^{2-\epsilon}\|g\|_\cE+\lambda^3\lesssim \lambda^3,
		\end{equation}
		where $\epsilon\ll1$. Thus, using \eqref{eq:est-K}, we obtain
        \begin{equation}\label{eq:B3-estim-1}
            |I| \lesssim \lambda^2|\theta|+\lambda^3+\|g\|_\cE^2\lesssim\eee^{-\frac52|t|}.
        \end{equation}
     Next, we consider the second term.
  \begin{equation}\label{eq:B3-estim-2}
      |II|=|\big\la (\phi_{R\lambda}-1)i\eee^{i\theta}\Lambda W_\lambda, i\Delta g \big\ra| \lesssim \|(\phi_{R\lambda}-1)i\eee^{i\theta}\Lambda W_\lambda\|_{\dot{H}^1}\|g\|_{\dot{H}^1}\lesssim R^{-2}\eee^{-\frac54|t|}.
  \end{equation}
  We can rewrite the third term as follows:
		\begin{align}\label{B33}
			|III|\lesssim&|\big\la i(\phi_{R\lambda}-1)\eee^{i\theta}\Lambda W_\lambda, i\big(f(\eee^{i\zeta}W_{\mu} + \eee^{i\theta}W_{\lambda} + g) -f(\eee^{i\zeta}W_{\mu} + \eee^{i\theta}W_{\lambda} ) \big) \big\ra\\
			 &+|\big\la i(\phi_{R\lambda}-1)\eee^{i\theta}\Lambda W_\lambda,i\big(f(\eee^{i\zeta}W_{\mu} + \eee^{i\theta}W_{\lambda} )-f(\eee^{i\zeta}W_\mu) - f(\eee^{i\theta}W_\lambda)\big) \big\ra|\\
			 =:&(i)+(ii).
		\end{align}
		From \eqref{eq:pointwise-2}, we have
		$$|f(\eee^{i\zeta}W_{\mu} + \eee^{i\theta}W_{\lambda} + g) -f(\eee^{i\zeta}W_{\mu} + \eee^{i\theta}W_{\lambda} )|\lesssim|\eee^{i\zeta}W_{\mu} + \eee^{i\theta}W_{\lambda} ||g|+|g|^2.$$
		Hence,
		\begin{align}\label{III1}
			|(i)|&\lesssim \|(\phi_{R\lambda}-1)\eee^{i\theta}\Lambda W_\lambda\|_{L^3}\big(\|\eee^{i\zeta}W_{\mu} + \eee^{i\theta}W_{\lambda}\|_{L^3}\|g\|_{L^3}+\|g\|_{L^3}^2\big)\\
			&\lesssim\|W\|_{L^3(|x|\geq R)}\big(\|g\|_{\cE}+\|g\|_{\cE}^2\big)\\
			&\lesssim R^{-2}\|g\|_{\cE}.
		\end{align}
		From \eqref{eq:pointwise-2}, we have
		$$|f(\eee^{i\zeta}W_{\mu} + \eee^{i\theta}W_{\lambda} )-f(\eee^{i\zeta}W_\mu) - f(\eee^{i\theta}W_\lambda)|\lesssim W_{\mu} W_{\lambda} .$$
		Combining $|x|\geq R\gg 1$, we obtain
		\begin{align}\label{III2}
			|(ii)|&\lesssim\|W_\lambda\|_{L^\infty(|x|\geq R)}\|W_\mu^2\|_{L^1(|x|\geq R)}\\
			&\lesssim\lambda^2R^{-4}\cdot R^{-2}=\lambda^2R^{-6}.
		\end{align}
		From \eqref{III1} and \eqref{III2}, we get 
		\begin{equation}\label{eq:B3-estim-3}
			|III|\lesssim R^{-2}\|g\|_{\cE}+\lambda^2R^{-6}\lesssim R^{-2}\|g\|_{\cE}\lesssim R^{-2}\|g\|_{\cE}\lesssim R^{-2}\eee^{-\frac54|t|}.
		\end{equation}
  
  From \eqref{eq:B3-estim-1}, \eqref{eq:B3-estim-2} and \eqref{eq:B3-estim-3}, we get 
\begin{equation}
  \label{eq:B3-estim}
  \big|B_3\big| \lesssim R^{-2}\eee^{-\frac54|t|}.
\end{equation}
  \textbf{Forth row.}
  Differentiating $\la {-}\phi_{R\lambda(t)}\eee^{i\theta(t)}W_{\lambda(t)}, g(t)\ra = 0$ we obtain
  \begin{equation}
    \begin{aligned}
      0 &= \dd t \la {-}\phi_{R\lambda(t)}\eee^{i\theta(t)}W_{\lambda(t)}, g(t)\ra\ra\\
      &=\frac{\lambda'}{R\lambda^2}\la x\cdot \nabla\phi_{R\lambda}\eee^{i\theta}W_\lambda, g\ra -\theta'\la i\phi_{R\lambda}\eee^{i\theta}W_\lambda, g\ra + \frac{\lambda'}{\lambda}\la \phi_{R\lambda}\eee^{i\theta}\Lambda W_\lambda, g\ra - \la \phi_{R\lambda}\eee^{i\theta}W_\lambda, \partial_t g\ra \\
      &= \zeta'\la i\phi_{R\lambda}\eee^{i\theta}W_\lambda, i\eee^{i\zeta} W_\mu\ra - \frac{\mu'}{\mu}\la \phi_{R\lambda}\eee^{i\theta}W_\lambda, \eee^{i\zeta}\Lambda W_\mu\ra \\
      &+ \theta'\big(\la \phi_{R\lambda}\eee^{i\theta}W_\lambda, i\eee^{i\theta} W_\lambda\ra-\la i\phi_{R\lambda}\eee^{i\theta}W_\lambda, g\ra\big)\\
      &+ \frac{\lambda'}{\lambda}\big(\la {-}\phi_{R\lambda}\eee^{i\theta}W_\lambda, \eee^{i\theta}\Lambda W_\lambda\ra +\la \phi_{R\lambda}\eee^{i\theta}\Lambda W_\lambda, g\ra + \frac{1}{R\lambda}\la x\cdot \nabla\phi_{R\lambda}\eee^{i\theta}W_\lambda, g\ra\big) \\
      &- \big\la \phi_{R\lambda}\eee^{i\theta} W_\lambda, i\Delta g + i\big(f(\eee^{i\zeta}W_{\mu} + \eee^{i\theta}W_{\lambda} + g) - f(\eee^{i\zeta}W_\mu) - f(\eee^{i\theta}W_\lambda)\big) \big\ra,
  \end{aligned}
  \end{equation}
  which yields
  \begin{align}
    M_{41} &= \mu^{-2}\la i\phi_{R\lambda}\eee^{i\theta}W_\lambda, i\eee^{i\zeta} W_\mu\ra= O(R^\epsilon\lambda^2) = O(R^\epsilon\eee^{-2|t|}), \\
    M_{42} &= \mu^{-2}\la \phi_{R\lambda}\eee^{i\theta}W_\lambda, \eee^{i\zeta}\Lambda W_\mu\ra =O(R^\epsilon\lambda^2) =  O(R^\epsilon\eee^{-2|t|}), \\
    M_{43} &= \lambda^{-2}\big(\la \phi_{R\lambda}\eee^{i\theta}W_\lambda, i\eee^{i\theta} W_\lambda\ra-\la i\phi_{R\lambda}\eee^{i\theta}W_\lambda, g\ra\big) = O(R^{-2}+R^{\epsilon}\|g\|_\cE)= O(R^{-2}), \\
    M_{44} &= \lambda^{-2}\big(\la {-}\phi_{R\lambda}\eee^{i\theta}W_\lambda, \eee^{i\theta}\Lambda W_\lambda\ra +\la \phi_{R\lambda}\eee^{i\theta}\Lambda W_\lambda, g\ra+ \frac{1}{R\lambda}\la x\cdot \nabla\phi_{R\lambda}\eee^{i\theta}W_\lambda, g\ra\big) \\
    &= \|W\|_{L^2}^2 + O(R^{-2}+R^{\epsilon}\|g\|_\cE)=\|W\|_{L^2}^2 + O(R^{-2}),
  \end{align}
  where $\epsilon\ll 1$.

  Let us consider the term
  \begin{align}
    B_4 =& \big\la \phi_{R\lambda}\eee^{i\theta} W_\lambda, i\Delta g + i\big(f(\eee^{i\zeta}W_{\mu} + \eee^{i\theta}W_{\lambda} + g) - f(\eee^{i\zeta}W_\mu) - f(\eee^{i\theta}W_\lambda)\big) \big\ra\\
    =&\big\la \eee^{i\theta} W_\lambda, i\Delta g + i\big(f(\eee^{i\zeta}W_{\mu} + \eee^{i\theta}W_{\lambda} + g) - f(\eee^{i\zeta}W_\mu) - f(\eee^{i\theta}W_\lambda)\big) \big\ra\\
	&+\big\la (\phi_{R\lambda}-1)\eee^{i\theta} W_\lambda, i\Delta g \big\ra\\
	&+\big\la (\phi_{R\lambda}-1)\eee^{i\theta} W_\lambda, i\big(f(\eee^{i\zeta}W_{\mu} + \eee^{i\theta}W_{\lambda} + g) - f(\eee^{i\zeta}W_\mu) - f(\eee^{i\theta}W_\lambda)\big)\big\ra\\
	=:&I+II+III.
  \end{align}
  Arguing as for $B_4$ in \cite{Jacek:nls}, we obtain
  \begin{equation}\label{B41}
	|I-C_1\lambda^2|  \lesssim\lambda^{2-\epsilon}\|g\|_\cE+\lambda^2(|\theta|^2+|\zeta+\frac{\pi}{2}|+|\mu-1|)+\lambda^3 +\|g\|_\cE^2\lesssim\|g\|_\cE^2,
  \end{equation}
   where $C_1$ is given by \eqref{eq:explicit-1} and $\epsilon\ll1$. Thus,
   \begin{equation}\label{eq:B4-estim-1}
            |I| \lesssim \lambda^2+\|g\|_\cE^2\lesssim\eee^{-2|t|}.
    \end{equation}
   The other three terms $II, III$ can be treated similarly to those in the third row, and we can obtain
  \begin{equation}\label{eq:B4-estim}
		|B_4 |  \lesssim\|g\|_\cE^2+R^{-2}\eee^{-\frac54|t|}\lesssim R^{-2}\eee^{-\frac54|t|}.
  \end{equation}
 
  \textbf{Conclusion}
  From the bounds on the coefficients $M_{ij}$ obtained above it follows that the matrix $(M_{ij})$ can write
  \begin{equation}
    \label{eq:mod-system-approx}
    \begin{gathered}
    \begin{pmatrix}
      M_{11} & M_{12} & M_{13} & M_{14} \\ M_{21} & M_{22} & M_{23} & M_{24} \\ M_{31} & M_{32} & M_{33} & M_{34} \\ M_{41} & M_{42} & M_{43} & M_{44}
    \end{pmatrix}= \\
    \begin{pmatrix}
      \|W\|_{L^2}^2 + O(R^{-2}) & O(R^{-2}) & O(\eee^{\frac19|t|}) & O(\eee^{\frac19|t|}) \\ O(R^{-2}) & \|W\|_{L^2}^2 + O(R^{-2}) & O(\eee^{\frac19|t|}) & O(\eee^{\frac19|t|}) \\ O(R^\epsilon\eee^{-2|t|}) & O(R^\epsilon\eee^{-2|t|}) & \|W\|_{L^2}^2 + O(R^{-2}) & O(R^{-2}) \\ O(R^\epsilon\eee^{-2|t|}) & O(R^\epsilon\eee^{-2|t|}) & O(R^{-2}) & \|W\|_{L^2}^2 + O(R^{-2})
    \end{pmatrix},
  \end{gathered}
  \end{equation}
  where $\epsilon\ll1$. Let $(m_{jk}) = (M_{jk})^{-1}$. It is easy to see that the Cramer's rule implies that $(m_{jk})$ is also of the form given in \eqref{eq:mod-system-approx},
  with $\|W\|_{L^2}^{-2}$ instead of $\|W\|_{L^2}^2$ for the diagonal terms.
 
  Resuming \eqref{eq:B1-estim}, \eqref{eq:B2-estim}, \eqref{eq:B3-estim} and \eqref{eq:B4-estim}, we have
  \begin{equation}
    \label{eq:B-estim}
    |B_1| + |B_2|  \lesssim \eee^{-\frac54|t|}, \; |B_3| + |B_4|\lesssim R^{-2}\eee^{-\frac54|t|}.
  \end{equation}
  This and the form of the matrix $(m_{jk})$ directly imply $|\zeta'| + |\mu'| \lesssim \eee^{-\frac{41}{36}|t|}$, hence \eqref{eq:mod-zeta} and \eqref{eq:mod-mu}. Note that the coefficients in the third and the fourth row of the matrix $(m_{jk})$ let us gain an additional factor $R^{-2}$. Thus, we obtain
  $\big|\lambda\lambda' \big| \lesssim R^{-2} \eee^{-\frac54|t|}$ and $\big|\lambda^2\theta' \big| \lesssim R^{-2} \eee^{-\frac54|t|}$, which imply \eqref{eq:mod-l} and \eqref{eq:mod-th}.
\end{proof}

    \subsubsection{Refined modulation}
    \label{Refined modulation}
   The rough estimates for the modulation parameters obtained in Lemma~\ref{lem:basicmod} are not sufficient to close the bootstrap argument. We therefore derive refined estimates with improved decay rates. The proof follows the same strategy as that of Lemma~\ref{lem:basicmod}, the main difference being the introduction of the modified modulation parameters. Throughout this proof, we take $R_0(t)=t$ and assume that $R_0(t)$ is sufficiently large. 

	\begin{lemma}
		\label{lem:mod}
		Under the assumptions of Lemma \ref{lem:basicmod}, we define
        \begin{align}\label{tilde}
        &\tilde{\zeta}(t)=\zeta(t)-\frac{1}{ \|W\|_{L^2}^2}L_1(t),\; \tilde{\mu}(t)=\mu(t)-\frac{1}{ \|W\|_{L^2}^2}L_2(t), \\
        &\tilde{\theta}(t)=\theta(t)-\frac{1}{ \|W\|_{L^2}^2}L_3(t),\; \tilde{\lambda}(t)=\lambda(t)-\frac{1}{ \|W\|_{L^2}^2}L_4(t),
        \end{align}
        where $L_i\;(i=1,\ldots,4)$ are defined 
        \begin{align}
            L_1(t)&:=-\int_{t}^{\infty} \frac{1}{\mu(s)^2}\dd s \la i\phi_{R_0(s)}\eee^{i\zeta(s)}\Lambda W_{\mu(s)}, g(s)\ra  ds,\\
            L_2(t)&:=-\int_{t}^{\infty} \frac{1}{\mu(s)}\dd s \la -\phi_{R_0(s)}\eee^{i\zeta(s)} W_{\mu(s)}, g(s)\ra  ds,\\
            L_3(t)&:=-\int_{t}^{\infty} \frac{1}{\tilde{\lambda}(s)^2}\dd s \la i\phi_{R_0(s)}\eee^{i\theta(s)}\Lambda W_{\lambda(s)}, g(s)\ra  ds,\\
            L_4(t)&:=-\int_{t}^{\infty} \frac{1}{\lambda(s)}\dd s \la -\phi_{R_0(s)}\eee^{i\theta(s)}W_{\lambda(s)}, g(s)\ra  ds.
        \end{align}
		Then 
		\begin{align}
			\big|\tilde{\zeta}(t) - \zeta(t)\big| &\lesssim \eee^{-\frac98|t|}, \label{eq:diff-zeta} \\
			|\tilde{\mu}(t) - \mu(t)| &\lesssim  \eee^{-\frac98|t|}, \label{eq:diff-mu} \\
			|\tilde{\lambda}(t) -\lambda(t)| &\lesssim  \eee^{-\frac{11}{8}|t|}, \label{eq:diff-lambda} \\
            |\tilde{\theta}(t) - \theta(t)| &\lesssim \eee^{-\frac98|t|}, \label{eq:diff-theta} 
		\end{align}
		and
		\begin{align}
			\label{eq:mod-tildezeta}
			|\tilde{\zeta}'(t)| &\leq c\,\eee^{-\frac54|t|}, \\
			\label{eq:mod-tildemu}
			|\tilde{\mu}'(t)| &\leq c\,\eee^{-\frac54|t|}, \\
			\label{eq:mod-tildel}
			\Big|\tilde{\lambda}'(t) -\tilde{\lambda}(t)\Big| &\leq c\,\eee^{-\frac54|t|}, \\
			\label{eq:mod-tildeth}
			\Big|\tilde{\theta}'(t)+2\theta(t)-\frac{K(t)}{\lambda(t)^2\|W\|_{L^2}^2}\Big| &\leq c\,\eee^{-\frac34|t|},
		\end{align}
        for $T \leq t \leq T_1$, where $c$ decays polynomially  as $t\to -\infty$ and can be arbitrarily small when $|T_0|$ large enough, and
		\begin{equation}
			\label{eq:K-def}
			K := -\big\la \eee^{i\theta}\Lambda W_\lambda, f(\eee^{i\zeta}W_{\mu} + \eee^{i\theta}W_{\lambda} + g) - f(\eee^{i\zeta}W_\mu + \eee^{i\theta}W_\lambda) - f'(\eee^{i\zeta}W_\mu + \eee^{i\theta}W_\lambda)g \big\ra.
		\end{equation}
	\end{lemma}
    
        \begin{remark}
        We could use any polynomially growing function instead of $R_0(t) = t$.
        \end{remark}

        {\begin{remark}
        We note that there is no circular dependence in the definitions introduced above: once $L_4(t)$ is defined, we can define $\wt \lambda(t)$, which allows to define $L_3(t)$, and finally $\wt \theta(t)$.
        \end{remark}
	\begin{proof}[Proof of Lemma~\ref{lem:mod}]
        First, we prove \eqref{eq:diff-zeta}- \eqref{eq:diff-lambda}. Let 
        \[\tilde{L}_1(t):=\frac{1}{\mu(t)^2}\la i\phi_{R_0(t)}\eee^{i\zeta(t)}\Lambda W_{\mu(t)}, g(t)\ra, \]
        we have
		\begin{equation}\label{L1tilde}
			|\tilde{L}_1|
			\lesssim\big|\frac{1}{\mu^2}\big|\|\phi_{R_0}\eee^{i\zeta}\Lambda W_\mu\|_{L^{\frac32}}\|g\|_{L^3}
			\lesssim\|\phi_{\frac{R_0}{\mu}}\eee^{i\zeta}\Lambda W\|_{L^{\frac32}}\|g\|_{\cE}
			\lesssim \big(\frac{R_0}{\mu}\big)^\epsilon\|g\|_{\cE}\lesssim \eee^{-\frac98|t|},
		\end{equation}
		where $\epsilon\ll 1$. From the definion of $L_1$ and using \eqref{eq:mod-mu}, we have
		\begin{align}
			|L_1'-\tilde{L}'_1|&\lesssim\Big|\frac{\mu'}{\mu^3}\la i\phi_{R_0}\eee^{i\zeta}\Lambda W_\mu, g\ra\Big|
			\lesssim\big|\frac{\mu'}{\mu^3}\big|\|\phi_{R_0}\eee^{i\zeta}\Lambda W_\mu\|_{L^{\frac32}}\|g\|_{L^3}\\
			&\lesssim\big|\frac{\mu'}{\mu}\big|\|\phi_{\frac{R_0}{\mu}}\eee^{i\zeta}\Lambda W\|_{L^{\frac32}}\|g\|_{\cE}
			\lesssim \big|\frac{\mu'}{\mu}\big|\big(\frac{R_0}{\mu}\big)^\epsilon\|g\|_{\cE}\\
			&\lesssim \eee^{-\frac98|t|}.
		\end{align}
		Hence, 
		\begin{equation}\label{L1error}
			|L_1-\tilde{L}_1|\lesssim\eee^{-\frac98|t|}.
		\end{equation}
		By the definition of $\tilde{\zeta}$, \eqref{L1tilde}, and \eqref{L1error}, we have
		\begin{equation*}
			|\tilde{\zeta} - \zeta|\lesssim|L_1|\lesssim|L_1-\tilde{L}_1|+|\tilde{L}_1|\lesssim\eee^{-\frac98|t|}.
		\end{equation*}
		Thus, \eqref{eq:diff-zeta} holds. Similarly, we can obtain \eqref{eq:diff-mu}. As for \eqref{eq:diff-lambda}, we define \[ \tilde{L}_4(t):=\frac{1}{\lambda(t)}\la i\phi_{R_0(t)}\eee^{i\theta(t)} W_{\lambda(t)}, g(t)\ra, \]
        we have
		\begin{equation}\label{L4tilde}
			|\tilde{L}_4|
			\lesssim\big|\frac{1}{\lambda}\big|\|\phi_{R_0}\eee^{i\theta}W_\lambda\|_{L^{\frac32}}\|g\|_{L^3}
			\lesssim |\lambda| \|\phi_{\frac{R_0}{\lambda}}\eee^{i\theta}W\|_{L^{\frac32}}\|g\|_{\cE}
			\lesssim |\lambda|\big(\frac{R_0}{\lambda}\big)^\epsilon\|g\|_{\cE}\lesssim \eee^{-\frac{17}{8}|t|},
		\end{equation}
		where $\epsilon\ll 1$. From the definion of $L_4$ and using \eqref{eq:mod-l}, we have
		\begin{align}
			|L_4'-\tilde{L}_4'|&\lesssim\Big|\frac{\lambda'}{\lambda^2}\la i\phi_{R_0(t)}\eee^{i\theta} W_\lambda, g\ra\Big|
			\lesssim\big|\frac{\lambda'}{\lambda^2}\big|\|\phi_{R_0}\eee^{i\theta}W_\lambda\|_{L^{\frac32}}\|g\|_{L^3}\\
			&\lesssim\big|\lambda'\big|\|\phi_{\frac{R_0}{\lambda}}\eee^{i\theta}W\|_{L^{\frac32}}\|g\|_{\cE}
			\lesssim \big|\lambda'\big|\big(\frac{R_0}{\lambda}\big)^\epsilon\|g\|_{\cE}\\
			&\lesssim \eee^{-\frac{11}{8}|t|}.
		\end{align}
		Hence, 
		\begin{equation}\label{L4error}
			|L_4-\tilde{L}_4|\lesssim\eee^{-\frac{11}{8}|t|}.
		\end{equation}
		By the definition of $\tilde{\lambda}$, \eqref{L4tilde}, and \eqref{L4error}, we have
		\begin{equation*}
			|\tilde{\lambda} - \lambda|\lesssim|L_4|\lesssim|L_4-\tilde{L}_4|+|\tilde{L}_4|\lesssim\eee^{-\frac{11}{8}|t|}.
		\end{equation*}
		Thus, \eqref{eq:diff-lambda} holds.
        The proof of \eqref{eq:diff-theta} differs slightly from those of the previous three estimates because it requires the estimate for $\tilde{\lambda}'$. Therefore, we postpone it until the end of the proof, after establishing the estimate for $\tilde{\lambda}'$.
        
		Next, we prove \eqref{eq:mod-tildezeta}- \eqref{eq:mod-tildeth}. The proof follows the same strategy as that of Lemma~\ref{lem:basicmod}. Taking into account the correction terms in the modified modulation parameters, we obtain a linear system of the form:
		\begin{equation}
			\label{eq:mod-system-nonli}
			\begin{pmatrix}
				\tilde{M}_{11} & \tilde{M}_{12} & \tilde{M}_{13} & \tilde{M}_{14} \\ \tilde{M}_{21} & \tilde{M}_{22} & \tilde{M}_{23} & \tilde{M}_{24} \\ \tilde{M}_{31} & \tilde{M}_{32} & \tilde{M}_{33} & \tilde{M}_{34} \\ \tilde{M}_{41} & \tilde{M}_{42} & \tilde{M}_{43} & \tilde{M}_{44}
			\end{pmatrix} \begin{pmatrix}\mu^2 \zeta' \\ \mu \mu' \\ \tilde{\lambda}^2 \theta' \\ \lambda\lambda'\end{pmatrix} = \begin{pmatrix}\tilde{B}_1+\mu^2L_1' \\ \tilde{B}_2+\mu L_2' \\ \tilde{B}_3+\tilde{\lambda}^2L_3' \\ \tilde{B}_4+\lambda L_4' \end{pmatrix},
		\end{equation}
		where the coefficients $\tilde{M}_{ij}$ and $\tilde{B}_i$ depend on $g$, $\zeta$, $\mu$, $\theta$ and $\lambda$. We will now compute all these coefficients and prove appropriate bounds.

		\textbf{First row.}
		Differentiating $\la i\phi_{R_0(t)}\eee^{i\zeta(t)}\Lambda W_{\mu(t)}, g(t)\ra $, we obtain
		\begin{equation}\label{L1}
			\begin{aligned}
				\mu^2(t)L_1(t)':=&  \dd t \la i\phi_{R_0(t)}\eee^{i\zeta(t)}\Lambda W_{\mu(t)}, g(t)\ra \\
				=& -\frac{R_0'}{R_0^2}\la ix\cdot\nabla \phi_{R_0}i\eee^{i\zeta}\Lambda W_\mu, g\ra-\zeta'\la \phi_{R_0}\eee^{i\zeta}\Lambda W_\mu, g\ra - \frac{\mu'}{\mu}\la i\phi_{R_0}\eee^{i\zeta}\Lambda\Lambda W_\mu, g\ra + \la i\phi_{R_0}\eee^{i\zeta}\Lambda W_\mu, \partial_t g\ra \\
				=& \zeta'\big({-}\la \phi_{R_0}i\eee^{i\zeta}\Lambda W_\mu, i\eee^{i\zeta} W_\mu\ra - \la \phi_{R_0}\eee^{i\zeta}\Lambda W_\mu, g\ra\big) \\
				&+ \frac{\mu'}{\mu}\big(\la i\phi_{R_0}\eee^{i\zeta}\Lambda W_\mu, \eee^{i\zeta}\Lambda W_\mu\ra - \la i\phi_{R_0}\eee^{i\zeta}\Lambda\Lambda W_\mu, g\ra\big) \\
				&+ \theta'\la i\phi_{R_0}\eee^{i\zeta}\Lambda W_\mu, -i\eee^{i\theta} W_\lambda\ra 
				+ \frac{\lambda'}{\lambda}\la i\phi_{R_0}\eee^{i\zeta}\Lambda W_\mu, \eee^{i\theta}\Lambda W_\lambda\ra \\
				&+ \big\la i\phi_{R_0}\eee^{i\zeta}\Lambda W_\mu, i\Delta g + i\big(f(\eee^{i\zeta}W_{\mu} + \eee^{i\theta}W_{\lambda} + g) - f(\eee^{i\zeta}W_\mu) - f(\eee^{i\theta}W_\lambda)\big) \big\ra\\
				&-\frac{R_0'}{R_0^2}\la ix\cdot\nabla \phi_{R_0}i\eee^{i\zeta}\Lambda W_\mu, g\ra.
			\end{aligned}
		\end{equation}
		Note that $\la -\Lambda W_\mu, W_\mu\ra = \|W_\mu\|_{L^2}^2 = \mu^2 \|W\|_{L^2}^2$. Hence, we get
		\begin{align}
			\tilde{M}_{11} &= \mu^{-2}\big({-}\la \phi_{R_0}i\eee^{i\zeta}\Lambda W_\mu, i\eee^{i\zeta} W_\mu\ra - \la \phi_{R_0}\eee^{i\zeta}\Lambda W_\mu, g\ra\big)\\
			&= \mu^{-2}\big({-}\la i\eee^{i\zeta}\Lambda W_\mu, i\eee^{i\zeta} W_\mu\ra-\la (\phi_{R_0}-1)i\eee^{i\zeta}\Lambda W_\mu, i\eee^{i\zeta} W_\mu\ra - \la \phi_{R_0}\eee^{i\zeta}\Lambda W_\mu, g\ra\big)\\
			& = \|W\|_{L^2}^2 + O\Big(\mu^2R_0^{-2}+\big(\frac{R_0}{\mu}\big)^{\epsilon}\|g\|_\cE\Big),\\
			\tilde{M}_{12} &= \mu^{-2}\big(\la i\phi_{R_0}\eee^{i\zeta}\Lambda W_\mu, \eee^{i\zeta}\Lambda W_\mu\ra - \la i\phi_{R_0}\eee^{i\zeta}\Lambda\Lambda W_\mu, g\ra\big) \\
			&=O\Big(\mu^2R^{-2}_0+\big(\frac{R_0}{\mu}\big)^{\epsilon}\|g\|_\cE\Big),\\
			\tilde{M}_{13} &= \lambda^{-2}\la i\phi_{R_0}\eee^{i\zeta}\Lambda W_\mu, -i\eee^{i\theta} W_\lambda\ra = O\Big((\frac{R_0}{\lambda})^{\epsilon}\Big), \\
			\tilde{M}_{14} &= \lambda^{-2}\la i\phi_{R_0}\eee^{i\zeta}\Lambda W_\mu, \eee^{i\theta}\Lambda W_\lambda\ra = O\Big((\frac{R_0}{\lambda})^{\epsilon}\Big),
		\end{align}
		where $\epsilon\ll 1$.
		
		Let us consider the term
		\begin{align}
			\label{eq:tildeB1}
			\tilde{B}_1 = &-\big\la i\phi_{R_0}\eee^{i\zeta}\Lambda W_\mu, i\Delta g + i\big(f(\eee^{i\zeta}W_{\mu} + \eee^{i\theta}W_{\lambda} + g) - f(\eee^{i\zeta}W_\mu) - f(\eee^{i\theta}W_\lambda)\big) \big\ra+\frac{R_0'}{R_0^2}\la ix\cdot\nabla \phi_{R_0}i\eee^{i\zeta}\Lambda W_\mu, g\ra,\\
			=&-\big\la i\eee^{i\zeta}\Lambda W_\mu+i(\phi_{R_0}-1)\eee^{i\zeta}\Lambda W_\mu, i\Delta g + i\big(f(\eee^{i\zeta}W_{\mu} + \eee^{i\theta}W_{\lambda} + g) - f(\eee^{i\zeta}W_\mu) - f(\eee^{i\theta}W_\lambda)\big) \big\ra\\
			=&-\big\la i\eee^{i\zeta}\Lambda W_\mu, i\Delta g + i\big(f(\eee^{i\zeta}W_{\mu} + \eee^{i\theta}W_{\lambda} + g) - f(\eee^{i\zeta}W_\mu) - f(\eee^{i\theta}W_\lambda)\big) \big\ra -\big\la i(\phi_{R_0}-1)\eee^{i\zeta}\Lambda W_\mu, i\Delta g  \big\ra\\
			&-\big\la i(\phi_{R_0}-1)\eee^{i\zeta}\Lambda W_\mu i\big(f(\eee^{i\zeta}W_{\mu} + \eee^{i\theta}W_{\lambda} + g) - f(\eee^{i\zeta}W_\mu) - f(\eee^{i\theta}W_\lambda)\big) \big\ra
			+\frac{R_0'}{R_0^2}\la ix\cdot\nabla \phi_{R_0}i\eee^{i\zeta}\Lambda W_\mu, g\ra\\
			=:&I+II+III+IV.
		\end{align}
		For the first term $I$, it can be treated similarly to $B_1$ that in \cite[Lemma 3.1]{Jacek:nls}, and we can get 
		\begin{equation}\label{B11}
			|I|\lesssim \|g\|_\cE^2+\lambda^2\lesssim\lambda^2.
		\end{equation}
		Similarly to \eqref{eq:B3-estim-2}, we estimate the second term as
		\begin{align}\label{B12}
			|II|\lesssim R_0^{-2}\|g\|_{\cE}.
		\end{align}
		The third term can be treated similarly to \eqref{B33} and we can get
		\begin{equation}\label{B13}
		    |III|\lesssim R_0^{-2}\|g\|_{\cE}.
		\end{equation}
		For the fourth term, using the defination of $\phi$, we have
     \begin{align}\label{B14}
			|IV| &\lesssim \Big|\frac{R_0'}{R_0}\Big|\|\nabla \phi_{R_0}\Lambda W_\mu\|_{L^{\frac32}}\|g\|_{L^3}\\
			&\lesssim \Big|\frac{R_0'}{R_0}\Big|\| W\|_{L^{\frac32}(R_0\leq|x|\leq2R_0)}\|g\|_{\cE}\\
			&\lesssim \Big|\frac{R_0'}{R_0}\Big|R_0^\epsilon\|g\|_{\cE}\lesssim c(t)\|g\|_{\cE},
    \end{align}
    where $\epsilon\ll 1$, and $c(t)=|t|^{-\frac12}$, which is sufficiently small as $t \to -\infty$.

		Taking the sum of \eqref{B11}, \eqref{B12}, \eqref{B13} and \eqref{B14} we obtain
		\begin{equation}
			\label{eq:tildeB1-estim}
			|\tilde{B}_1| \lesssim c_1(t)\|g\|_{\cE}\ll \eee^{-\frac54|t|},
		\end{equation}
		where $c_1(t)=|t|^{-\frac12}$, which is sufficiently small as $t \to -\infty$.
		
		\textbf{Second row.}
		Differentiating $\la -\phi_{R_0(t)}\eee^{i\zeta(t)}W_{\mu(t)}, g(t)\ra$,  we obtain
		\begin{equation}\label{L2}
			\begin{aligned}
				\mu(t) L_2(t)' :=& \dd t \la -\phi_{R_0(t)}\eee^{i\zeta(t)}W_{\mu(t)}, g(t)\ra \\
				=& \frac{R_0'}{R_0^2}\la x\cdot\nabla\phi_{R_0}\eee^{i\zeta}W_\mu, g\ra -\zeta'\la i\phi_{R_0}\eee^{i\zeta}W_\mu, g\ra + \frac{\mu'}{\mu}\la \phi_{R_0}\eee^{i\zeta}\Lambda W_\mu, g\ra - \la \phi_{R_0}\eee^{i\zeta}W_\mu, \partial_t g\ra \\
				=& \zeta'\big(\la \phi_{R_0}\eee^{i\zeta}W_\mu, i\eee^{i\zeta} W_\mu\ra - \la i\phi_{R_0}\eee^{i\zeta} W_\mu, g\ra\big) \\
				&+ \frac{\mu'}{\mu}\big({-}\la \phi_{R_0}\eee^{i\zeta}W_\mu, \eee^{i\zeta}\Lambda W_\mu\ra + \la \phi_{R_0}\eee^{i\zeta}\Lambda W_\mu, g\ra\big) \\
				&+ \theta'\la \phi_{R_0}\eee^{i\zeta}W_\mu, i\eee^{i\theta} W_\lambda\ra + \frac{\lambda'}{\lambda}\la {-}\phi_{R_0}\eee^{i\zeta}W_\mu, \eee^{i\theta}\Lambda W_\lambda\ra \\
				&- \big\la \phi_{R_0}\eee^{i\zeta} W_\mu, i\Delta g + i\big(f(\eee^{i\zeta}W_{\mu} + \eee^{i\theta}W_{\lambda} + g) - f(\eee^{i\zeta}W_\mu) - f(\eee^{i\theta}W_\lambda)\big) \big\ra\\
				&+ \frac{R_0'}{R_0^2}\la x\cdot\nabla\phi_{R_0}\eee^{i\zeta}W_\mu, g\ra,
			\end{aligned}
		\end{equation}
		which yields
		\begin{align}
			\tilde{M}_{21} &= \mu^{-2}\big(\la \phi_{R_0}\eee^{i\zeta} W_\mu, i\eee^{i\zeta} W_\mu\ra - \la i\phi_{R_0}\eee^{i\zeta} W_\mu, g\ra\big) = O\Big(\Big(\frac{R_0}{\mu}\Big)^\epsilon\|g\|_\cE+\mu^2R_0^{-2}\Big), \\
			\tilde{M}_{22} &= \mu^{-2}\big({-}\la \phi_{R_0}\eee^{i\zeta}W_\mu, \eee^{i\zeta}\Lambda W_\mu\ra + \la \phi_{R_0}\eee^{i\zeta}\Lambda W_\mu, g\ra\big) = \|W\|_{L^2}^2 + O\Big(\Big(\frac{R_0}{\mu}\Big)^\epsilon\|g\|_\cE+\mu^2R_0^{-2}\Big), \\
			\tilde{M}_{23} &= \lambda^{-2}\la \phi_{R_0}\eee^{i\zeta} W_\mu, i\eee^{i\theta} W_\lambda\ra = O\Big(\Big(\frac{R_0}{\lambda}\Big)^\epsilon\Big), \\
			\tilde{M}_{24} &= \lambda^{-2}\la -\phi_{R_0}\eee^{i\zeta} W_\mu, \eee^{i\theta}\Lambda W_\lambda\ra = O\Big(\Big(\frac{R_0}{\lambda}\Big)^\epsilon\Big).
		\end{align}
		
		Consider now the term
		\begin{equation}
			\label{eq:tildeB2}
			\begin{aligned}
				\tilde{B}_2 = &\big\la \phi_{R_0}\eee^{i\zeta} W_\mu, i\Delta g + i\big(f(\eee^{i\zeta}W_{\mu} + \eee^{i\theta}W_{\lambda} + g) - f(\eee^{i\zeta}W_\mu) - f(\eee^{i\theta}W_\lambda)\big) \big\ra\\
				&- \frac{R_0'}{R_0^2}\la x\cdot\nabla\phi_{R_0}\eee^{i\zeta}W_\mu, g\ra.
			\end{aligned}
		\end{equation}
		Using a similar argument as the proof of \eqref{eq:tildeB1-estim} yields
		\begin{equation}
			\label{eq:tildeB2-estim}
			|\tilde{B}_2| \lesssim c_2(t)\|g\|_{\cE}\ll  \eee^{-\frac54|t|},
		\end{equation}
		where $c_2(t)=|t|^{-\frac12}$, which is sufficiently small as $t \to -\infty$.
		
		\textbf{Third row.}
		Differentiating $\la i\phi_{R_0(t)}\eee^{i\theta(t)}\Lambda W_{\lambda(t)}, g(t)\ra $, we obtain
		\begin{equation}\label{L3}
			\begin{aligned}
				\tilde{\lambda}(t)^2L_3(t)' :=&\dd t \la i\phi_{R_0(t)}\eee^{i\theta(t)}\Lambda W_{\lambda(t)}, g(t)\ra\\
				 = &-\frac{R_0'}{R_0^2}\la x\cdot \nabla\phi_{R_0}i\eee^{i\theta}\Lambda W_\lambda, g\ra-\theta'\la \phi_{R_0}\eee^{i\theta}\Lambda W_\lambda, g\ra - \frac{\lambda'}{\lambda}\la i\phi_{R_0}\eee^{i\theta}\Lambda\Lambda W_\lambda, g\ra + \la i\phi_{R_0}\eee^{i\theta}\Lambda W_\lambda, \partial_t g\ra \\
				=& \zeta'\la i\phi_{R_0}\eee^{i\theta}\Lambda W_\lambda, -i\eee^{i\zeta} W_\mu\ra + \frac{\mu'}{\mu}\la i\phi_{R_0}\eee^{i\theta}\Lambda W_\lambda, \eee^{i\zeta}\Lambda W_\mu\ra \\
				&+ \theta'\big(\la i\phi_{R_0}\eee^{i\theta}\Lambda W_\lambda, {-}i\eee^{i\theta} W_\lambda\ra -\la \phi_{R_0}\eee^{i\theta}\Lambda W_\lambda, g\ra\big) \\
				&+ \frac{\lambda'}{\lambda}\big(\la i\phi_{R_0}\eee^{i\theta}\Lambda W_\lambda, \eee^{i\theta}\Lambda W_\lambda\ra - \la i\phi_{R_0}\eee^{i\theta}\Lambda\Lambda W_\lambda, g\ra\big) \\
				&+ \big\la i\phi_{R_0}\eee^{i\theta}\Lambda W_\lambda, i\Delta g + i\big(f(\eee^{i\zeta}W_{\mu} + \eee^{i\theta}W_{\lambda} + g) - f(\eee^{i\zeta}W_\mu) - f(\eee^{i\theta}W_\lambda)\big) \big\ra\\
				&-\frac{R_0'}{R_0^2}\la x\cdot \nabla\phi_{R_0}i\eee^{i\theta}\Lambda W_\lambda, g\ra,
			\end{aligned}
		\end{equation}
		which yields
		\begin{align}
			\tilde{M}_{31} &= \mu^{-2}\la i\phi_{R_0}\eee^{i\theta}\Lambda W_\lambda, -i\eee^{i\zeta} W_\mu\ra =O(R_0^\epsilon\lambda^{2-\epsilon}), \\
			\tilde{M}_{32} &= \mu^{-2}\la i\phi_{R_0}\eee^{i\theta}\Lambda W_\lambda, \eee^{i\zeta}\Lambda W_\mu\ra = O(R_0^\epsilon\lambda^{2-\epsilon}), \\
			\tilde{M}_{33} &= \tilde{\lambda}^{-2}\big(\la i\phi_{R_0}\eee^{i\theta}\Lambda W_\lambda, {-}i\eee^{i\theta} W_\lambda\ra -\la \phi_{R_0}\eee^{i\theta}\Lambda W_\lambda, g\ra\big) = \|W\|_{L^2}^2 + O\Big(\Big(\frac{R_0}{\lambda}\Big)^\epsilon\|g\|_\cE+\lambda^2R_0^{-2}\Big), \\
			\tilde{M}_{34} &= \lambda^{-2}\big(\la i\phi_{R_0}\eee^{i\theta}\Lambda W_\lambda, \eee^{i\theta}\Lambda W_\lambda\ra - \la i\phi_{R_0}\eee^{i\theta}\Lambda\Lambda W_\lambda, g\ra\big) = O\Big(\Big(\frac{R_0}{\lambda}\Big)^\epsilon\|g\|_\cE+\lambda^2R_0^{-2}\Big),
		\end{align}
		where $\epsilon\ll 1$.
		
		Let us consider the term
		\begin{equation}
			\begin{aligned}
				\tilde{B}_3 =& -\big\la i\phi_{R_0}\eee^{i\theta}\Lambda W_\lambda, i\Delta g + i\big(f(\eee^{i\zeta}W_{\mu} + \eee^{i\theta}W_{\lambda} + g) - f(\eee^{i\zeta}W_\mu) - f(\eee^{i\theta}W_\lambda)\big) \big\ra \\
				&-\frac{R_0'}{R_0^2}\la x\cdot \nabla\phi_{R_0}i\eee^{i\theta}\Lambda W_\lambda, g\ra\\
				=& -\big\la i\eee^{i\theta}\Lambda W_\lambda, i\Delta g + i\big(f(\eee^{i\zeta}W_{\mu} + \eee^{i\theta}W_{\lambda} + g) - f(\eee^{i\zeta}W_\mu) - f(\eee^{i\theta}W_\lambda)\big) \big\ra \\
				 &-\big\la (\phi_{R_0}-1)i\eee^{i\theta}\Lambda W_\lambda, i\Delta g \big\ra \\
				 &-\big\la (\phi_{R_0}-1)i\eee^{i\theta}\Lambda W_\lambda, i\big(f(\eee^{i\zeta}W_{\mu} + \eee^{i\theta}W_{\lambda} + g) - f(\eee^{i\zeta}W_\mu) - f(\eee^{i\theta}W_\lambda)\big) \big\ra \\
				&-\frac{R_0'}{R_0^2}\la x\cdot \nabla\phi_{R_0}i\eee^{i\theta}\Lambda W_\lambda, g\ra\\
				=&:I+II+III+IV.
			\end{aligned}
		\end{equation}
		 By \eqref{B31}, we have 
         \begin{equation}\label{tildeB31}
			|I-K+C_2\theta\lambda^2|\lesssim\lambda^{2-\epsilon}\|g\|_\cE+\lambda^3\lesssim \lambda^3,
		\end{equation}
		where $\epsilon\ll1$.
		
		The other three terms $II, III, IV$ can be treated similarly to that in the first row, and we can obtain
			\begin{align}
			    |II|&\lesssim  R_0^{-2}\lambda^2\|g\|_\cE\lesssim \frac{1}{t^2}\lambda^2\|g\|_\cE,\label{B32}\\
			    |III|&\lesssim  R_0^{-2}\lambda^2\|g\|_\cE\lesssim \frac{1}{t^2}\lambda^2\|g\|_\cE,\label{tildeB33}\\
			    |IV|&\lesssim  \lambda^2\Big|\frac{R_0'}{R_0}\Big|\big({\frac{R_0}{\lambda}}\big)^\epsilon\|g\|_\cE\lesssim \frac{1}{t^{\frac12}}\lambda^{2-\epsilon}\|g\|_\cE,\label{B34}
			\end{align}
		where $\epsilon\ll1$.
		
		From \eqref{B31}, \eqref{B32}, \eqref{tildeB33} and \eqref{B34} we infer
		\begin{equation}
			\label{eq:tildeB3-estim}
			\Big|\tilde{B}_3 - K + C_2\theta\lambda^2\Big| \lesssim\lambda^3\lesssim\eee^{-3|t|}.
		\end{equation}
		In particular, using \eqref{eq:est-K}, we have
		\begin{equation}
			\label{eq:tildeB3-estim-rough}
			\big|\tilde{B}_3\big| \lesssim \lambda^2|\theta|+\lambda^3+\|g\|_\cE^2\lesssim\eee^{-\frac52|t|}.
		\end{equation}

		\textbf{Fourth row.}
		Differentiating $\la {-}\phi_{R_0(t)}\eee^{i\theta(t)}W_{\lambda(t)}, g(t)\ra $, we obtain
		\begin{equation}\label{L4}
			\begin{aligned}
				\lambda(t) L_4(t)' :=& \dd t \la -\phi_{R_0(t)}\eee^{i\theta(t)}W_{\lambda(t)}, g(t)\ra \\
				=& \frac{R_0'}{R_0^2}\la x\cdot\nabla\phi_{R_0}\eee^{i\theta}W_\lambda, g\ra-\theta'\la i\phi_{R_0}\eee^{i\theta}W_\lambda, g\ra + \frac{\lambda'}{\lambda}\la \phi_{R_0}\eee^{i\theta}\Lambda W_\lambda, g\ra - \la \phi_{R_0}\eee^{i\theta}W_\lambda, \partial_t g\ra \\
				=& \zeta'\la \phi_{R_0}\eee^{i\theta}W_\lambda, i\eee^{i\zeta} W_\mu\ra - \frac{\mu'}{\mu}\la \phi_{R_0}\eee^{i\theta}W_\lambda, \eee^{i\zeta}\Lambda W_\mu\ra \\
				&+ \theta'\big(\la \phi_{R_0}\eee^{i\theta}W_\lambda, i\eee^{i\theta} W_\lambda\ra-\la i\phi_{R_0}\eee^{i\theta}W_\lambda, g\ra\big)\\
				& + \frac{\lambda'}{\lambda}\big(\la {-}\phi_{R_0}\eee^{i\theta}W_\lambda, \eee^{i\theta}\Lambda W_\lambda\ra +\la \phi_{R_0}\eee^{i\theta}\Lambda W_\lambda, g\ra\big) \\
				&- \big\la \phi_{R_0}\eee^{i\theta} W_\lambda, i\Delta g + i\big(f(\eee^{i\zeta}W_{\mu} + \eee^{i\theta}W_{\lambda} + g) - f(\eee^{i\zeta}W_\mu) - f(\eee^{i\theta}W_\lambda)\big) \big\ra\\
				&+\frac{R_0'}{R_0^2}\la x\cdot\nabla\phi_{R_0}\eee^{i\theta}W_\lambda, g\ra,
			\end{aligned}
		\end{equation}
		which yields
		\begin{align}
			\tilde{M}_{41} &= \mu^{-2}\la i\phi_{R_0}\eee^{i\theta}W_\lambda, i\eee^{i\zeta} W_\mu\ra= O(R_0^\epsilon\lambda^{2-\epsilon}), \\
			\tilde{M}_{42} &= \mu^{-2}\la \phi_{R_0}\eee^{i\theta}W_\lambda, \eee^{i\zeta}\Lambda W_\mu\ra =O(R_0^\epsilon\lambda^{2-\epsilon}), \\
			\tilde{M}_{43} &= \lambda^{-2}\big(\la \phi_{R_0}\eee^{i\theta}W_\lambda, i\eee^{i\theta} W_\lambda\ra-\la i\phi_{R_0}\eee^{i\theta}W_\lambda, g\ra\big) = O\Big(\Big(\frac{R_0}{\lambda}\Big)^\epsilon\|g\|_\cE+\lambda^2R_0^{-2}\Big), \\
			\tilde{M}_{44} &= \lambda^{-2}\big(\la {-}\phi_{R_0}\eee^{i\theta}W_\lambda, \eee^{i\theta}\Lambda W_\lambda\ra +\la \phi_{R_0}\eee^{i\theta}\Lambda W_\lambda, g\ra\big) = \|W\|_{L^2}^2 + O\Big(\Big(\frac{R_0}{\lambda}\Big)^\epsilon\|g\|_\cE+\lambda^2R_0^{-2}\Big).
		\end{align}
		
		Let us consider the term
		\begin{equation}
			\begin{aligned}
				\tilde{B}_4 = &\big\la \phi_{R_0}\eee^{i\theta} W_\lambda, i\Delta g + i\big(f(\eee^{i\zeta}W_{\mu} + \eee^{i\theta}W_{\lambda} + g) - f(\eee^{i\zeta}W_\mu) - f(\eee^{i\theta}W_\lambda)\big) \big\ra\\
				&-\frac{R_0'}{R_0^2}\la x\cdot\nabla\phi_{R_0}\eee^{i\theta}W_\lambda, g\ra\\
				=&\big\la \eee^{i\theta} W_\lambda, i\Delta g + i\big(f(\eee^{i\zeta}W_{\mu} + \eee^{i\theta}W_{\lambda} + g) - f(\eee^{i\zeta}W_\mu) - f(\eee^{i\theta}W_\lambda)\big) \big\ra\\
				&+\big\la (\phi_{R_0}-1)\eee^{i\theta} W_\lambda, i\Delta g \big\ra\\
				&+\big\la (\phi_{R_0}-1)\eee^{i\theta} W_\lambda, i\big(f(\eee^{i\zeta}W_{\mu} + \eee^{i\theta}W_{\lambda} + g) - f(\eee^{i\zeta}W_\mu) - f(\eee^{i\theta}W_\lambda)\big)\big\ra\\
				&-\frac{R_0'}{R_0^2}\la x\cdot\nabla\phi_{R_0}\eee^{i\theta}W_\lambda, g\ra\\
				=:&I+II+III+IV.
			\end{aligned}
		\end{equation}
		By \eqref{B41}, we have
		\begin{equation}
			\label{tildeB41}
			|I-C_1\lambda^2|  \lesssim\lambda^{2-\epsilon}\|g\|_\cE+\lambda^2(|\theta|^2+|\zeta+\frac{\pi}{2}|+|\mu-1|)+\lambda^3 +\|g\|_\cE^2\lesssim\|g\|_\cE^2,
		\end{equation}
		where $C_1$ is given by \eqref{eq:explicit-1} and $\epsilon\ll1$.
		The other three terms $II, III, IV$ can be treated similarly to those in the third row, and we can obtain
		\begin{equation}
			\label{eq:tildeB4-estim}
			\Big|\tilde{B}_4 - C_1\lambda(t)^2\Big|  \lesssim\|g\|_\cE^2\lesssim\eee^{-\frac52|t|}.
		\end{equation}
		 In particular,
		\begin{equation}
			\label{eq:tildeB4-estim-rough}
			|\tilde{B}_4| \lesssim \lambda(t)^2+\|g\|_\cE^2\lesssim\lambda(t)^2\lesssim\eee^{-2|t|}.
		\end{equation}

		\textbf{Conclusion.}
		From the bounds on the coefficients $\tilde{M}_{ij}$ obtained above, it follows that the matrix $(\tilde{M}_{ij})$ can write
		\begin{equation}
			\label{eq:mod-tildesystem-approx}
			\begin{gathered}
				\begin{pmatrix}
					\tilde{M}_{11} & \tilde{M}_{12} & \tilde{M}_{13} & \tilde{M}_{14} 
				\end{pmatrix}= \\
				\begin{pmatrix}
					\|W\|_{L^2}^2 + O\Big(\Big(\frac{R_0}{\mu}\Big)^\epsilon\|g\|_\cE+\mu^2R_0^{-2}\Big)  & O\Big(\Big(\frac{R_0}{\mu}\Big)^\epsilon\|g\|_\cE+\mu^2R_0^{-2}\Big)  & O\Big(\Big(\frac{R_0}{\lambda}\Big)^\epsilon \Big)& O\Big(\Big(\frac{R_0}{\lambda}\Big)^\epsilon \Big)
				\end{pmatrix},\\
				\begin{pmatrix}
					\tilde{M}_{21} & \tilde{M}_{22} & \tilde{M}_{23} & \tilde{M}_{24}
				\end{pmatrix}=\\
				\begin{pmatrix}
					O\Big(\Big(\frac{R_0}{\mu}\Big)^\epsilon\|g\|_\cE+\mu^2R_0^{-2}\Big)  & \|W\|_{L^2}^2 + O\Big(\Big(\frac{R_0}{\mu}\Big)^\epsilon\|g\|_\cE+\mu^2R_0^{-2}\Big)  & O\Big(\Big(\frac{R_0}{\lambda}\Big)^\epsilon \Big) & O\Big(\Big(\frac{R_0}{\lambda}\Big)^\epsilon \Big)
				\end{pmatrix},\\
				\begin{pmatrix}
					\tilde{M}_{31} & \tilde{M}_{32} & \tilde{M}_{33} & \tilde{M}_{34}
				\end{pmatrix}=\\
				\begin{pmatrix}
					O(R_0^\epsilon\lambda^{2-\epsilon}) & O(R_0^\epsilon\lambda^{2-\epsilon}) & \|W\|_{L^2}^2 + O\Big(\Big(\frac{R_0}{\lambda}\Big)^\epsilon\|g\|_\cE+\lambda^2R_0^{-2}\Big)  & O\Big(\Big(\frac{R_0}{\lambda}\Big)^\epsilon\|g\|_\cE+\lambda^2R_0^{-2}\Big)
				\end{pmatrix},\\
				\begin{pmatrix}
					\tilde{M}_{41} & \tilde{M}_{42} & \tilde{M}_{43} & \tilde{M}_{44}
				\end{pmatrix}=\\
				\begin{pmatrix}
					O(R_0^\epsilon\lambda^{2-\epsilon}) & O(R_0^\epsilon\lambda^{2-\epsilon}) & O\Big(\Big(\frac{R_0}{\lambda}\Big)^\epsilon\|g\|_\cE+\lambda^2R_0^{-2}\Big) & \|W\|_{L^2}^2 + O\Big(\Big(\frac{R_0}{\lambda}\Big)^\epsilon\|g\|_\cE+\lambda^2R_0^{-2}\Big)
				\end{pmatrix},
			\end{gathered}
		\end{equation}
		where $\epsilon\ll 1$. In the diagonal terms, the second term $\Big(\frac{R_0}{\mu}\Big)^\epsilon\|g\|_\cE+\mu^2R_0^{-2}$ or $\Big(\frac{R_0}{\lambda}\Big)^\epsilon\|g\|_\cE+\lambda^2R_0^{-2}$ is negligible compared to the first $\|W\|_{L^2}^2$ and can therefore be ignored.
		
		Combining with \eqref{tilde}, we have
		\begin{equation}
			\label{eq:mod-tildesystem}
			\begin{pmatrix}
				\tilde{M}_{11} & \tilde{M}_{12} & \tilde{M}_{13} & \tilde{M}_{14} \\ \tilde{M}_{21} & \tilde{M}_{22} & \tilde{M}_{23} & \tilde{M}_{24} \\ \tilde{M}_{31} & \tilde{M}_{32} & \tilde{M}_{33} & \tilde{M}_{34} \\ \tilde{M}_{41} & \tilde{M}_{42} & \tilde{M}_{43} & \tilde{M}_{44}
			\end{pmatrix} \begin{pmatrix}\mu^2 \tilde{\zeta}' \\ \mu \tilde{\mu}' \\ \tilde{\lambda}^2 \tilde{\theta}' \\ \lambda\tilde{\lambda}'\end{pmatrix} = \begin{pmatrix}\tilde{B}_1 \\ \tilde{B}_2 \\ \tilde{B}_3 \\ \tilde{B}_4 \end{pmatrix}.
		\end{equation}
		Let $(\tilde{m}_{jk}) = (\tilde{M}_{jk})^{-1}$. It is easy to see that Cramer's rule implies that $(\tilde{m}_{jk})$ is also of the form given in \eqref{eq:mod-tildesystem},
		with $\|W\|_{L^2}^{-2}$ instead of $\|W\|_{L^2}^2$ for the diagonal terms.

		Now, we prove \eqref{eq:mod-tildezeta}- \eqref{eq:mod-tildeth}.
		Resuming \eqref{eq:tildeB1-estim}, \eqref{eq:tildeB2-estim}, \eqref{eq:tildeB3-estim-rough} and \eqref{eq:tildeB4-estim-rough}, we have
		\begin{equation}
			\label{eq:tildeB-estim}
			|\tilde{B}_1| + |\tilde{B}_2| \ll \eee^{-\frac54|t|}, |\tilde{B}_3|\lesssim\eee^{-\frac52|t|},  |\tilde{B}_4|\lesssim\eee^{-2|t|}.
		\end{equation}
		This and the form of the matrix $(\tilde{m}_{jk})$ directly imply $|\tilde{\zeta}'| + |\tilde{\mu}'| \ll \eee^{-\frac54|t|}$, hence \eqref{eq:mod-tildezeta} and \eqref{eq:mod-tildemu}.
		Note that the coefficients in the third and fourth rows of the matrix $(\tilde{m}_{jk})$ allow us to gain some additional decay factors.
		We can obtain $\big|\lambda\tilde{\lambda}' - \|W\|_{L^2}^{-2}\tilde{B}_4\big| \lesssim c\eee^{-\frac94|t|}$, which, together with \eqref{eq:tildeB4-estim} and \eqref{eq:explicit-1}, implies that
		\begin{equation}\label{delambda}
			\big|\tilde{\lambda}' - \lambda\big| \leq c\eee^{-\frac54|t|}.
		\end{equation}
		Thus, combining \eqref{delambda} with \eqref{eq:diff-lambda}, we can obtain
		\begin{equation}
			|\tilde{\lambda}'-\tilde{\lambda}|\lesssim	|\tilde{\lambda}'-\lambda| +	|\tilde{\lambda}-\lambda| \leq c\eee^{-\frac54|t|},
		\end{equation}
		i.e., \eqref{eq:mod-tildel} holds.
        Similarly, we can get 
        $\big|\tilde{\lambda}^2\tilde{\theta}' - \|W\|_{L^2}^{-2}\tilde{B}_3\big| \lesssim c\eee^{-3|t|}$, which, together with \eqref{eq:tildeB3-estim},  implies that
        \begin{equation}
			\big|\tilde{\lambda}^2\tilde{\theta}' - \|W\|_{L^2}^{-2}K + \|W\|_{L^2}^{-2}C_2\theta\lambda^2\big| \lesssim \eee^{-3|t|}.
        \end{equation}
        From \eqref{eq:explicit-1} and \eqref{eq:explicit-3}, we know $\|W\|_{L^2}^{-2}C_2=2$. Therefore, $\big|\tilde{\lambda}^2\tilde{\theta}' - \|W\|_{L^2}^{-2}K + 2\theta\lambda^2\big| \lesssim \eee^{-3|t|}.$
        From this, we can get the rough estimate of $\tilde{\theta}'$, which is 
        \begin{equation}
	    		\big|\tilde{\theta}' \big| \lesssim |\tilde{\lambda}^{-2}K|+|\tilde{\lambda}^{-2}\theta\lambda^2|+\tilde{\lambda}^{-2}\eee^{-3|t|}\lesssim \eee^{-\frac12|t|}.
        \end{equation}
        Combining this with \eqref{eq:diff-lambda}, we can get
        \begin{align}\label{detheta}
			\big|\lambda^2\tilde{\theta}' - \|W\|_{L^2}^{-2}K + 2\theta\lambda^2\big|& \leq \big|\lambda^2\tilde{\theta}' - \tilde{\lambda}^2\tilde{\theta}'\big| + \big|\tilde{\lambda}^2\tilde{\theta}' - \|W\|_{L^2}^{-2}K + 2\theta\lambda^2\big| \\ &\lesssim|\lambda+\tilde{\lambda}||\lambda-\tilde{\lambda}|\big|\tilde{\theta}' \big|+ \big|\tilde{\lambda}^2\tilde{\theta}' -\|W\|_{L^2}^{-2} K +2\theta\lambda^2\big|\\
            &\lesssim \eee^{-|t|}\eee^{-\frac{11}{8}|t|}\eee^{-\frac12|t|} + \eee^{-3|t|}\lesssim\eee^{-\frac{23}{8}|t|}.
        \end{align}
        Thus, we get \eqref{eq:mod-tildeth}.

        Finally, we prove \eqref{eq:diff-theta}. Let \[ \tilde{L}_3(t):=\frac{1}{\tilde{\lambda}(t)^2}\la i\phi_{R_0(t)}\eee^{i\theta(t)}\Lambda W_{\lambda(t)}, g(t)\ra,\] 
        we have
		\begin{equation}\label{L3tilde}
			|\tilde{L}_3|
			\lesssim\big|\frac{1}{\tilde{\lambda}^2}\big|\|\phi_{R_0}\eee^{i\theta}\Lambda W_\lambda\|_{L^{\frac32}}\|g\|_{L^3}
			\lesssim\|\phi_{\frac{R_0}{\lambda}}\eee^{i\theta}\Lambda W\|_{L^{\frac32}}\|g\|_{\cE}
			\lesssim \big(\frac{R_0}{\lambda}\big)^\epsilon\|g\|_{\cE}\lesssim \eee^{-\frac98|t|},
		\end{equation}
		where $\epsilon\ll 1$. From the definition of $L_3$ and using \eqref{eq:mod-tildel}, we have
		\begin{align}
			|L_3'-\tilde{L}_3'|
            &\lesssim\Big|\frac{\tilde{\lambda}'}{\tilde{\lambda}^3}\la i\phi_{R_0(t)}\eee^{i\theta}\Lambda W_\lambda, g\ra \Big| 
			\lesssim\big|\frac{\tilde{\lambda}'}{\tilde{\lambda}^3}\big|\|\phi_{R_0}\eee^{i\theta}\Lambda W_\lambda\|_{L^{\frac32}}\|g\|_{L^3}\\
			&\lesssim\big|\frac{\tilde{\lambda}'}{\tilde{\lambda}}\big|\|\phi_{R_0\lambda}\eee^{i\zeta}\Lambda W\|_{L^{\frac32}}\|g\|_{\cE}
			\lesssim \big|\frac{\tilde{\lambda}'}{\tilde{\lambda}}\big|\big(\frac{R_0}{\lambda}\big)^\epsilon\|g\|_{\cE}\\
			&\lesssim \eee^{-\frac98|t|}.
		\end{align}
		Hence, 
		\begin{equation}\label{L3error}
			|L_3-\tilde{L}_3|\lesssim\eee^{-\frac98|t|}.
		\end{equation}
        From the definition of $\tilde{\theta}$, \eqref{L3tilde}, and \eqref{L3error}, we have
        \begin{equation}
            |\tilde{\theta} - \theta|\lesssim|L_3|\lesssim|L_3-\tilde{L}_3|+|\tilde{L}_3| \lesssim\eee^{-\frac98|t|}.
        \end{equation}
        Thus, \eqref{eq:diff-theta} holds.

	\end{proof}

	    \subsection{Control of the stable and unstable component}
	    An important step is to control
	    the stable and unstable components $a_1^\pm(t) = \la \alpha_{\zeta(t), \mu(t)}^\pm, g(t)\ra$ and $a_2^\pm(t)= \la \alpha_{\theta(t), \lambda(t)}^\pm, g(t)\ra$. Recall that $\nu > 0$ is the positive eigenvalue of the linearized flow, see \eqref{eq:Y1Y2}.
	    \begin{lemma}
		\label{lem:proper}
        Under the assumptions of Lemma~\ref{lem:mod}, we define
        \begin{equation}\label{tildea1}
            \tilde{a}_1^\pm(t) = a_1^\pm(t) + \la \alpha_{\zeta(t), \mu(t)}^\pm, \eee^{i\theta(t)}W_{\lambda(t)}\ra.
        \end{equation}
        
        For $t \in [T, T_1]$ there holds
		\begin{align}
            |\tilde{a}_1^\pm(t) - a_1^\pm(t)|&\lesssim \eee^{-2|t|},\label{eq:diff-a1}\\
			\big| \dd t \tilde{a}_1^+(t) - \frac{\nu}{\mu(t)^2}\tilde{a}_1^+(t)\big| &\leq \frac{c}{\mu(t)^2}\eee^{-\frac53|t|}, \label{eq:proper-1p} \\
			\big| \dd t \tilde{a}_1^-(t) + \frac{\nu}{\mu(t)^2}\tilde{a}_1^-(t)\big| &\leq \frac{c}{\mu(t)^2}\eee^{-\frac53|t|}, \label{eq:proper-1m} \\
			\big| \dd t a_2^+(t) - \frac{\nu}{\lambda(t)^2}a_2^+(t)\big| &\leq \frac{c}{\lambda(t)^2}\eee^{-\frac53|t|}, \label{eq:proper-2p} \\
			\big| \dd t a_2^-(t) + \frac{\nu}{\lambda(t)^2}a_2^-(t)\big| &\leq \frac{c}{\lambda(t)^2}\eee^{-\frac53|t|}, \label{eq:proper-2m}
		\end{align}
		with $c \to 0$ as $|T_0| \to +\infty$.
	\end{lemma}
	\begin{proof}
        By the definition of $\tilde{a}_1^\pm$ given in \eqref{tildea1}, we have
        \begin{equation}
            |\tilde{a}_1^\pm - a_1^\pm|\lesssim \int_{|x|\leq1}W_\lambda \ud x + \sup_{|x|\geq 1}W_\lambda \lesssim \lambda^2 \lesssim \eee^{-2|t|},
        \end{equation}
        i.e., \eqref{eq:diff-a1} holds.
        
		We will give a proof of \eqref{eq:proper-1p} and \eqref{eq:proper-2p}, the other two inequalities being analogous.
		
		Applying the chain rule to the formula $\tilde{a}_1^+ =a_1^+ +  \la \alpha_{\zeta, \mu}^+, \eee^{i\theta}W_{\lambda}\ra = \la \alpha^+_{\zeta, \mu}, g + \eee^{i\theta}W_{\lambda}\ra $ and using the definition of $\alpha_{\zeta, \mu}^+$ 
		we obtain
		\begin{align}
			\dd t \tilde{a}_1^+ &= -\frac{\mu'}{\mu}\big\la\frac{\eee^{i\zeta}}{\mu^2}\big(\Lambda_{-1}\cY_{\mu}^{(2)} + i\Lambda_{-1}\cY_{\mu}^{(1)}\big), g + \eee^{i\theta}W_{\lambda}\big\ra
			+ \zeta'\big\la \frac{\eee^{i\zeta}}{\mu^2}\big(i\cY_{\mu}^{(2)} - \cY_{\mu}^{(1)}\big), g + \eee^{i\theta}W_{\lambda}\big\ra \\
            &+ \la \alpha_{\zeta, \mu}^+, \partial_t g\ra + \theta'\big\la \alpha_{\zeta, \mu}^+, i\eee^{i\theta}W_{\lambda}\big\ra - \frac{\lambda'}{\lambda}\big\la \alpha_{\zeta, \mu}^+, \eee^{i\theta}\Lambda W_{\lambda}\big\ra.
		\end{align}
		For the first term, it follows from \eqref{eq:bootstrap-lambda}, \eqref{eq:bootstrap-g}, and \eqref{eq:mod-mu} that
        \begin{align}
            &\big|-\frac{\mu'}{\mu}\big\la\frac{\eee^{i\zeta}}{\mu^2}\big(\Lambda_{-1}\cY_{\mu}^{(2)} + i\Lambda_{-1}\cY_{\mu}^{(1)}\big), g + \eee^{i\theta}W_{\lambda}\big\ra\big|\\
            \lesssim & |\mu'|\Big(\|\frac{\eee^{i\zeta}}{\mu^2}\big(\Lambda_{-1}\cY_{\mu}^{(2)} + i\Lambda_{-1}\cY_{\mu}^{(1)}\big)\|_{L^{\frac32}}\|g\|_{L^3}+\|\frac{\eee^{i\zeta}}{\mu^2}\big(\Lambda_{-1}\cY_{\mu}^{(2)} + i\Lambda_{-1}\cY_{\mu}^{(1)}\big)\|_{L^2}\|\eee^{i\theta}W_{\lambda}\|_{L^2}\Big)\\
            \lesssim & |\mu'|(\lambda+\|g\|_\cE)\ll \eee^{-2|t|}\ll \eee^{-\frac53|t|}.
        \end{align}
        The second term can be treated similarly.
		We are left with the third term, and we expand $\partial_t g$ according to \eqref{eq:dtg}.
		
		Let us consider, one by one, the contributions of the four terms in the second line of \eqref{eq:dtg}.
		\begin{enumerate} [1.]
			\item The term $\big\la \alpha_{\zeta, \mu}^+, -\zeta'i\eee^{i\zeta}W_{\mu}\big\ra$ is equal to $0$ thanks to \eqref{eq:proper-iW}.
			\item The term $\big\la \alpha_{\zeta,\mu}^+, \frac{\mu'}{\mu}\eee^{i\zeta}\Lambda W_{\mu}\big\ra$ is equal to $0$ thanks to \eqref{eq:proper-LW}.
			\item The term $\big\la \alpha_{\zeta, \mu}^+, -\theta'i\eee^{i\theta}W_{\lambda}\big\ra$ cancels with the fourth term.
			\item The term $\big\la \alpha_{\zeta, \mu}^+, \frac{\lambda'}{\lambda}\eee^{i\theta}\Lambda W_{\lambda}\big\ra$ cancels with the fifth term.
		\end{enumerate}
		
		Let us finally consider the contribution of the first line of \eqref{eq:dtg}.
		We have
		\begin{equation}
			\begin{gathered}
				i\Delta g + i\big(f(\eee^{i\zeta}W_{\mu} + \eee^{i\theta}W_{\lambda} + g) - f(\eee^{i\zeta}W_{\mu}) - f(\eee^{i\theta}W_{\lambda})\big) = \\
				Z_{\zeta,\mu}g + i\big(f(\eee^{i\zeta}W_{\mu} + \eee^{i\theta}W_{\lambda} + g) - f(\eee^{i\zeta}W_{\mu}) - f(\eee^{i\theta}W_{\lambda}) - f'(\eee^{i\zeta}W_{\mu})g\big).
			\end{gathered}
		\end{equation}
		Since $\alpha_{\zeta, \mu}^+$ is an eigenfunction of $Z_{\zeta, \mu}^*$,
		with eigenvalue $\frac{\nu}{\mu^2}$, we obtain $\la\alpha_{\zeta, \mu}^+, Z_{\zeta, \mu}g\ra = \frac{\nu}{\mu^2}a_1^+$, hence we need to show that
		\begin{equation}
			\big|\big\la\alpha_{\zeta,\mu}^+, i\big(f(\eee^{i\zeta}W_{\mu} + \eee^{i\theta}W_{\lambda} + g) - f(\eee^{i\zeta}W_{\mu}) - f(\eee^{i\theta}W_{\lambda}) - f'(\eee^{i\zeta}W_{\mu})g\big)\big\ra\big| \ll \eee^{-\frac53|t|}.
		\end{equation}
		The proof of \eqref{B11} yields the bound $\eee^{-2|t|}\ll \eee^{-\frac53|t|}$. Together with \eqref{eq:diff-a1}, we obtain \eqref{eq:proper-1p}.
		
		We turn to the proof of \eqref{eq:proper-2p}.
		Applying the chain rule to the formula $a_2^+(t) = \la \alpha_{\zeta(t), \mu(t)}^+, g(t)\ra$ and using the definition of $\alpha_{\theta, \lambda}^+$ 
		we obtain
		\begin{equation}
			\dd t a_2^+ = -\frac{\lambda'}{\lambda}\big\la\frac{\eee^{i\theta}}{\lambda^2}\big(\Lambda_{-1}\cY_{\lambda}^{(2)} + i\Lambda_{-1}\cY_{\lambda}^{(1)}\big), g\big\ra
			+\theta'\big\la \frac{\eee^{i\theta}}{\lambda^2}\big(i\cY_{\lambda}^{(2)} - \cY_{\lambda}^{(1)}\big), g\big\ra + \la \alpha_{\theta, \lambda}^+, \partial_t g\ra.
		\end{equation}
		The first two terms are treated as in the case of $\tilde{a}_1^+$. In the third term, we expand $\partial_t g$ using \eqref{eq:dtg}.
		Let us consider, one by one, the contributions of the four terms in the second line of \eqref{eq:dtg}.
		\begin{enumerate}[1.]
			\item In order to bound the term $\big\la \alpha_{\theta, \lambda}^+, -\zeta'i\eee^{i\zeta}W_{\mu}\big\ra$, notice that
			\begin{equation}
				\|\alpha_{\theta,\lambda}^+\|_{L^1} \lesssim \int_{\bR^6}\frac{1}{\lambda^2}\big(|\cY_{\lambda}^{(1)}| + |\cY_{\lambda}^{(2)}|\big)\ud x \lesssim \lambda^2\lesssim  \eee^{-2|t|} \ll  \frac{1}{ \lambda^2}\eee^{-\frac53|t|}.
			\end{equation}
			This is sufficient since $\|{-}\zeta'i\eee^{i\zeta}W_{\mu}\|_{L^\infty} \lesssim 1$.
			\item The term $\big\la \alpha_{\theta, \lambda}^+, \frac{\mu'}{\mu}\eee^{i\zeta}\Lambda W_{\mu}\big\ra$ is analogous.
			\item The term $\big\la \alpha_{\theta, \lambda}^+, -\theta'i\eee^{i\theta}W_{\lambda}\big\ra$ is equal to $0$ thanks to \eqref{eq:proper-iW}.
			\item The term $\big\la \alpha_{\theta, \lambda}^+, \frac{\lambda'}{\lambda}\eee^{i\theta}\Lambda W_{\lambda}\big\ra$ is equal to $0$ thanks to \eqref{eq:proper-LW}.
		\end{enumerate}
		
		Let us finally consider the contribution of the first line of \eqref{eq:dtg}.
		We have
		\begin{equation}
			\begin{gathered}
				i\Delta g + i\big(f(\eee^{i\zeta}W_{\mu} + \eee^{i\theta}W_{\lambda} + g) - f(\eee^{i\zeta}W_{\mu}) - f(\eee^{i\theta}W_{\lambda})\big) = \\
				Z_{\theta, \lambda}g + i\big(f(\eee^{i\zeta}W_{\mu} + \eee^{i\theta}W_{\lambda} + g) - f(\eee^{i\zeta}W_{\mu}) - f(\eee^{i\theta}W_{\lambda}) - f'(\eee^{i\theta}W_{\lambda})g\big).
			\end{gathered}
		\end{equation}
		Since $\alpha_{\theta, \lambda}^+$ is an eigenfunction of $Z_{\theta, \lambda}^*$
		with eigenvalue $\frac{\nu}{\lambda^2}$, we obtain $\la\alpha_{\theta, \lambda}^+, Z_{\theta, \lambda}g\ra = \frac{\nu}{\lambda^2}a_2^+$, hence we need to show that
		\begin{equation}
			\label{eq:a2p-final}
			\lambda^2\big|\big\la\alpha_{\theta, \lambda}^+, i\big(f(\eee^{i\zeta}W_{\mu} + \eee^{i\theta}W_{\lambda} + g) - f(\eee^{i\zeta}W_{\mu}) - f(\eee^{i\theta}W_{\lambda}) - f'(\eee^{i\theta}W_{\lambda})g\big)\big\ra\big| \ll \eee^{-\frac53|t|}.
		\end{equation}
		Similarly to \eqref{I2}, we obtain
		\begin{equation}
			\label{eq:a2p-final1}
			\begin{gathered}
				\lambda^2\big|\big\la\alpha_{\theta, \lambda}^+, i(f'(\eee^{i\zeta}W_{\mu} + \eee^{i\theta}W_{\lambda})-f'(\eee^{i\theta}W_{\lambda}))g\big\ra\big| \ll \lambda\|g\|_\cE \\
				\lesssim \eee^{-|t|}\eee^{-\frac54|t|} \ll \eee^{-\frac53|t|}.
			\end{gathered}
		\end{equation}
		Next, we prove
		\begin{equation}
			\label{eq:a2p-final2}
			\lambda^2\big|\big\la\alpha_{\theta, \lambda}^+, i\big(f(\eee^{i\zeta}W_{\mu}+ \eee^{i\theta}W_{\lambda} + g) - f(\eee^{i\zeta}W_{\mu} + \eee^{i\theta}W_{\lambda})
			- f'(\eee^{i\zeta}W_{\mu} + \eee^{i\theta}W_{\lambda})g\big)\big\ra\big| \lesssim \|g\|_\cE^2 \ll \eee^{-\frac53|t|}.
		\end{equation}
		 Using \eqref{eq:pointwise-1} with $z_1 = \eee^{i\zeta}W_{\mu} + \eee^{i\theta}W_{\lambda}$ and $z_2 = g$ yields
		\begin{equation}
			\label{eq:B3-estim-11}
			|f(\eee^{i\zeta}W_{\mu} + \eee^{i\theta}W_{\lambda} + g) - f(\eee^{i\zeta}W_{\mu} + \eee^{i\theta}W_{\lambda}) - f'(\eee^{i\zeta}W_{\mu} + \eee^{i\theta}W_{\lambda})g| \lesssim |g|^2.
		\end{equation}
		Using the H\"older inequality we arrive at \eqref{eq:a2p-final2}.
		
		Using \eqref{eq:pointwise-2} we get
		\begin{equation}
			\|f(\eee^{i\zeta}W_{\mu} + \eee^{i\theta}W_{\lambda}) - f(\eee^{i\zeta}W_{\mu}) - f(\eee^{i\theta}W_{\lambda})\|_{L^\infty} \lesssim \|W_{\lambda} W_{\mu}\|_{L^\infty} \lesssim \frac{1}{\lambda^2}.
		\end{equation}
		By a change of variable, $\|\lambda^2 \alpha_{\theta, \lambda}^+\|_{L^1} \lesssim \lambda^4$, hence
		\begin{equation}
			\label{eq:a2p-final3}
			\lambda^2\big|\big\la\alpha_{\theta, \lambda}^+, i\big(f(\eee^{i\zeta}W_{\mu}+ \eee^{i\theta}W_{\lambda}) - f(\eee^{i\zeta}W_{\mu}) - f(\eee^{i\theta}W_{\lambda})\big)\big\ra\big| \lesssim \lambda^2 \lesssim  \eee^{-2|t|} \ll  \eee^{-\frac53|t|}.
		\end{equation}
		Taking the sum of \eqref{eq:a2p-final1}, \eqref{eq:a2p-final2} and \eqref{eq:a2p-final3}, and using the triangle inequality, we obtain \eqref{eq:a2p-final}.
	\end{proof}
	
	\section{Bootstrap Argument}
	\label{sec:boot}
	We turn to the heart of the proof, which consists in establishing the bootstrap estimates.
	We consider a solution $u(t)$, decomposed according to \eqref{eq:decompose}, \eqref{eq:param-rough} and \eqref{eq:orth}.
	The initial data at time $T \leq T_0$ is chosen as follows.
	\begin{lemma}
		\label{lem:initial}
		There exists $T_0 < 0$ such that for all $T \leq T_0$ the following holds.
        Let $\lambda^0 := \eee^{-|T|}$.
        For all $a_2^0$
		satisfying
		\begin{equation}
			\label{eq:initial-assum}
			|a_2^0| \leq \frac12\eee^{-\frac53|T|},
		\end{equation}
		there exists $g^0 \in X^1$ satisfying
		\begin{gather}
			\label{eq:initial-orth}
			\la \phi_{R} \Lambda W, g^0\ra = \la i\phi_{R}W, g^0\ra = \la i\phi_{R\lambda^0}\Lambda W_{\lambda^0}, g^0\ra = \la {-}\phi_{R\lambda^0}W_{\lambda^0}, g^0\ra = 0, \\
			\label{eq:initial-unstable}
			\la \alpha_{-\frac{\pi}{2},1}^-, g^0\ra = 0,\quad \la \alpha_{-\frac{\pi}{2},1}^+, g^0\ra = 0,\quad 
			\la \alpha_{0,\lambda^0}^-, g^0\ra = 0,\quad \la \alpha_{0, \lambda^0}^+, g^0\ra = a_2^0, \\
			\label{eq:initial-size}
			\|g^0\|_{\cE} \lesssim \eee^{-\frac53|T|}.
		\end{gather}
		This $g^0$ is continuous for the $X^1$ topology with respect to 
        $a_2^0$.
	\end{lemma}
	\begin{remark}
		For the continuity, we just claim that the function $g^0$ constructed in the proof
		is continuous with respect to
        $a_2^0$.
		Clearly, $g^0$ is not uniquely determined by \eqref{eq:initial-orth}, \eqref{eq:initial-unstable} and \eqref{eq:initial-size}.
	\end{remark}
	\begin{remark}
		Condition \eqref{eq:initial-orth} is exactly \eqref{eq:orth} with $\big( \zeta, \mu,\theta, \lambda\big) = \big(-\frac{\pi}{2}, 1, 0, \lambda^0\big)$.
		Hence, if we consider the solution $u(t)$ of \eqref{eq:nls} with initial data $u(T) = -iW + W_{\lambda^0} + g^0$
		and decompose it according to \eqref{eq:decompose}, then $g(T) = g^0$ and the initial values of the modulation parameters are
		$\big(\zeta(T), \mu(T),\theta(T), \lambda(T)\big) = \big({-}\frac{\pi}{2}, 1, 0, \lambda^0\big)$.
	\end{remark}
	\begin{proof}[Proof of Lemma~\ref{lem:initial}]
		We consider functions of the form
		\begin{equation}
			g^0 = a_1^+i\alpha_{-\frac{\pi}{2}, 1}^- - a_1^- i\alpha_{-\frac{\pi}{2}, 1}^+ + b_1 W + c_1 (-i\Lambda W) + a_2^+( \lambda^0)^2i\alpha_{0,  \lambda^0}^- - a_2^-( \lambda^0)^2i\alpha_{0,  \lambda^0}^+ + b_2 iW_{ \lambda^0} + c_2 \Lambda W_{ \lambda^0},
		\end{equation}
		with $a_1^+$, $a_1^-$, $b_1$, $c_1$, $a_2^+$, $a_2^-$, $b_2$, $c_2$ being real numbers.
		Let $\Phi: \bR^8 \to \bR^8$ be the linear map defined as follows:
		\begin{equation}
			\begin{gathered}
				\Phi(a_1^+, a_1^-, b_1, c_1, a_2^+, a_2^-, b_2, c_2) := \\
				\big(\la \alpha_{-\frac{\pi}{2}, 1}^+, g^0\ra, \la \alpha_{-\frac{\pi}{2}, 1}^-, g^0\ra, \la \phi_{R}\Lambda W, g^0\ra, \la i\phi_{R}W, g^0\ra, \\
				\la \alpha_{0,  \lambda^0}^+, g^0\ra, \la \alpha_{0,  \lambda^0}^-, g^0\ra, \big\la ( \lambda^0)^{-2}i\phi_{R\lambda^0}\Lambda W_{ \lambda^0}, g^0\big\ra, \big\la {-}( \lambda^0)^{-2}\phi_{R\lambda^0}W_{ \lambda^0}, g^0\big\ra\big).
			\end{gathered}
		\end{equation}
		Using \eqref{eq:proper-iW}, \eqref{eq:proper-LW}, \eqref{eq:Y1Y2-prod}, and the fact that $\lambda^0$ is small, we obtain that the matrix of $\Phi$ is strictly diagonally dominant, Which, combined with the implicit function theorem, implies the result.
	\end{proof}
	In the remaining part of this section, we will analyze solutions $u(t)$ of \eqref{eq:nls}
	with the initial data $u(T) = -iW + W_{ \lambda^0} + g^0$,
	where $g^0$ is given by the previous lemma. 
    \begin{proposition}
		\label{prop:bootstrap}
		There exists $T_0 <0$ with the following property.
		Let $T < T_1 < T_0$, $\lambda^0 := \eee^{-|T|}$ and 
        $a_2^0$ satisfy \eqref{eq:initial-assum}.
		Let $g^0 \in X^1$ be given by Lemma~\ref{lem:initial} and consider the solution $u(t)$ of \eqref{eq:nls}
		with the initial data $u(T) = -iW + W_{\lambda^0} + g^0$.
		Suppose that $u(t)$ exists on the time interval $[T, T_1]$, that for $t \in [T, T_1]$
		conditions \eqref{eq:bootstrap-zeta}-\eqref{eq:bootstrap-g} hold, and moreover that
		\begin{equation}
			\label{eq:bootstrap-unstable}
			 |a_2^+(t)| \leq \eee^{-\frac53|t|}.
		\end{equation}
		Then for $t \in [T, T_1]$ there holds
		\begin{align}
			\big|\zeta(t) + \frac{\pi}{2}\big| &\leq \frac 12 \eee^{-\frac98|t|}, \label{eq:bootstrap-better-zeta} \\
			|\mu(t) - 1| &\leq \frac 12 \eee^{-\frac98|t|},  \label{eq:bootstrap-better-mu} \\
			\big|\lambda(t) - \eee^{-|t|}\big| &\leq \frac 12\eee^{-\frac54|t|}, \label{eq:bootstrap-better-lambda} \\
			|a_1^+(t)| &\leq \frac 12  \eee^{-\frac53|t|}, \label{eq:bootstrap-better-a1p}\\
			|\theta(t)| &\leq \frac 12 \eee^{-\frac12|t|}, \label{eq:bootstrap-better-theta} \\
			\|g(t)\|_\cE &\leq \frac 12 \eee^{-\frac54|t|}. \label{eq:bootstrap-better-g}
		\end{align}
	\end{proposition}
	Before we give a proof, we need a little preparation. Following Lemma 4.5 and Lemma 4.7 in \cite{Jacek:nls}, we have the following two Lemmas.
	\subsection{A virial-type correction}
	The delicate part of the proof of Proposition~\ref{prop:bootstrap} will be to control $\theta(t)$. For this, we will need to use a virial functional, which we now define.
	
	\begin{lemma}
		\label{lem:fun-a}
		For any $c > 0$ and $\bar{R} > 0$ there exists a radial function $q(x) = q_{c,\bar{R}}(x) \in C^{3,1}(\bR^6)$ with the following properties:
		\begin{enumerate}[label=(P\arabic*)]
			\item $q(x) = \frac 12 |x|^2$ for $|x| \leq \bar{R}$, \label{enum:approx}
			\item there exists $\wt R > 0$ (depending on $c$ and $\bar{R}$) such that $q(x) \equiv \tx{const}$ for $|x| \geq \wt R$, \label{enum:support}
			\item $|\grad q(x)| \lesssim |x|$ and $|\Delta q(x)| \lesssim 1$ for all $x \in \bR^6$, with constants independent of $c$ and $\bar{R}$, \label{enum:gradlap}
			\item $\sum_{1\leq j, k\leq 6} \big(\partial_{x_j x_k} q(x)\big) \conj{v_j} v_k \geq -c\sum_{j=1}^6 |v_j|^2$, for all $x \in \bR^6, v_j \in \bC$, \label{enum:convex}
			\item $\Delta^2 q(x) \leq c\cdot|x|^{-2}$, for all $x \in \bR^6$. \label{enum:bilapl}
		\end{enumerate}
	\end{lemma}
	\begin{remark}
		We require $C^{3, 1}$ regularity in order not to worry about boundary terms in Pohozaev identities, see the proof of \eqref{eq:A-pohozaev}.
	\end{remark}

	In the sequel $q(x)$ always denotes a function of class $C^{3, 1}(\bR^6)$ verifying \ref{enum:approx}--\ref{enum:bilapl}
	with sufficiently small $c$ and sufficiently large $R$.
	
	For $\lambda > 0$ we define the operators $A(\lambda)$ and $A_0(\lambda)$ as follows.
	\begin{align}
		\label{eq:op-A}
		[A(\lambda)h](x) &:= \frac{1}{3\lambda^2}\Delta q\big(\frac{x}{\lambda}\big)h(x) + \frac{1}{\lambda}\grad q\big(\frac{x}{\lambda}\big)\cdot \grad h(x), \\
		[A_0(\lambda)h](x) &:= \frac{1}{2\lambda^2}\Delta q\big(\frac{x}{\lambda}\big)h(x) + \frac{1}{\lambda}\grad q\big(\frac{x}{\lambda}\big)\cdot \grad h(x). \label{eq:op-A0}
	\end{align}
	Combining these definitions with the fact that $q(x)$ is an approximation of $\frac 12 |x|^2$, we see that $A(\lambda)$ and $A_0(\lambda)$ are approximations (in a sense not yet made precise)
	of $\lambda^{-2}\Lambda$ and $\lambda^{-2}\Lambda_0$ respectively.
	We will write $A$ and $A_0$ instead of $A(1)$ and $A_0(1)$ respectively. Note the following scale-change formulas, which follow directly from the definitions:
	\begin{equation}
		\label{eq:A-rescale}
		\forall h\in \cE:\qquad A(\lambda)(h_\lambda) = \lambda^{-2}(Ah)_\lambda,\quad A_0(\lambda)(h_\lambda) = \lambda^{-2}(A_0 h)_\lambda.
	\end{equation}
	\begin{lemma}
		\label{lem:op-A}
		The operators $A(\lambda)$ and $A_0(\lambda)$ have the following properties:
		\begin{itemize}
			\item for $\lambda > 0$, the families $\{A(\lambda)\}$, $\{A_0(\lambda)\}$, $\{\lambda\partial_\lambda A(\lambda)\}$, $\{\lambda\partial_\lambda A_0(\lambda)\}$
			are bounded in $\scrL(\cE; \dot H^{-1})$ and the families $\{\lambda A(\lambda)\}$, $\{\lambda A_0(\lambda)\}$ are bounded in $\scrL(\cE; L^2)$,
			with the bound depending on the choice of the function $q(x)$,
			\item for all complex-valued $h_1, h_2 \in X^1(\bR^6)$ and $\lambda > 0$ there holds
			\begin{gather}
				\label{eq:A-by-parts}
				\la A(\lambda)h_1, f(h_1 + h_2) - f(h_1) - f'(h_1)h_2\ra = -\la A(\lambda)h_2, f(h_1+h_2) - f(h_1)\ra, \\
				\la h_1, A_0(\lambda)h_2\ra = -\la A_0(\lambda)h_1, h_2\ra, \qquad \text{hence $iA_0(\lambda)$ is a symmetric operator,} \label{eq:A0-by-parts}
			\end{gather}
			\item for any $c_0 > 0$, if we choose $c$ in Lemma~\ref{lem:fun-a} small enough, then for all $h \in X^1$ there holds
			\begin{equation}
				\label{eq:A-pohozaev}
				\la A_0(\lambda)h, \Delta h\ra \leq \frac{c_0}{\lambda^2} \|h\|_{\cE}^2 - \frac{1}{\lambda^2}\int_{|x| \leq \bar{R}\lambda}|\grad h(x)|^2 \ud x.
			\end{equation}
		\end{itemize}
	\end{lemma}
	
	\subsection{Closing the bootstrap}
	\begin{proof}[Proof of Proposition~\ref{prop:bootstrap}]
		We split the proof into three steps. First, we prove \eqref{eq:bootstrap-better-zeta} and \eqref{eq:bootstrap-better-mu}. Next, we use the fundamental solution of the ODE to prove \eqref{eq:bootstrap-better-lambda} and \eqref{eq:bootstrap-better-a1p}.
		Then, we use the virial functional and variational estimates to prove \eqref{eq:bootstrap-better-theta},
		with $\frac 12$ replaced by any strictly positive constant.
		To do this, we have to deal somehow with the term $\|W\|_{L^2}^{-2}K$ in
		the modulation equation \eqref{eq:mod-tildeth}. It involves terms quadratic in $g$, which is the critical size and will not allow us to recover the small constant.
		However, it turns out that we can use a virial functional to absorb the essential part of $K$.
		Proving \eqref{eq:bootstrap-better-theta} is the most difficult step. Finally, \eqref{eq:bootstrap-better-g} will follow from variational estimates.
		
		\textbf{Step 1.}
		Integrating \eqref{eq:mod-zeta} on $[T, t]$ and using the fact that $\zeta(T) = -\frac{\pi}{2}$, we get
		\begin{equation}
			\big|\zeta(t) + \frac{\pi}{2}\big| = \big|\zeta(t) - \zeta(T)\big| = \big|\int_T^t \zeta'(\tau)\ud \tau\big| \leq c\int_T^t\eee^{-\frac98|\tau|} \leq c\cdot \frac89\eee^{-\frac98|t|} \leq \frac 12 \eee^{-\frac98|t|},
		\end{equation}
        i.e. \eqref{eq:bootstrap-better-zeta} holds.
		The proof of \eqref{eq:bootstrap-better-mu} is similar.
		
		\textbf{Step 2.}
		By \eqref{eq:mod-tildel} and the fundamental solution of the ODE, we have 
	    \begin{align}
	    	\tilde{\lambda}(t)&=\eee^t\Big(\tilde{\lambda}(T)\eee^{-T} + c\int_{T}^{t}\eee^{-s}\eee^{-\frac54|s|}\ud s\Big)\\
	    	&=\eee^t\tilde{\lambda}^0\eee^{-T} + c\eee^t\int_{T}^{t}\eee^{\frac14s}\ud s\\
	    	&=\eee^t\big(\eee^{-|T|}+C\eee^{-\frac{11}{8}|T|}\big)\eee^{|T|} + 4c\eee^t\big(\eee^{-\frac14|t|}-\eee^{-\frac14|T|}\big)\\
	    	&\leq\eee^{-|t|} + 4c\eee^{-\frac54|t|} - 4c\eee^{-|t|}\eee^{-\frac14|T|} + C\eee^{-|t|}\eee^{-\frac{3}{8}|T|}\\
            &\leq\eee^{-|t|} + 4c\eee^{-\frac54|t|}+ C\eee^{-|t|}\eee^{-\frac{3}{8}|T|}.
	    \end{align}
		Combining this with \eqref{eq:diff-lambda}, we obtain 
        \begin{equation}
            |\lambda(t)- \eee^{-|t|}|\leq|\lambda(t)-\tilde{\lambda}(t)|+|\tilde{\lambda}(t)-\eee^{-|t|}|\leq C\eee^{-\frac{11}{8}|t|} + 4c\eee^{-\frac54|t|} + C\eee^{-|t|}\eee^{-\frac38|T|}\leq  2C\eee^{-\frac{11}{8}|t|} + 4c\eee^{-\frac54|t|}\leq \frac12\eee^{-\frac54|t|},
        \end{equation}
        if $t$ is large enough and $c\leq \frac{1}{16}$.
        Hence, \eqref{eq:bootstrap-better-lambda} holds. From \eqref{nu} and \eqref{eq:bootstrap-mu}, we obtain that $\frac{\nu}{\mu^2}<1$. Then using \eqref{eq:proper-1p}, we can prove \eqref{eq:bootstrap-better-a1p} similarly.
        \begin{align}
	    	\tilde{a}_1^+(t)&=\eee^{\frac{\nu}{\mu^2}t}\Big(\tilde{a}_1^+(T)\eee^{-\frac{\nu}{\mu^2}T} + c\int_{T}^{t}\eee^{-\frac{\nu}{\mu^2}s}\eee^{\frac{5}{3}s}\ud s\Big)\\
	    	&\leq\eee^{t-T}(a_1^0+C\eee^{-2|T|})\eee^{-\frac{\nu}{\mu^2}T} + c\eee^{t}\int_{T}^{t}\eee^{(\frac{5}{3}-1)s}\ud s\\
	    	&\leq C \eee^{2t}\eee^{T-t} + \frac32c(\eee^{\frac53t}-\eee^{t+\frac23T})\\
            &\leq C\eee^{-2|t|} + \frac32c\eee^{-\frac{5}{3}|t|}.
	    \end{align}
        Combining this with \eqref{eq:diff-a1}, we obtain
        \begin{equation}
            |a_1^+(t)|\leq |a_1^+(t)-\tilde{a}_1^+(t)|+|\tilde{a}_1^+(t)|\leq C\eee^{-2|t|} + C\eee^{-2|t|} + \frac32c\eee^{-\frac{5}{3}|t|} \leq  2C\eee^{-2|t|} + \frac32c\eee^{-\frac{5}{3}|t|}\leq \frac12\eee^{-\frac53|t|},
        \end{equation}
        if $t$ is large enough and $c\leq \frac13$. Thus, \eqref{eq:bootstrap-better-a1p} holds.

		\textbf{Step 3.}
		First, let us show that for $t \in [T, T_1]$ there holds
		\begin{equation}
			\label{eq:bootstrap-stable}
			|a_1^-(t)| < \eee^{-\frac53|t|} , \qquad |a_2^-(t)| < \eee^{-\frac53|t|} .
		\end{equation}
		This is verified initially, see \eqref{eq:initial-unstable}. From \eqref{eq:diff-a1}, we know $\tilde{a}_1^-(T) < \eee^{-\frac53|T|} $. Suppose that $T_2 \in (T, T_1)$ is the last time for which $|\tilde{a}_1^-(t)| < \eee^{-\frac53|t|}$ holds for $t \in [T, T_2)$.
		Let, for example, $\tilde{a}_1^-(T_2) = \eee^{-\frac53|T_2|}$.
		\eqref{eq:proper-1m} implies that $\dd t \tilde{a}_1^-(T_2) < 0$, which contradicts the assumption that $\tilde{a}_1^-(t) < \eee^{-\frac53|T_2|}$ for $t < T_2$. Using \eqref{eq:diff-a1} again, we can prove $|a_1^-(t)| < \eee^{-\frac53|t|}$ if $t$ large enough.
		The proof of the other inequality is similar.
		
		Let $R$ be a sufficiently large constant, independent of $t$, as defined below. We will prove that if $T_0$ is chosen large enough (depending on $R$), then
		\begin{equation}
			\label{eq:bootstrap-bbetter-theta}
			|\theta(t)| \lesssim R^{-2}\eee^{-\frac12|t|},\qquad \text{for }t \in [T, T_1].
		\end{equation}
		By the conservation of energy, \eqref{eq:coer-bound} and \eqref{eq:initial-size}, we have
		\begin{equation}
			\label{eq:bootstrap-energy}
			\big|E(u) - 2E(W)\big| = \big|E(u(T)) - 2E(W)\big| \lesssim \eee^{-3|T|}\leq \eee^{-3|t|},
		\end{equation}
		hence \eqref{eq:coer-conclusion} yields
		\begin{equation}
			\label{eq:bootstrap-bbetter-theta-leq}
			\theta\lambda^2 \lesssim \eee^{-3|t|} \quad\Rightarrow\quad \theta \lesssim \eee^{-3|t|+2|t|} = \eee^{-|t|} \ll \eee^{-\frac12|t|}.
		\end{equation}
		It remains to prove that 
		\begin{equation}
			\label{eq:bootstrap-bbetter-theta-geq}
			\theta \gtrsim -R^{-2}\eee^{-\frac12|t|}.
		\end{equation}
		To this end, we consider the following real scalar function
		\begin{equation}
			\label{eq:psi}
			\psi(t) := \tilde{\theta}(t) - \frac{1}{2\|W\|_{L^2}^2}\la g(t), i A_0(\lambda(t))g(t)\ra.
		\end{equation}
		We will show that for $t \in [T, T_1]$ there holds
		\begin{equation}
			\label{eq:deriv-psi}
			\psi'(t) \gtrsim -R^{-2}\eee^{-\frac12|t|}.
		\end{equation}

		From \eqref{eq:bootstrap-bbetter-theta-leq} we get $\theta\lesssim \eee^{-|t|}\ll\eee^{-\frac12|t|}$.
		Hence, choosing $|T_0|$ large enough, \eqref{eq:mod-tildeth} yields
		\begin{equation}
			\label{eq:deriv-psi-1}
			\begin{aligned}
				\psi' &\geq -2\theta+ \frac{K}{\lambda^2\|W\|_{L^2}^2} - \frac{1}{4}R^{-2}\eee^{-\frac12|t|}
				- \frac{1}{2\|W\|_{L^2}^2}\dd t\la g, i A_0(\lambda)g\ra \\
				&\geq \frac{1}{\|W\|_{L^2}^2}\Big(\frac{1}{\lambda^2}K - \frac 12 \dd t\la g, i A_0(\lambda)g\ra\Big) - \frac{1}{2}R^{-2}\eee^{-\frac12|t|},
			\end{aligned}
		\end{equation}
		so we need to compute $\frac 12 \dd t\la g, i A_0(\lambda)g\ra$, up to terms of order $\ll \eee^{-\frac12|t|}$.
		In this proof, the sign $\simeq$ will mean ``up to terms of order $\ll \eee^{-\frac12|t|}$ as $|T_0| \to +\infty$''.
		
		Since $iA_0(\lambda)$ is symmetric, we have
		\begin{equation}
			\label{eq:deriv-correction}
			\frac 12 \dd t \la g, i A_0(\lambda) g\ra = \frac 12 \lambda'\la g, i \partial_{\lambda} A_0(\lambda)g\ra + \la \partial_t g, i A_0(\lambda) g\ra.
		\end{equation}
		The first term is of size $\lesssim \big|\frac{\lambda'}{\lambda}\big|\cdot \|g\|_\cE^2 \ll \eee^{-\frac12|t|}$, hence negligible.
		We expand $\partial_t g$ according to \eqref{eq:dtg}. Consider the terms in the second line of \eqref{eq:dtg}. It follows from Lemma \ref{lem:basicmod}
		and the fact that $\|A_0(\lambda)g\|_{\dot H^{-1}} \lesssim \|g\|_\cE$ that their contribution is $\lesssim (|\zeta'|+|\frac{\mu'}{\mu}|+|\theta'|+|\frac{\lambda'}{\lambda}|)\|g\|_\cE \lesssim R^{-2} \eee^{\frac34|t|}\eee^{-\frac54|t|} = R^{-2}\eee^{-\frac12|t|}$.
		Hence we can write
		\begin{equation}
			\label{eq:deriv-correction-1}
			\frac 12 \dd t \la g, i A_0(\lambda) g\ra \simeq \la \Delta g + f(\eee^{i\zeta}W_{\mu} + \eee^{i\theta}W_{\lambda} + g) - f(\eee^{i\zeta}W_{\mu}) - f(\eee^{i\theta}W_{\lambda}), A_0(\lambda) g\ra + \tilde{C}R^{-2}\eee^{-\frac12|t|}.
		\end{equation}
		We now check that
		\begin{equation}
			\label{eq:deriv-correction-2}
			|\la f(\eee^{i\zeta}W_{\mu} + \eee^{i\theta }W_{\lambda}) - f(\eee^{i\zeta}W_{\mu}) - f(\eee^{i\theta}W_{\lambda}), A_0(\lambda)g\ra| \ll \eee^{-\frac12|t|}.
		\end{equation}
		The function $A_0(\lambda)g$ is supported in the ball of radius $\wt R\lambda$. In this region, we have $ W_{\mu}\lesssim 1$. Hence, \eqref{eq:pointwise-2} yields
		$|\la f(\eee^{i\zeta}W_{\mu} + \eee^{i\theta}W_{\lambda}) - f(\eee^{i\zeta}W_{\mu}) - f(\eee^{i\theta}W_{\lambda})| \lesssim W_{\lambda}$.
		By a change of variable, we obtain
		\begin{equation}
			\|W_{\lambda}\|_{L^2(|x| \leq \wt R\lambda)} = \lambda\|W\|_{L^2(|x| \leq \wt R)} \lesssim \eee^{-|t|}.
		\end{equation}
		By the first property in Lemma~\ref{lem:op-A}, there holds $\|A_0(\lambda)g\|_{L^2} \lesssim \lambda^{-1}\|g\|_\cE \lesssim \eee^{-\frac14|t|}$.
		Hence, the Cauchy-Schwarz inequality implies \eqref{eq:deriv-correction-2} (with a large margin).
		By the triangle inequality, \eqref{eq:deriv-correction-1} and \eqref{eq:deriv-correction-2} yield
		\begin{equation}
			\label{eq:deriv-correction-3}
			\frac 12 \dd t \la g, i A_0(\lambda) g\ra \simeq \la \Delta g + f(\eee^{i\zeta}W_{\mu} + \eee^{i\theta}W_{\lambda} + g) - f(\eee^{i\zeta}W_{\mu} + \eee^{i\theta}W_{\lambda}), A_0(\lambda) g\ra + \tilde{C}R^{-2}\eee^{-\frac12|t|}.
		\end{equation}
		We transform the right hand side using \eqref{eq:A-by-parts}, \eqref{eq:A-pohozaev} and the fact that $A_0(\lambda)g = \frac{1}{6\lambda^2} \Delta q\big(\frac{\cdot}{\lambda}\big)g + A(\lambda)g$.
		Note that for any $c_1 > 0$ we have $\frac{c_0}{\lambda^2}\|g\|_\cE^2 \leq c_1\eee^{-\frac12|t|}$ if we choose $c_0$ small enough, thus
		\begin{equation}
			\label{eq:deriv-correction-expand}
			\begin{aligned}
				&\frac 12 \dd t \la g, i A_0(\lambda) g\ra \leq \tilde{C}R^{-2}\eee^{-\frac12|t|} + c_1\eee^{-\frac12|t|} \\
				&-\frac{1}{\lambda^2}\Big(\int_{|x| \leq \bar{R}\lambda}|\grad g|^2 \ud x - \big\la f(\eee^{i\zeta}W_{\mu} + \eee^{i\theta}W_{\lambda}+ g) - f(\eee^{i\zeta}W_{\mu} + \eee^{i\theta}W_{\lambda}), \frac 16 \Delta q\big(\frac{\cdot}{\lambda}\big)g\big\ra\Big) \\
				&- \la A(\lambda)(\eee^{i\zeta}W_{\mu} + \eee^{i\theta}W_{\lambda}), f(\eee^{i\zeta}W_{\mu} + \eee^{i\theta}W_{\lambda} + g) - f(\eee^{i\zeta}W_{\mu} + \eee^{i\theta}W_{\lambda}) - f'(\eee^{i\zeta}W_{\mu} + \eee^{i\theta}W_{\lambda})g\ra,
			\end{aligned}
		\end{equation}
		where $c_1$ can be made arbitrarily small. Consider the second line. We will check that
		\begin{equation}
			\label{eq:deriv-correction-4}
			\Big|\big\la f(\eee^{i\zeta}W_{\mu} + \eee^{i\theta}W_{\lambda} + g) - f(\eee^{i\zeta}W_{\mu}+ \eee^{i\theta}W_{\lambda}), \frac 16\Delta q\big(\frac{\cdot}{\lambda}\big)g\big\ra - \la f'(\eee^{i\theta}W_{\lambda})g, g\ra\Big| \ll\eee^{-\frac52|t|}.
		\end{equation}
		Indeed, $\Delta q$ is bounded, thus $\big\|\frac 16 \Delta q\big(\frac{\cdot}{\lambda}\big)g\big\|_{L^3} \lesssim \|g\|_\cE$.
		By \eqref{eq:pointwise-1} we have
		\begin{equation}
			\|f(\eee^{i\zeta}W_{\mu} + \eee^{i\theta}W_{\lambda}+ g) - f(\eee^{i\zeta}W_{\mu} + \eee^{i\theta}W_{\lambda}) - f'(\eee^{i\zeta}W_{\mu} + \eee^{i\theta}W_{\lambda})g\|_{L^\frac32} \lesssim \|g\|_\cE^2 \ll \|g\|_\cE.
		\end{equation}
		Now from \eqref{eq:pointwise-5} we obtain
		\begin{equation}
			\big\|\big(f'(\eee^{i\zeta}W_{\mu} + \eee^{i\tilde{\theta}}W_{\lambda})- f'(\eee^{i\theta}W_{\lambda})\big)g\big\|_{L^\frac32(|x| \leq \wt R\lambda)}
			\lesssim \|f'(\eee^{i\zeta}W_{\mu})\|_{L^3(|x| \leq \wt R\lambda)}\|g\|_\cE \ll \|g\|_\cE.
		\end{equation}
		We have obtained
		\begin{equation}
			\Big|\big\la f(\eee^{i\tilde{\zeta}}W_{\mu}+ \eee^{i\theta}W_{\lambda} + g) - f(\eee^{i\zeta}W_{\mu} + \eee^{i\theta}W_{\lambda}), \frac 16 \Delta q\big(\frac{\cdot}{\lambda}\big)g\big\ra - \big\la f'(\eee^{i\theta}W_{\lambda})g, \frac 16 \Delta q\big(\frac{\cdot}{\lambda}\big)g\big\ra\Big| \ll\eee^{-\frac52|t|}.
		\end{equation}
		But $\frac 16 \Delta q\big(\frac{x}{\lambda}\big) = 1$ for $|x| \leq \bar{R}\lambda$ and $\|f'(\eee^{i\theta}W_{\lambda})\|_{L^3(|x| \geq \bar{R}\lambda)} \ll 1$ for $\bar{R}$ large. This proves \eqref{eq:deriv-correction-4}.
		
		The bounds \eqref{eq:bootstrap-unstable} and \eqref{eq:bootstrap-stable}, together with \eqref{eq:coer-L-2}, imply that
		\begin{equation}
			\int_{|x| \leq \bar{R}\lambda}|\grad g|^2 \ud x - \la f'(\eee^{i\theta}W_{\lambda})\big)g, g\ra \geq -c_2\|g\|_\cE^2,
		\end{equation}
		with $c_2$ as small as we like by enlarging $\bar{R}$. Thus, we have obtained that the second line in \eqref{eq:deriv-correction-expand} is $\leq c_2\eee^{-\frac12|t|}$,
		with $c_2$ which can be made arbitrarily small.
		
		We are left with the third line of \eqref{eq:deriv-correction-expand}. We will show that it equals $\frac{1}{\lambda^2}K$ up to negligible terms.
		The support of $A(\lambda)(\eee^{i\zeta}W_{\mu})$ is contained in $|x| \leq \wt R\lambda$ and $\|A(\lambda)(\eee^{i\zeta}W_{\mu})\|_{L^\infty} \lesssim \lambda^{-2}$.
		Hence,
		\begin{equation}
			\|A(\lambda)(\eee^{i\zeta}W_{\mu})\|_{L^3(|x| \leq \wt R\lambda)} \lesssim  \lambda^{-2} \lambda^{2} \sim 1.
		\end{equation}
		From \eqref{eq:pointwise-1} and H\"older we have
		\begin{equation}
			\|f(\eee^{i\zeta}W_{\mu} + \eee^{i\theta}W_{\lambda} + g) - f(\eee^{i\zeta}W_{\mu} + \eee^{i\theta}W_{\lambda}) - f'(\eee^{i\zeta}W_{\mu} + \eee^{i\theta}W_{\lambda})g\|_{L^\frac32} \lesssim \|g\|_\cE^2 \ll \eee^{-\frac12|t|}.
		\end{equation}
		Thus, in the third line of \eqref{eq:deriv-correction-expand} we can replace $A(\lambda)(\eee^{i\zeta}W_{\mu} + \eee^{i\theta}W_{\lambda})$ by $A(\lambda)(\eee^{i\theta}W_{\lambda})$.
		Property \ref{enum:gradlap} implies that $|AW - \Lambda W| \lesssim W$ pointwise, with a constant independent of $c$ and $R$ used in the definition of the function $q$.
		After rescaling and phase change we obtain $\big|A(\lambda)(\eee^{i\theta}W_{\lambda}) - \frac{1}{\lambda^2}\eee^{i\theta}\Lambda W_{\lambda}\big| \lesssim \frac{1}{\lambda^2} W_{\lambda}$.
		But $A(\lambda)W = \frac{1}{\lambda^2}\Lambda W_{\lambda}$ for $|x| \leq \bar{R}\lambda$, so we obtain
		\begin{equation}
			\begin{aligned}
				&\Big|\la A(\lambda)(\eee^{i\theta}W_{\lambda}) - \frac{1}{\lambda^2}\eee^{i\theta}\Lambda W_{\lambda}, f(\eee^{i\zeta}W_{\mu}+ \eee^{i\theta}W_{\lambda} + g) - f(\eee^{i\zeta}W_{\mu} + \eee^{i\theta}W_{\lambda}) - f'(\eee^{i\zeta}W_{\mu} + \eee^{i\theta}W_{\lambda})g\ra\Big| \\
				& \lesssim \frac{1}{\lambda^2}\int_{|x| \geq \bar{R}\lambda}W_{\lambda}\cdot |f(\eee^{i\zeta}W_{\mu} + \eee^{i\theta}W_{\lambda} + g) - f(\eee^{i\zeta}W_{\tilde{\mu }}+ \eee^{i\theta}W_{\lambda}) - f'(\eee^{i\zeta}W_{\mu} + \eee^{i\theta}W_{\lambda})g|\ud x\\
				& \lesssim \frac{1}{\lambda^2}\int_{|x| \geq \bar{R}\lambda}W_{\lambda}\cdot g^2\ud x\lesssim\frac{1}{\lambda^2}\|W_{\lambda}\|_{L^3(|x| \geq \bar{R}\lambda)}\|g\|_{L^3}^2\lesssim c_3\eee^{-\frac12|t|},
			\end{aligned}
		\end{equation}
		where $c_3$ arbitrarily small as $\bar{R} \to +\infty$.

		Resuming all the computations starting with \eqref{eq:deriv-correction}, we have shown that
		\begin{equation}\label{eq:est-theta}
			\frac 12 \dd t \la g, iA_0(\lambda) g\ra - \frac{1}{\lambda^2}K \lesssim R^{-2}\eee^{-\frac12|t|}.
		\end{equation}
		Hence \eqref{eq:deriv-psi-1} yields \eqref{eq:deriv-psi}.
		
		Since $\theta(T) = 0$, we have $|\tilde{\theta}(T)|\lesssim|\tilde{\theta}(T)-\theta(T)|\lesssim \eee^{-\frac98|T|}\ll \eee^{-\frac12|T|} $. 
		Integrating \eqref{eq:deriv-psi} on $[T, t]$ and using $\|g(T)\|_\cE^2 \ll \eee^{-\frac12|T|}$, we get $\psi(t) \gtrsim -R^{-2}\eee^{-\frac12|t|}$. But $|\la g(t), A_0(\lambda)g(t)\ra| \lesssim \|g(t)\|_{\cE}^2 \leq \eee^{-\frac52|t|} \ll \eee^{-\frac12|t|}$. Hence, we obtain $\tilde{\theta}(t) \gtrsim -R^{-2}\eee^{-\frac12|t|}$. Combining this with \eqref{eq:diff-theta}, we obtain \eqref{eq:bootstrap-bbetter-theta-geq}.
		This finishes the proof of \eqref{eq:bootstrap-better-theta} if $R$ is large enough.
		
		\textbf{Step 4.}
		From \eqref{eq:coer-conclusion} we obtain $\|g\|_\cE^2 + C_0(\ln R)^{\frac{4}{3}} \theta\lambda^2 \leq C(\ln R)^{\frac43}\eee^{-3|t|}$, hence
		\begin{equation}\label{eq:est-g}
			\|g\|_\cE^2 \leq -C_0 (\ln R)^{\frac{4}{3}}\theta\lambda^2 + C(\ln R)^{\frac{4}{3}}\eee^{-3|t|}\leq C_0 (\ln R)^{\frac{4}{3}} R^{-2}\eee^{-\frac12|t|}\eee^{-2|t|}+ C(\ln R)^{\frac{4}{3}}\eee^{-3|t|}\leq \frac14\eee^{-\frac52|t|},
		\end{equation}
		provided that $R$ and $T_0$ are large enough. This yields \eqref{eq:bootstrap-better-g}.
	\end{proof}
    \begin{remark}
        In order to prove \eqref{eq:bootstrap-better-theta} and \eqref{eq:bootstrap-better-g}, we need $R$ to be sufficiently large, as can be seen from \eqref{eq:bootstrap-bbetter-theta} and \eqref{eq:est-g}. Thus, these two estimates determine the choice of $R$. Once $R$ is fixed, we can choose $c_0$ such that $c_1\lesssim R^{-2}$ and choose $\bar{R}$ such that $c_2, c_3\lesssim R^{-2}$. Consequently, \eqref{eq:est-theta} follows.
    \end{remark}
	
	\subsection{Choice of the initial data by a topological argument}
	The bootstrap in Proposition~\ref{prop:bootstrap} leaves out the control of $a_2^+(t)$.
	We will tackle this problem here.
	
	\begin{proposition}
		\label{prop:shooting}
		Let $|T_0|$ be large enough. For all $T < T_0$ there exist $\lambda^0 := \eee^{-|T|}$ and 
        $a_2^0$ satisfy \eqref{eq:initial-assum}
		such that the solution $u(t)$ with the initial data $u(T) = -iW + W_{\lambda^0} + g^0$ exists on the time interval $[T, T_0]$
		and for $t\in[T, T_0]$ the bounds \eqref{eq:bootstrap-better-zeta}, \eqref{eq:bootstrap-better-mu}, \eqref{eq:bootstrap-better-lambda}, \eqref{eq:bootstrap-better-a1p}, \eqref{eq:bootstrap-better-theta}, \eqref{eq:bootstrap-better-g},
		\begin{equation}
			|a_2^+(t)| \leq \frac 12 \eee^{-\frac53|t|} \label{eq:bootstrap-better-a2p}
		\end{equation}
		hold.
	\end{proposition}
	The proof is similar to that in \cite[Proposition 4.8]{Jacek:nls}, except that \cite{Jacek:nls} considers three variables, whereas we deal with only one variable. We give the proof for the reader's convenience.
	For $t \in [T, T_0], \;\wt a_2 \in \bR$ we denote
	\begin{equation*}
		X_t(\wt a_2) := \wt a_2\eee^{-\frac53|t|}.
	\end{equation*}
	We see that $a_2^+(t)$ satisfies \eqref{eq:bootstrap-better-a2p} if and only if
	\begin{equation}
		X_t^{-1}(a_2^+(t)) \in Q := \Big[{-}\frac 12, \frac 12\Big].
	\end{equation}
	\begin{lemma}
		Assume that $a_2^+(t)$ satisfies \eqref{eq:proper-2p} on the time interval $t \in (T_1, T_2)$
		and that
		\begin{equation}
		p_2 := X_t^{-1}( a_2^+(t)) \in Q \setminus \partial Q\qquad \text{for all }t \in (T_1, T_2).
		\end{equation}
		Then for all $t \in (T_1, T_2)$ there holds
		\begin{equation}\label{eq:cube-p2} 
			\Big|p_2'(t) +\Big(\frac53- \frac{\nu}{ \lambda(t)^2}\Big)p_2(t)\Big| \leq \frac{c}{\lambda(t)^2}, 
		\end{equation}
		where $c > 0$ can be made arbitrarily small by taking $T_0$ large enough.
	\end{lemma}
	\begin{proof}
		By the definition of $p_2(t)$, we have
		$a_2^+(t) = \eee^{-\frac53|t|}p_2(t)$, which yields
		\begin{equation}
			\dd t a_2^+ - \frac{\nu}{\lambda(t)^2}a_2^+ = \eee^{-\frac53|t|}\Big(p_2'(t) -\frac{\nu}{\lambda(t)^2}p_2(t)\Big) + \frac53\eee^{-\frac53|t|}p_2(t),
		\end{equation}
		so \eqref{eq:proper-2p} implies \eqref{eq:cube-p2}.
	\end{proof}
	For $M > 1,\;p \in \bR$ we denote
	\begin{equation}
		V^M(p) := \{p + r_2: \sign(r_2) = \sign(p)\text{ and }\max|r_2| < M|r_2	|\}.
	\end{equation}
	\begin{lemma}
		\label{lem:top-2}
		Assume that $a_2^+(t)$ satisfies \eqref{eq:bootstrap-unstable}, and \eqref{eq:proper-2p}
		for $t \in (T_1, T_2)$.
		There exists a constant $M > 0$, depending on $T_1$ and $T_2$,
		such that if for some $T_3 \in (T_1, T_2)$ there holds
		$|p_2(T_3)| \geq \frac 14$, then for all $t \in (T_3, T_2)$ there holds $p_2(t) \in V^M(p_2(T_3))$.
	\end{lemma}
	\begin{proof}
		By \eqref{eq:bootstrap-lambda}, we have $\frac53- \frac{\nu}{ \lambda(t)^2}\leq 0.$ Thus, from the previous lemma, we infer that there exist strictly positive constants $m_1$ and $M_1$, depending on $T_1$ and $T_2$,
		such that $|p_2'(t)| \leq M_1$ and
		\begin{equation}
			|p_2(t)| \geq \frac 14\quad \Rightarrow\quad |p_2'(t)| \geq m_1\text{ and }\sign p_2'(t) = \sign p_2(t).
		\end{equation}
		It is sufficient to take $M > \frac{M_1}{m_1}$.
	\end{proof}
	\begin{proof}[Proof of Proposition~\ref{prop:shooting}]
		The proof proceeds by contradiction. Supposing that the result does not hold, we will construct a continuous retraction $\Phi : Q \to \partial Q$, $\Phi(p) = p$ for $p \in \partial Q$. It is a well-known fact from topology that such a function $\Phi$ does not exist.
		
		Let $p^0 \in Q$. Take $  a_2^0= X_{T}(p^0)$ and let $g^0$ be given by Lemma~\ref{lem:initial}.
		Let $u: [T, T_+) \to \cE$ be the solution of \eqref{eq:nls} for the initial data $u(T) = -iW + W_{\lambda^0} + g^0$.
		We will say that the solution $u$ is associated with $p^0 \in Q$.
		
		Let $T_2$ be the infimum of the values of $t \in [T, T_+)$ such that \eqref{eq:bootstrap-better-zeta}, \eqref{eq:bootstrap-better-mu}, \eqref{eq:bootstrap-better-theta}, \eqref{eq:bootstrap-better-g},
		\eqref{eq:bootstrap-better-lambda}, \eqref{eq:bootstrap-better-a1p} or \eqref{eq:bootstrap-better-a2p} does not hold.
		By our assumption that Proposition~\ref{prop:shooting} is false, we have that $T_2$ exists and $T_2 < T_0$.
		Indeed, if all the listed conditions were satisfied for $t \in [T, T_+)$, then Lemma~\ref{cor:leaves-compact} would imply that $T_+ > T_0$,
		hence all the conditions would hold on $[T, T_0]$, which contradicts the assumption.
		
		Set $p^1 := X_{T_2}^{-1}(a_2^+(T_2))$.
		By continuity $p^1 \in Q$, and we will show that in fact $p^1 \in \partial Q$.
		Indeed, by continuity of the flow, the assumptions of Proposition~\ref{prop:bootstrap} are satisfied for $T_1 = T_2 + \tau$ for some $\tau > 0$.
		Hence \eqref{eq:bootstrap-better-zeta}, \eqref{eq:bootstrap-better-mu}, \eqref{eq:bootstrap-better-lambda}, \eqref{eq:bootstrap-better-a1p}, \eqref{eq:bootstrap-better-theta} and \eqref{eq:bootstrap-better-g}
		continue to hold on $[T_2, T_2 + \tau]$, so the condition \eqref{eq:bootstrap-better-a2p}
		is violated somewhere on $[T_2, T_2 + \tau]$ for every $\tau > 0$. By continuity of the parameters with respect to time,
		this yields $p^1 \in \partial Q$.
		
		We set
		\begin{equation}
			\Phi: Q \to \partial Q,\qquad \Phi(p^0) := p^1.
		\end{equation}
		It is immediate from the definition that $\Phi(p) = p$ for $p \in \partial Q$, and it remains to show that $\Phi$ is continuous.
		
		Let $p^0 \in Q$, $\Phi(p^0) = p^1 \in \partial Q$ and $\varepsilon > 0$.
		Let $M$ be the constant from Lemma~\ref{lem:top-2} for $T_1 = T$ and $T_2 = T_0$.
		We will consider the case $p^1= \frac 12$, the other cases being similar.
		It is clear that for $\delta > 0$ small enough $V_\delta := V^M\big(\frac 12 - \delta\big) \cap \partial Q$
		is an $\varepsilon$-neighborhood of $p^1$.
		Thus, by Lemma~\ref{lem:top-2}, in order to finish the proof it suffices to show that
		if $q^0 \in Q$ with $|q^0 - p^0|$ small enough, then the solution associated with $q$ passes through $V_\delta$.
		
		If $p^0 = p^1 \in \partial Q$, this is obvious, since $V_\delta$ is in this case a neighborhood of $p^0$.
		In the case $p^0 \in Q \setminus \partial Q$, the solution associated with $p^0$ passes through $V_\delta$
		before reaching $\partial Q$. Thus, by the continuous dependence on the initial data,
		the solution associated with $q^0$ passes through $V_\delta$ if $|q^0 - p^0|$ is small enough.
	\end{proof}
	
	\begin{proof}[Proof of Theorem~\ref{thm:deux-bulles}]
		Let $T_0 < 0$ be given by Proposition~\ref{prop:shooting} and let $T_0, T_1, T_2, \ldots$
		be a decreasing sequence tending to $-\infty$.
		For $n \geq 1$, let $u_n$ be the solution given by Proposition~\ref{prop:shooting}.
		Inequalities \eqref{eq:bootstrap-better-zeta}, \eqref{eq:bootstrap-better-mu},
		\eqref{eq:bootstrap-better-lambda}, \eqref{eq:bootstrap-better-theta} and \eqref{eq:bootstrap-better-g} yield
		\begin{equation}
			\label{eq:uniform}
			\Big\|u_n(t) - \Big({-}iW + W_{\eee^{-|t|}}\Big)\Big\|_\cE \lesssim \eee^{-\frac12|t|},
		\end{equation}
		for all $t \in [T_n, T_0]$ and with a constant independent of $n$.
		Upon passing to a subsequence, we can assume that $u_n(T_0) \wto u_0 \in \cE$.
		Let $u$ be the solution of \eqref{eq:nls} with the initial condition $u(T_0) = u_0$.
		Lemma ~\ref{cor:weak-cont} implies that $u$ exists on the time interval $({-}\infty, T_0]$
		and for all $t \in ({-}\infty, T_0]$ there holds $u_n(t) \wto u(t)$.
		Passing to the weak limit in \eqref{eq:uniform} finishes the proof.
	\end{proof}

    \appendix
    \section{Heuristic computation of the concentration rate}
    \label{decay estimate}

    In this section, we present a heuristic argument
 allowing us to predict the asymptotic behavior of the scales of the bubbles.

    Let the Lagrangian be defined by
	\begin{equation}
		\mathcal{L}(u)=\iint_{\bR^{1+6}}\Big(\frac12\Im\big(\bar{u}\partial_t u\big)+\frac12|\nabla u|^2-\frac13|u|^3\Big)\ud x\ud t.
	\end{equation}
	Thus, we have
	\begin{multline}
		\mathcal{L}(u+h)-\mathcal{L}(u)= \\
        \iint\Big(\frac12\Im\big(\bar{h}\partial_t u\big)+\frac12\Im\big(\bar{u}\partial_t h\big)+\frac12\big(\nabla\bar{h}\cdot\nabla u+\nabla\bar{u}\cdot\nabla h\big)-|u|u\bar{h}\Big)\ud x\ud t+o(\|h\|).
	\end{multline}
	Note that 
	\begin{multline}
		\iint\frac12\Im\big(\bar{h}\partial_t u\big)+\frac12\Im\big(\bar{u}\partial_t h\big)=\iint\frac12\Im\big(\bar{h}\partial_t u-h\partial_t\bar{u}\big)
		= \\
        -\iint\frac12\Re\big(\bar{h}i\partial_t u+h\bar{i\partial_t u}\big)=-\int\la i\partial_t u,h\ra
	\end{multline}
    and
	\begin{equation*}
		\iint\frac12\big(\nabla\bar{h}\cdot\nabla u+\nabla\bar{u}\cdot\nabla h\big)=-\int\la \Delta u,h\ra.
	\end{equation*}
	Therefore, the Euler-Lagrange equation is
	\begin{equation}
			i\partial_t u + \Delta u+ f(u) = 0, \quad f(u) := |u|u,
	\end{equation}
	i.e., \eqref{eq:nls}.
	
	We plug in the two-bubble ansatz $u(t) = \eee^{i\zeta(t)}W_{\mu(t)} + \eee^{i\theta(t)}W_{\lambda(t)} $.
    We first focus on $\iint\frac12\Im\big(\bar{u}\partial_t u\big)\ud x\ud t$.
    We have
	$$\partial_t u = \zeta'i\eee^{i\zeta}W_\mu - \frac{\mu'}{\mu}\eee^{i\zeta}\Lambda W_\mu + \theta'i\eee^{i\theta}W_\lambda - \frac{\lambda'}{\lambda}\Lambda W_\lambda,$$
	thus
	$$\iint\frac12\Im\big(\bar{u}\partial_t u\big)\ud x\ud t\simeq C_1\int\Big(\frac12\theta'\lambda^2+\frac12\zeta'\mu^2\Big)\ud t,$$
    where $C_1$ is given by \eqref{eq:explicit-1}.
    As for the energy term $\iint\big(\frac12|\grad u|^2 - \frac 13 |u|^3\big)\ud x\ud t$, we only keep the leading term computed in Lemma~\ref{lem:coer-sans-g}.
    Discarding the terms related to the less concentrated bubble, we obtain the \emph{reduced Lagrangian}
	$$\wt \scrL(\theta, \lambda) = C_1 \int\Big(\frac12\theta'\lambda^2+\theta\lambda^2\Big)\ud t.$$
	The corresponding Euler-Lagrange equations are
	\begin{equation}
		\begin{cases}
			\lambda\lambda' = \lambda^2, \\
			\theta'\lambda+2\theta\lambda= 0.
		\end{cases}
	\end{equation}
   We thus have
    \begin{equation}
    	\lambda'\simeq\lambda,\;\theta'\simeq-2\theta,
    \end{equation}
    so it is natural to assume $\lambda(t)\simeq \eee^{-|t|}.$

\end{document}